\documentclass[preprint,12pt]{elsarticle}
\usepackage{amsmath}
\usepackage{amsfonts}
\usepackage{amssymb}
\usepackage{verbatim}
\usepackage[
bookmarks=true,         bookmarksnumbered=true,                        colorlinks=true,
pdfstartview=FitV,
linkcolor=blue, citecolor=blue, urlcolor=blue]{hyperref}
\usepackage{amssymb}

\allowdisplaybreaks[4]
\newtheorem{theorem}{Theorem}

\newtheorem{corollary}[theorem]{Corollary}

\newtheorem{example}[theorem]{Example}

\newtheorem{lemma}[theorem]{Lemma}

\newtheorem{proposition}[theorem]{Proposition}
\newtheorem{remark}[theorem]{Remark}

\numberwithin{theorem}{section}
\numberwithin{equation}{section}

\newenvironment{proof}[1][Proof]{\noindent\textit{#1.} }{\ \rule{0.5em}{0.5em}}

\begin{document}

\begin{frontmatter}



\title{Convergence of densities of some functionals of Gaussian processes}


\author{Yaozhong Hu\footnote{
Y.  Hu is
partially supported by a grant from the Simons Foundation
\#209206.}, \ Fei Lu\corref{cor1} \ and \ David Nualart
\footnote{
D. Nualart is supported by the NSF grant  DMS1208625}}
\ead{hu@math.ku.edu; flu@lbl.gov; nualart@math.ku.edu}

\address{Department of Mathematics,
University of Kansas,
Lawrence, Kansas, 66045 USA}

\begin{abstract}
The aim of this paper is to establish the uniform convergence of the
densities of a sequence of random variables, which are functionals of an
underlying Gaussian process, to a normal density. Precise estimates for the
uniform distance are derived by using the techniques of Malliavin calculus,
combined with Stein's method for normal approximation. We need to assume
some non-degeneracy conditions. First, the study is focused on random
variables in a fixed Wiener chaos, and later, the results are extended to
the uniform convergence of the derivatives of the densities and to the case
of random vectors in some fixed chaos, which are uniformly non-degenerate in
the sense of Malliavin calculus. Explicit upper bounds for the uniform norm
are obtained for random variables in the second Wiener chaos, and an
application to the convergence of densities of the least square estimator
for the drift parameter in Ornstein-Uhlenbeck processes is discussed.
\end{abstract}


\end{frontmatter}


\tableofcontents
\section{Introduction}

There has been a recent interest in studying normal approximations for
sequences of multiple stochastic integrals. Consider a sequence of multiple
stochastic integrals of order $q\geq 2$, $F_{n}=I_{q}(f_{n})$, with variance
$\sigma ^{2}>0$, with respect to an isonormal Gaussian process $%
X=\{X(h),h\in \mathfrak{H}\}$ associated with a Hilbert space $\mathfrak{H}$%
. It was proved by Nualart and Peccati \cite{NuPe05} and Nualart and
Ortiz-Latorre \cite{NuOL08} that $F_{n}$ converges in distribution to the
normal law $N(0,\sigma ^{2})$ as $n\rightarrow \infty $ if and only if one
of the following three equivalent conditions holds:

\begin{itemize}
\item[(i)] $\lim_{n\rightarrow \infty }E[F_{n}^{4}]=3\sigma ^{4}$
(convergence of the fourth moments).

\item[(ii)] For all $1\leq r\leq q-1$, $f_{n}\otimes _{r}f_{n}$ converges to
zero, where $\otimes _{r}$denotes the contraction of order $r$ (see equation
(\ref{cntt2})).

\item[(iii)] $\left\Vert DF_{n}\right\Vert _{\mathfrak{H}}^{2}$ (see
definition in Section 2) converges to $q\sigma ^{2}$ in $L^{2}(\Omega )$ as $%
n$ tends to infinity.
\end{itemize}

A new methodology to study normal approximations and to derive quantitative
results combining Stein's method with Malliavin calculus was introduced by
Nourdin and Peccati \cite{NP09ptrf} (see also Nourdin and Peccati \cite{NP12}%
). As an illustration of the power of this method, let us mention the
following estimate for the total variation distance between the law $%
\mathcal{L}(F)$ of $F=I_{q}(f)$ and distribution $\gamma =N(0,\sigma ^{2})$,
where $\sigma ^{2}=E[F^{2}]$:
\begin{equation*}
d_{TV}(\mathcal{L}(F),\gamma )\leq \frac{2}{q\sigma ^{2}}\sqrt{\mathrm{Var}%
\left( \Vert DF\Vert _{\mathfrak{H}}^{2}\right) }\leq \frac{2\sqrt{q-1}}{%
\sigma ^{2}\sqrt{3q}}\sqrt{E[F^{4}]-3\sigma ^{4}}.
\end{equation*}%
This inequality can be used to show the above equivalence (i)-(iii). A
recent result of Nourdin and Poly \cite{NoPo12} says that the convergence in
law for a sequence of multiple stochastic integrals of order $q\geq 2$ is
equivalent to the convergence in total variation if the limit is not
constant. As a consequence, for a sequence $F_{n}$ of nonzero multiple
stochastic integrals of order $q\geq 2$, the limit in law to is equivalent
to the limit of the densities in $L^{1}(\mathbb{R})$, provided the limit is
not constant. A multivariate extension of this result has been derived in
\cite{NNP12}.

\medskip The aim of this paper is to study the uniform convergence of the
densities of a sequence of random vectors $F_{n}$ to the normal density
using the techniques of Malliavin calculus, combined with Stein's method for
normal approximation. It is well-known that to guarantee that each $F_{n}$
has a density we need to assume that the norm of the Malliavin derivative of
$F_{n}$ has negative moments. Thus, a natural assumption to obtain uniform
convergence of densities is to assume uniform boundedness of the negative
moments of the corresponding Malliavin derivatives. Our first result
(Theorem \ref{qRateThm}) says that if $F$ is a multiple stochastic integral
of order $q\geq 2$ such that $E[F^{2}]=\sigma ^{2}$ and $M:=E(\Vert DF\Vert
_{\mathfrak{H}}^{-6})<\infty $, we have
\begin{equation}
\sup_{x\in \mathbb{R}}|f_{F}(x)-\phi (x)|\leq C\sqrt{E[F^{4}]-3\sigma ^{4}},
\label{z1}
\end{equation}%
where $f_{F}$ is the density of $F$, $\phi $ is the density of the normal
law $N(0,\sigma ^{2})$ and the constant $C$ depends on $q$, $\sigma $ and $M$%
. We can also replace the expression in the right-hand side of (\ref{z1}) by
$\sqrt{\mathrm{Var}\left( \Vert DF\Vert _{\mathfrak{H}}^{2}\right) }$. The
main idea to prove this result is to express the density of $F$ using
Malliavin calculus:
\begin{equation*}
f_{F}(x)=E[\mathbf{1}_{\{F>x\}}q\Vert DF\Vert _{\mathfrak{H}}^{-2}F]-E[%
\mathbf{1}_{\{F>x\}}\langle DF,D(\Vert DF\Vert _{\mathfrak{H}}^{-2})\rangle
_{\mathfrak{H}}].
\end{equation*}%
Then, one can find an estimate of the form (\ref{z1}) for the terms $%
E[|\langle DF,D(\Vert DF\Vert _{\mathfrak{H}}^{-2})\rangle _{\mathfrak{H}}|]$
and $E[|q\Vert DF\Vert _{\mathfrak{H}}^{-2}-\sigma ^{-2}|]$. On the other
hand, taking into account that
\begin{equation*}
\phi (x)=\sigma ^{-2}E[\mathbf{1}_{\{N>x\}}N],
\end{equation*}%
it suffices to estimate the difference
\begin{equation*}
E[\mathbf{1}_{\{F>x\}}F]-E[\mathbf{1}_{\{N>x\}}N],
\end{equation*}%
which can be done by Stein's method. The estimate (\ref{z1}) leads to the
uniform convergence of the densities in the above equivalence of conditions
(i) to (iii) if we assume that $\sup_{n}E(\Vert DF_{n}\Vert _{\mathfrak{H}%
}^{-6})<\infty $.

\medskip This methodology is extended in the paper in several directions. We
consider the uniform approximation of the $m$th derivative of the density of
$F$ by the corresponding densities $\phi^{(m)}$, in the case of random
variables in a fixed chaos of order $q\ge 2$. In Theorem \ref{qDeriRate} we
obtain an inequality similar to (\ref{z1}) assuming that $E(\|DF\|^{-\beta}
_{\mathfrak{H}}) <\infty$ for some $\beta >6m +6\left( \lfloor \frac m2
\rfloor \vee 1 \right)$. Again the proof is obtained by a combination of
Malliavin calculus and the Stein's method. Here we need to consider Stein's
equation for functions of the form of $h(x)=\mathbf{1}_{\left\{ x>a\right\}
}p(x)$, where $p$ is a polynomial.

\medskip For a $d$ dimensional random vector $F=(F^{1},\dots ,F^{d})$ whose
components are multiple stochastic integrals of orders $q_{1},\dots ,q_{d}$,
$q_{i}\geq 2$, we assume non degeneracy condition $E[\det \gamma
_{F}^{-p}]<\infty $ for all $p\geq 1$, where $\gamma _{F}=\left( \langle
DF,,DF\rangle \right) _{1\leq i,j\leq d}$ denotes the Malliavin matrix of $F$%
. Then, for any multi-index $\beta =(\beta _{1},\dots ,\beta _{k})$, $1\leq
\beta _{i}\leq d$, we obtain the estimate (see Theorem \ref{MultiThm0})
\begin{equation}
\sup_{x\in \mathbb{R}^{d}}\left\vert \partial _{\beta }f_{F}(x)-\partial
_{\beta }\phi (x)\right\vert \leq C\left( |V-I|^{\frac{1}{2}}+\sum_{j=1}^{d}%
\sqrt{E[F_{j}^{4}]-3(E[F_{j}^{2}])^{2}}\right) ,  \label{multi-main}
\end{equation}
where $V$ is the covariance matrix of $F$, $\phi $ is the standard $d$
dimensional normal density, and $\partial _{\beta }=\frac{\partial ^{k}}{%
\partial x_{\beta _{1}}\cdots \partial x_{\beta _{k}}}$. As a consequence,
we derive the uniform convergence of the densities and their derivatives for
a sequence of vectors of multiple stochastic integrals, under the
assumption $\sup_{n}E[\det \gamma _{F_{n}}^{-p}]<\infty $ for all $p\geq 1$.
A multivariate extension of Stein's method is required for noncontinuous
functions with polynomial growth (see Proposition \ref{Stein-prop}). While
univariate Stein's equations with non-smooth test functions have been
extensively studied, relatively few results are available for the
multivariate case, see \cite{CGS11,CM08,Mec06,NPRv10,Raic04,RR09}, so this
result has its own interest.

\medskip
We also consider the case of random variables $F$ such that $E[F]=0$
and $E[F^{2}]=\sigma ^{2}$, belonging to the Sobolev space $\mathbb{D}^{2,s}$
for some $s>4$. In this case, under a non-degeneracy assumption of the form $%
E[|\langle DF,-DL^{-1}F\rangle _{\mathfrak{H}}|^{-r}|]<\infty $ for some $%
r>2 $, we derive an estimate for the uniform distance between the density of
$F$ and the density of the normal law $N(0,\sigma ^{2})$.

\bigskip

In a recent paper \cite{NPS13}, Nourdin, Peccati and Swan have obtained an
upper bound on the total variation distance between the law of a vector of
multiple stochastic integrals and a normal distribution, using a combination
of entropy techniques and Malliavin calculus. Their main result can be
briefly stated as follow. Let $F=(F^{1},\dots ,F^{d})$ be a $d$ dimensional
random vector whose components are multiple stochastic integrals of orders $%
q_{1},\dots ,q_{d}$, $q_{i}\geq 2$, respectively. Suppose the covariance of $%
F$ is the identity matrix. Denote $\phi (x)$ the density of $N\sim N(0,%
\mathrm{Id})$. Then the relative entropy $D(F||N)$ of $F$ satisfies
$D(F||N):=E[\log f_{F}(F)-\log \phi (F)]\leq C\Delta \left\vert \log \Delta
\right\vert$,
where $C>0$ is a constant and $\Delta =E[\left\vert F\right\vert
^{4}-\left\vert N\right\vert ^{4}]$. This leads to the bound
\begin{equation*}
\left\Vert f_{F}-\phi \right\Vert _{L^{1}(\mathbb{R}^{d})}\leq \sqrt{2D(F||N)%
}\leq C\sqrt{\Delta \left\vert \log \Delta \right\vert }\,. \label{e.tv}
\end{equation*}%
This result refines some estimates obtained in \cite{NNP12}. In the case $d=1
$, this result also refines our estimate (\ref{f-Lp}), where by taking $p=1$
and $\alpha $ close to $\frac{1}{2}$ we can only get $\Delta ^{\frac{1%
}{4}-\epsilon}$ with $\epsilon>0$ arbitrarily small.

Convergence of densities in uniform distance (also in total variation) has
also been studied using Fisher information theory via Shimizu's inequality (see
for instance, \cite{Sh75,BCG12,BJ04,Joh04})
\begin{equation}
\sup_{x\in \mathbb{R}}|f_{F}(x)-\phi (x)|\leq C\sqrt{I(F||N)},
\label{e.fisher}
\end{equation}
where $F$ is a random variable with density $f\in C^{1}(\mathbb{R})$, $\phi $
is the density of $N\sim N(0,1)$, and $I(F||N):=\int_{\mathbb{R}}\left(
\frac{f_{F}^{\prime }(x)}{f_{F}(x)}-\frac{\phi ^{\prime }(x)}{\phi (x)}%
\right) ^{2}f(x)dx$ is the relative Fisher information. Recently, Bobkov,
Chistyakov and G\"otze \cite{BCG12} studied the rate of convergence to $0$ of
$I(F_{n}||N)$ for $F_{n}=\frac{1}{\sqrt{n}}\sum_{i=1}^{n}X_{i}$, where $%
\left\{ X_{i}\right\} _{i\geq a}$ are i.i.d. random variables with mean $0$
and variance $1$, assuming that $f_{F_{n_{0}}}^{\prime
}\in L^{1}(\mathbb{R})$ for some $n_{0}$.

In general, when studying uniform convergence of densities, one is necessarily
led to introduce some stringent assumptions on the regularity of the laws of
the underlying random variables. Here we showed that these assumptions can
be reduced to requirements about the finiteness of the negative moments of
Malliavin matrices.

\bigskip

The paper is organized as follows. Section \ref{Prelim} introduces some
preliminary results of Gaussian analysis, Malliavin calculus and Stein's
method for normal approximations. Section \ref{DFmla} is devoted to density
formulae with elementary estimates using Malliavin calculus. The density
formulae themselves are well-known results, but we present explicit formulae
with useful estimates, such as the H\"{o}lder continuity and boundedness
estimates in theorems $\ref{density}$ and $\ref{density2}$. The boundedness
estimates enable us to prove the $L^{p}$ convergence of the densities (see (%
\ref{f-Lp})). The H\"{o}lder continuity estimates can be used to provide a
short proof for the convergence of densities based on a compactness
argument, assuming convergence in law (see Theorem \ref{gThm}). Section \ref%
{ChaosConv} proves the convergence of densities of random variables in a
fixed Wiener chaos, and Section \ref{MultiConv} discusses convergence of
densities for random vectors. In Section \ref{GenrlConv}, the convergence of
densities for sequences of general centered square integrable random
variables are studied.

\medskip The main difficulty in the application of the above results is to
verify the existence of negative moments for the determinant of the
Malliavin matrix. We provide explicit sufficient conditions for this
condition for random variables in the second Wiener chaos in Section \ref%
{Application}. As an application we derive the uniform convergence of the
densities and their derivatives to the normal distribution, as time goes to
infinity, for the least squares estimator of the parameter $\theta $ in the
Ornstein-Uhlenbeck process: $dX_{t}=-\theta X_{t}dt+\gamma dB_{t}$, where $%
B=\{B_{t},t\geq 0\}$ is a standard Brownian motion. Some technical results
and proofs are included in Section \ref{Appdx}.

\medskip Along this paper, we denote by $C$ (maybe with subindexes) a
generic constant that might depend on quantities such as the order of
multiple stochastic integrals $q$, the order of the derivatives $m$, the
variance $\sigma^2$ or the negative moments of the Malliavin derivative. We
denote by $\left\Vert \cdot \right\Vert _{p}$ the norm in the space $%
L^{p}(\Omega )$.

\section{Preliminaries\label{Prelim}}

In the first two subsections, we introduce some basic elements of Gaussian
analysis and Malliavin calculus, for which we refer to \cite{Nu06,NP12} for
further details. In the last subsection, we shall introduce some basic
estimates from the univariate Stein's method.

\subsection{Isonormal Gaussian process and multiple integrals}

Let $\mathfrak{H}$ be a real separable Hilbert space (with its inner product
and norm denoted by\thinspace $\left\langle \cdot ,\cdot \right\rangle _{%
\mathcal{\mathfrak{H}}}$ and $\left\Vert \cdot \right\Vert _{\mathfrak{H}}$,
respectively). For any integer $q\geq 1$, let $\mathfrak{H}^{\otimes q}$($%
\mathfrak{H}^{\odot q}$) be the $q$th tensor product (symmetric tensor
product) of $\mathfrak{H}$. Let $X=\{X(h),h\in \mathfrak{H}\}$ be an
isonormal Gaussian process associated with the Hilbert space $\mathfrak{H}$,
defined on a complete probability space $\left( \Omega ,\mathcal{F},P\right)
$. That is, $X$ is a centered Gaussian family of random variables such that $%
E[X(h)X(g)]=\left\langle h,g\right\rangle _{\mathfrak{H}}$ for all $h,g\in
\mathfrak{H}$. We assume that the $\sigma $-field $\mathcal{F}$ is generated
by $X$.

For every integer $q\geq 0$, the \emph{$q$th Wiener chaos} (denoted by $%
\mathcal{H}_{q}$) of $X$ is the closed linear subspace of $L^{2}(\Omega )$
generated by the random variables $\left\{ H_{q}(X(h)):h\in \mathfrak{H}%
,\left\Vert h\right\Vert _{\mathfrak{H}}=1\right\} $, where $H_{q}$ is the $%
q $th \emph{Hermite polynomial} recursively defined by $H_{0}(x)=1$, $%
H_{1}(x)=x$ and
\begin{equation}
H_{q+1}(x)=xH_{q}(x)-qH_{q-1}(x),\quad q\geq 1.  \label{HmtP}
\end{equation}%
For every integer $q\geq 1$, the mapping $I_{q}(h^{\otimes q})=H_{q}(X(h))$,
where $\Vert h\Vert _{\mathfrak{H}}=1$, can be extended to a linear isometry
between $\mathfrak{H}^{\odot q}$ (equipped with norm $\sqrt{q!}\left\Vert
\cdot \right\Vert _{\mathfrak{H}^{\otimes q}}$) and $\mathcal{H}_{q}$
(equipped with $L^{2}(\Omega )$ norm). For $q=0$, $\mathcal{H}_{0}=\mathbb{R}
$, and $I_{0}$ is the identity map. The mapping $I_{q}$ is called the
multiple stochastic integral of order $q$.

It is well-known (Wiener chaos expansion) that $L^{2}(\Omega )$ can be
decomposed into the infinite orthogonal sum of the spaces $\mathcal{H}_{q}$.
That is, any random variable $F\in L^{2}(\Omega )$ has the following chaos
expansion:%
\begin{equation}
F=\sum_{q=0}^{\infty }I_{q}(f_{q}),  \label{Chaos}
\end{equation}%
where $f_{0}=E[F]$, and $f_{q}\in \mathfrak{H}^{\odot q},q\geq 1$, are
uniquely determined by $F$. \ For every $q\geq 0$ we denote by $J_{q}$ the
orthogonal projection on the $q$th Wiener chaos $\mathcal{H}_{q}$, so $%
I_{q}(f_{q})=J_{q}F$.

Let $\left\{ e_{n},n\geq 1\right\} $ be a complete orthonormal basis of $%
\mathfrak{H}$. Given $f\in \mathfrak{H}^{\odot q}$ and $g\in \mathfrak{H}%
^{\odot p}$, for $r=0,\dots ,p\wedge q$ the $r$--th contraction of $f$
and $g$ is the element of $\mathfrak{H}^{\otimes \left( p+q-2r\right) }$
defined by
\begin{equation}
f\otimes _{r}g=\sum_{i_{1},\dots ,i_{r}=1}^{\infty }\left\langle
f,e_{i_{1}}\otimes \cdots \otimes e_{i_{r}}\right\rangle _{\mathfrak{H}%
^{\otimes r}}\otimes \left\langle g,e_{i_{1}}\otimes \cdots \otimes
e_{i_{r}}\right\rangle _{\mathfrak{H}^{\otimes r}}.  \label{cntt1}
\end{equation}%
Notice that $f\otimes _{r}g$ is not necessarily symmetric. We denote by $f%
\widetilde{\otimes }_{r}g$ its symmetrization. Moreover, $f\otimes _{0}g=$ $%
f\otimes g$, and for $p=q$, $f\otimes _{q}g=\left\langle f,g\right\rangle _{%
\mathfrak{H}^{\otimes q}}$. For the product of two multiple stochastic
integrals we have the multiplication formula%
\begin{equation}
I_{p}(f)I_{q}(g)=\sum_{r=0}^{p\wedge q}r!\binom{p}{r}\binom{q}{r}%
I_{p+q-2r}(f\otimes _{r}g).  \label{MltFml}
\end{equation}

In the particular case $\mathfrak{H}=L^{2}(A,\mathcal{A},\mu )$, where $(A,%
\mathcal{A)}$ is a measurable space and $\mu $ is a $\sigma $-finite and
nonatomic measure, one has that $\mathfrak{H}^{\otimes q}=L^{2}(A^{q},%
\mathcal{A}^{\otimes q},\mu ^{\otimes q})$ and $\mathfrak{H}^{\odot q}$ is
the space of symmetric and square-integrable functions on $A^{q}$. Moreover,
for every $f\in \mathfrak{H}^{\odot q}$, $I_{q}(f)$ coincides with the $q$th%
\textit{\ multiple Wiener--It\^{o} integral} of $f$ with respect to $X$, and
($\ref{cntt1}$) can be written as
\begin{eqnarray}
f\otimes _{r}g\left( t_{1},\dots ,t_{p+q-2r}\right) &=&\int_{A^{r}}f\left(
t_{1},\dots ,t_{q-r},s_{1},\dots ,s_{r}\right)  \label{cntt2} \\
&&\times g\left( t_{1+q-r},\dots ,t_{p+q-r},s_{1},\dots ,s_{r}\right) d\mu
(s_{1})\dots d\mu (s_{r}).  \notag
\end{eqnarray}

\subsection{Malliavin operators}

We introduce some basic facts on Malliavin calculus with respect to the
Gaussian process $X$. Let $\mathcal{S}$ denote the class of smooth random
variables of the form $F=f(X(h_{1}),\dots ,X(h_{n}))$, where $h_{1},\dots
,h_{n}$ are in $\mathfrak{H}$, $n\geq 1$, and $f$ $\in $ $C_{p}^{\infty }(%
\mathbb{R}^{n})$, the set of smooth functions $f$ such that $f$ itself and
all its partial derivatives have at most polynomial growth. Given $%
F=f(X(h_{1}),\dots ,X(h_{n}))$ in $\mathcal{S}$, its Malliavin derivative $%
DF $ is the $\mathfrak{H}$--valued random variable given by
\begin{equation*}
DF=\sum_{i=1}^{n}\frac{\partial f}{\partial x_{i}}(X(h_{1}),\dots
,X(h_{n}))h_{i}.
\end{equation*}%
The derivative operator $D$ is a closable and unbounded operator on $%
L^{2}(\Omega )$ taking values in $L^{2}(\Omega ;\mathfrak{H})$. By iteration
one can define higher order derivatives $D^{k}F$ $\in L^{2}(\Omega ;%
\mathfrak{H}^{\odot k})$. For any integer $k\geq 0$ and any $p\geq 1$ and we
denote by $\mathbb{D}^{k,p}$ the closure of $\mathcal{S}$ with respect to
the norm $\left\Vert \cdot \right\Vert _{k,p}$ given by:
\begin{equation*}
\left\Vert F\right\Vert _{k,p}^{p}=\sum_{i=0}^{k}E(\left\vert \left\vert
D^{i}F\right\vert \right\vert _{\mathfrak{H}^{\otimes i}}^{p}).
\end{equation*}%
For $k=0$ we simply write $\Vert F\Vert _{0,p}=\Vert F\Vert _{p}$. For any $%
p\geq 1$ and $k\geq 0$, we set $\mathbb{D}^{\infty ,p}=\cap _{k\geq 0}%
\mathbb{D}^{k,p}$, $\mathbb{D}^{k,\infty }=\cap _{p\geq 1}\mathbb{D}^{k,p}$
and $\mathbb{D}^{\infty }=\cap _{k\geq 0}\mathbb{D}^{k,\infty }$.

We denote by $\delta $ (the $divergence$ operator) the adjoint operator of $%
D $, which is an unbounded operator from a domain in $L^{2}(\Omega ;%
\mathfrak{H})$ to $L^{2}(\Omega )$. An element $u\in L^{2}(\Omega ;\mathfrak{%
H})$ belongs to the domain of $\delta $ if and only if it verifies%
\begin{equation*}
\left\vert E[\left\langle DF,u\right\rangle _{\mathfrak{H}}]\right\vert \leq
c_{u}\sqrt{E[F^{2}]}
\end{equation*}%
for any $F\in \mathbb{D}^{1,2}$, where $c_{u}$ is a constant depending only
on $u$. In particular, if $u\in \mathrm{Dom~}\delta $, then $\delta (u)$ is
characterized by the following duality relationship
\begin{equation}
E(\delta (u)F)=E(\langle DF,u\rangle _{\mathfrak{H}})  \label{duality}
\end{equation}%
for any $F\in \mathbb{D}^{1,2}$. This formula extends to the multiple
integral $\delta ^{q}$, that is, for $u\in \mathrm{Dom~}\delta ^{q}$ and $%
F\in \mathbb{D}^{q,2}$ we have%
\begin{equation*}
E(\delta ^{q}(u)F)=E(\langle D^{q}F,u\rangle _{\mathfrak{H}^{\otimes q}}).
\end{equation*}

We can factor out a scalar random variable in the divergence in the
following sense. Let $F\in \mathbb{D}^{1,2}$ and $u\in \mathrm{Dom}~\delta $
such that $Fu\in L^{2}(\Omega ;\mathfrak{H})$. Then $Fu\in \mathrm{Dom}%
~\delta $ and
\begin{equation}
\delta \left( Fu\right) =F\delta (u)-\left\langle DF,u\right\rangle _{%
\mathfrak{H}},  \label{factorOut}
\end{equation}%
provided the right-hand side is square integrable. The operators $\delta $
and $D$ have the following commutation relationship
\begin{equation}
D\delta (u)=u+\delta (Du)  \label{DeltaD}
\end{equation}%
for any $u\in \mathbb{D}^{2,2}(\mathfrak{H})$ (see \cite[page 37]{Nu06}).

The following version of Meyer's inequality (see \cite[Proposition 1.5.7]%
{Nu06}) will be used frequently in this paper. Let $V$ be a real separable
Hilbert space. We can also introduce Sobolev spaces $\mathbb{D}^{k,p}(V)$ of
$V$-valued random variables for $p\geq 1$ and integer $k\geq 1$. Then, for
any $p>1$ and $k\geq 1$, the operator $\delta $ is continuous from $\mathbb{D%
}^{k,p}(V\otimes \mathfrak{H})$ into $\mathbb{D}^{k-1,p}(V)$. That is,
\begin{equation}
\left\Vert \delta (u)\right\Vert _{k-1,p}\leq C_{k,p}\left\Vert u\right\Vert
_{k,p}.  \label{Meyer}
\end{equation}

The operator $L$ defined on the Wiener chaos expansion as $%
L=\sum_{q=0}^{\infty }(-q)J_{q}$ is the infinitesimal generator of the
Ornstein--Uhlenbeck semigroup $T_{t}=\sum_{q=0}^{\infty }e^{-qt}J_{q}$. Its
domain in $L^{2}(\Omega )$ is
\begin{equation*}
\mathrm{Dom}\ L=\left\{ F\in L^{2}(\Omega ):\sum_{q=1}^{\infty
}q^{2}\left\Vert J_{q}F\right\Vert _{2}^{2}<\infty \right\} =\mathbb{D}%
^{2,2}.
\end{equation*}%
The relation between the operators $D$, $\delta $ and $L$ is explained in
the following formula (see \cite[Proposition 1.4.3]{Nu06}). For $F\in
L^{2}(\Omega )$, $F\in \mathrm{Dom}~L$ if and only if $F\in \mathrm{Dom}%
(\delta D)$ (i.e., $F\in \mathbb{D}^{1,2}$ and $DF\in \mathrm{Dom}~\delta $%
), and in this case

\begin{equation}
\delta DF=-LF.  \label{DeltaDL}
\end{equation}

For any $F\in L^{2}(\Omega )$, we define $L^{-1}F=-\sum_{q=1}^{\infty
}q^{-1}J_{q}(F)$. The operator $L^{-1}$ is called the pseudo-inverse of $L$.
Indeed, for any $F\in L^{2}(\Omega )$, we have that $L^{-1}F\in \mathrm{Dom}%
\ L$, and
\begin{equation*}
LL^{-1}F=L^{-1}LF=F-E[F].
\end{equation*}

We list here some properties of multiple integrals which will be used in
Section 4. Fix $q\geq 1$ and let $f\in \mathfrak{H}^{\odot q}$. We have $%
I_{q}(f)=\delta ^{q}(f)$ and $DI_{q}(f)=qI_{q-1}(f)$, and hence $I_{q}(f)\in
\mathbb{D}^{\infty ,2}$. The multiple stochastic integral $I_{q}(f)$
satisfies \emph{hypercontractivity} property:
\begin{equation}
\left\Vert I_{q}(f)\right\Vert _{p}\leq C_{q,p}\left\Vert
I_{q}(f)\right\Vert _{2}\text{ for any }p\in \lbrack 2,\infty ).
\label{Hyper}
\end{equation}%
This easily implies that $I_{q}(f)\in \mathbb{D}^{\infty }$ and for any $%
1\leq k\leq q$ and $p\geq 2$,
\begin{equation*}
\Vert I_{q}(f)\Vert _{k,p}\leq C_{q,k,p}\left\Vert I_{q}(f)\right\Vert _{2}.
\end{equation*}%
As a consequence, for any $F\in \oplus _{l=1}^{q}\mathcal{H}_{l}$, we have
\begin{equation}
\left\Vert F\right\Vert _{k,p}\leq C_{q,k,p}\left\Vert F\right\Vert _{2}.
\label{HyperMeyer}
\end{equation}
For any random variable $F$ in the chaos of order $q\ge 2$, we have (see
\cite{NP12}, Equation (5.2.7))
\begin{equation}  \label{equi}
\frac 1 {q^2} \mathrm{Var} \left(\|DF \|^2_{\mathfrak{H}} \right) \le \frac {%
q-1}{3q} \left( E[F^4]- (E[F^2])^2 \right) \le (q-1) \mathrm{Var} \left(\|DF
\|^2_{\mathfrak{H}} \right).
\end{equation}

In the case where $\mathfrak{H}$ is $L^{2}(A,\mathcal{A},\mu )$, for an
integrable random variable $F=\sum_{q=0}^{\infty }I_{q}(f_{q})\in \mathbb{D}%
^{1,2}$, its derivative can be represented as an element in of $L^{2}\left(
A\times \Omega \right) $ given by
\begin{equation*}
D_{t}F=\sum_{q=1}^{\infty }qI_{q}(f_{q}(\cdot ,t)),~t\in A.
\end{equation*}

\subsection{Stein's method of normal approximation\label{UnivStein}}

We shall now give a brief account of Stein's method of univariate normal
approximation and its connection with Malliavin calculus. For a more
detailed exposition we refer to \cite{CGS11,NP12,St72}.

Let $F$ be an arbitrary random variable and let $N$ be a $N(0,\sigma ^{2})$
distributed random variable, where $\sigma^2>0$. Consider the distance
between the law of $F$ and the law of $N$ given by
\begin{equation}
d_{\mathcal{H}}(F,N)=\sup_{h\in \mathcal{H}}\left\vert
E[h(F)-h(N)]\right\vert  \label{distance}
\end{equation}%
for a class of functions $\mathcal{H}$ such that $E[h(F)]$ and $E[h(N)]$ are
well-defined for $h\in \mathcal{H}$. Notice first the following fact (which
is usually referred as \emph{Stein's lemma}): a random variable $N$ is $%
N(0,\sigma ^{2})$ distributed if and only if $E[\sigma ^{2}f^{\prime
}(N)-Nf(N)]=0$ for all absolutely continuous functions $f$ such that $%
E[\left\vert f^{\prime }(N)\right\vert ]<\infty $. This suggests that the
distance of $E[\sigma ^{2}f^{\prime }(F)-Ff(F)]$ from zero may quantify the
distance between the law of $F$ and the law of $N$. To see this, for each
function $h$ such that $E[\left\vert h(N)\right\vert ]<\infty $, Stein \cite%
{St72} introduced the \emph{Stein's equation}:%
\begin{equation}
f^{\prime }(x)-\frac{x}{\sigma ^{2}}f(x)=h(x)-E[h(N)]  \label{Stein's equ}
\end{equation}%
for all $x\in \mathbb{R}$. For a random variable $F$ such that $E[\left\vert
h(F)\right\vert ]<\infty $, any solution $f_{h}$ to Equation (\ref{Stein's
equ}) verifies
\begin{equation}
\frac{1}{\sigma ^{2}}E[\sigma ^{2}f_{h}^{\prime }(F)-Ff_{h}(F)]=E[h(F)-h(N)],
\label{S-equ}
\end{equation}%
and the distance defined in (\ref{distance}) can be written as
\begin{equation}
d_{\mathcal{H}}(F,N)=\frac{1}{\sigma ^{2}}\sup_{h\in \mathcal{H}}\left\vert
E[\sigma ^{2}f_{h}^{\prime }(F)-Ff_{h}(F)]\right\vert .  \label{p-metric}
\end{equation}%
The unique solution to (\ref{Stein's equ}) verifying $\lim_{x\rightarrow \pm
\infty }$ $e^{-x^{2}/(2\sigma ^{2})}f(x)=0$ is
\begin{equation}
f_{h}(x)=e^{x^{2}/(2\sigma ^{2})}\int_{-\infty
}^{x}\{h(y)-E[h(N)]\}e^{-y^{2}/(2\sigma ^{2})}dy.  \label{solu}
\end{equation}

From (\ref{p-metric}) and (\ref{solu}), one can get bounds for probability
distances like the total variation distance, where we let $\mathcal{H}$
consist of all indicator functions of measurable sets, Kolmogorov distance,
where we consider all the half-line indicator functions and Wasserstein
distance, where we take $\mathcal{H}$\ to be the set of all
Lipschitz-continuous functions with Lipschitz constant equal to 1.

In the present paper, we shall consider the case when $h:$ $\mathbb{%
R\rightarrow R}$ is given by $h(x)=\mathbf{1}_{\left\{ x>z\right\} }H_{k}(x)$
for any integer $k\geq 1$ and $z\in \mathbb{R}$, where $H_{k}(x)$ is the $k$%
th Hermite polynomial. More generally, we have the following lemma whose
proof can be found in the Appendix. it should be pointed out that the
univariate Stein's equations have been extensively studied. For example, we
refer to \cite[Section 2.2]{CGS11} and \cite[Lemma 8.2]{NPR10} when the test
functions have sub-polynomial growth.

\begin{lemma}
\label{solu-ctrl} Suppose $\left\vert h(x)\right\vert \leq a\left\vert
x\right\vert ^{k}+b$ for some integer $k\geq 0$ and some nonnegative numbers
$a,b$. Then, the solution $f_{h}$ to the Stein's equation (\ref{Stein's equ}%
) given by (\ref{solu}) satisfies%
\begin{equation*}
\left\vert f_{h}^{\prime }(x)\right\vert \leq aC_{k}\sum_{i=0}^{k}\sigma
^{k-i}\left\vert x\right\vert ^{i}+4b
\end{equation*}%
for all $x\in \mathbb{R}$, where $C_{k}$ is a constant depending only on $k$.
\end{lemma}

Nourdin and Peccati \cite{NP09ptrf, NP12} combined Stein's method with
Malliavin calculus to estimate the distance between the distributions of
regular functionals of an isonormal Gaussian process and the normal
distribution $N(0,\sigma ^{2})$. The basic ingredient is the following
integration by parts formula. For $F\in \mathbb{D}^{1,2}$ with $E[F]=0$ and
any function $f\in C^{1}$ such that $E[|f^{\prime }(F)|]<\infty $, using (%
\ref{DeltaDL}) and (\ref{duality}) we have
\begin{eqnarray*}
E[Ff(F)] &=&E[LL^{-1}Ff(F)]=E[-\delta DL^{-1}Ff(F)] \\
&=&E[\left\langle -DL^{-1}F,Df(F)\right\rangle ]=E[f^{\prime
}(F)\left\langle -DL^{-1}F,DF\right\rangle _{\mathfrak{H}}].
\end{eqnarray*}%
Then, it follows that
\begin{equation}
E[\sigma ^{2}f^{\prime }(F)-Ff(F)]=E[f^{\prime }(F)(\sigma ^{2}-\left\langle
DF,-DL^{-1}F\right\rangle _{\mathfrak{H}})].  \label{MS-equ}
\end{equation}%
Combining Equation (\ref{MS-equ}) with (\ref{S-equ}) and Lemma \ref%
{solu-ctrl} we obtain the following result.

\begin{lemma}
\label{MS-ctrl}Suppose $h:$ $\mathbb{R\rightarrow R}$ verifies $\left\vert
h(x)\right\vert \leq a\left\vert x\right\vert ^{k}+b$ for some $a,b\geq 0$
and some integer $k\geq 0$. Let $N\sim N(0,\sigma ^{2})$ and let $F\in
\mathbb{D}^{1,2k}$ with $\left\Vert F\right\Vert _{2k}\leq c\sigma $ for
some $c>0$. Then there exists a constant $C_{k,c}$ depending only on $k$ and
$c$\ such that%
\begin{equation*}
\left\vert E[h(F)-h(N)]\right\vert \leq \sigma ^{-2}[aC_{k,c}\sigma
^{k}+4b]\left\Vert \sigma ^{2}-\left\langle DF,-DL^{-1}F\right\rangle _{%
\mathfrak{H}}\right\Vert _{2}.
\end{equation*}
\end{lemma}

\begin{proof}
From (\ref{S-equ}), (\ref{MS-equ}) and Lemma (\ref{solu-ctrl}), it suffices
to notice that $\left\Vert \sum_{i=0}^{k}\sigma ^{k-i}\left\vert
F\right\vert ^{i}\right\Vert _{2}\leq \sum_{i=0}^{k}\left\Vert F\right\Vert
_{2k}^{i}\sigma ^{k-i}\leq C_{k,c}\sigma ^{k}\,.$
\end{proof}

\section{Density formulae\label{DFmla}}

In this section, we present explicit formulae for the density of a random
variable and its derivatives, using the techniques of Malliavin calculus.

\subsection{Density formulae}

We shall present two explicit formulae for the density of a random variable,
with estimates of its uniform and H\"{o}lder norms.

\begin{theorem}
\label{density} Let $F\in \mathbb{D}^{2,s}$ such that $E[\left\vert
F\right\vert ^{2p}]<\infty $ and $E[\left\Vert DF\right\Vert _{\mathfrak{H}%
}^{-2r}]<\infty $ for $p,r,s>1$ satisfying $\frac{1}{p}+\frac{1}{r}+\frac{1}{%
s}=1$. Denote
\begin{equation*}
w=\left\Vert DF\right\Vert _{\mathfrak{H}}^{2},~u=w^{-1}DF.
\end{equation*}%
Then $u\in \mathbb{D}^{1,p^{\prime }}$ with $p^{\prime }=\frac{p}{p-1}$ and $%
F$ has a density given by%
\begin{equation}
f_{F}\left( x\right) =E\left[ \mathbf{1}_{\{F>x\}}\delta \left( u\right) %
\right] .  \label{Fmla1}
\end{equation}%
Furthermore, $f_{F}\left( x\right) $ is bounded and H\"{o}lder continuous of
order $\frac{1}{p}$, that is%
\begin{equation}
f_{F}\left( x\right) \leq C_{p}\left\Vert w^{-1}\right\Vert _{r}\left\Vert
F\right\Vert _{2,s}\left( 1\wedge (\left\vert x\right\vert ^{-2}\left\Vert
F\right\Vert _{2p}^{2})\right) ,  \label{fBd}
\end{equation}%
\begin{equation}
\left\vert f_{F}\left( x\right) -f_{F}\left( y\right) \right\vert \leq
C_{p}\left\Vert w^{-1}\right\Vert _{r}^{1+\frac{1}{p}}\left\Vert
F\right\Vert _{2,s}^{1+\frac{1}{p}}\left\vert x-y\right\vert ^{\frac{1}{p}}
\label{fHld}
\end{equation}%
for any $x,y\in \mathbb{R}$, where $C_{p}$ is a constant depending only on $%
p $.
\end{theorem}

\begin{proof}
Note that
\begin{equation*}
Du=w^{-1}D^{2}F-2w^{-2}\left( D^{2}F\otimes _{1}DF\right) \otimes DF.
\end{equation*}%
Applying Meyer's inequality (\ref{Meyer}) and H\"{o}lder's inequality we
have
\begin{eqnarray}
\left\Vert \delta \left( u\right) \right\Vert _{p^{\prime }} &\leq
&C_{p}\left\Vert u\right\Vert _{1,p^{\prime }}\leq C_{p}(\left\Vert
u\right\Vert _{p^{\prime }}+\left\Vert Du\right\Vert _{p^{\prime }})  \notag
\\
&\leq &C_{p}(\left\Vert w^{-1}\left\Vert DF\right\Vert _{\mathfrak{H}%
}\right\Vert _{p^{\prime }}+3\left\Vert w^{-1}\left\Vert D^{2}F\right\Vert _{%
\mathfrak{H}\otimes \mathfrak{H}}\right\Vert _{p^{\prime }})  \notag \\
&\leq &3C_{p}\left\Vert w^{-1}\right\Vert _{r}(\left\Vert DF\right\Vert
_{s}+\left\Vert D^{2}F\right\Vert _{s}).  \label{deltaBd}
\end{eqnarray}%
Then $u\in \mathbb{D}^{1,p^{\prime }}$ and the density formula (\ref{Fmla1})
holds (see, for instance, Nualart \cite[Proposition 2.1.1]{Nu06}). From $%
E[\delta (u)]=0$ and H\"{o}lder's inequality it follows that
\begin{equation}
\left\vert E\left[ \mathbf{1}_{\left\{ F>x\right\} }\delta \left( u\right) %
\right] \right\vert \leq P\left( \left\vert F\right\vert >\left\vert
x\right\vert \right) ^{\frac{1}{p}}\left\Vert \delta \left( u\right)
\right\Vert _{^{p^{\prime }}}\leq \left( 1\wedge (\left\vert x\right\vert
^{-2p}\left\Vert F\right\Vert _{2p}^{2p})\right) ^{\frac{1}{p}}\left\Vert
\delta \left( u\right) \right\Vert _{^{p^{\prime }}}.  \label{dens1-1}
\end{equation}%
Then (\ref{fBd}) follows from (\ref{dens1-1}) and (\ref{deltaBd}).

Finally, for $x<y$ $\in \mathbb{R}$, noticing that $\mathbf{1}_{\{F>x\}}-%
\mathbf{1}_{\{F>y\}}=\mathbf{1}_{\{x<F\leq y\}}$, we have
\begin{equation*}
\left\vert f_{F}\left( x\right) -f_{F}\left( y\right) \right\vert \leq
\left( E[\mathbf{1}_{\{x<F\leq y\}}]\right) ^{\frac{1}{p}}\left\Vert \delta
\left( u\right) \right\Vert _{p^{\prime }}.
\end{equation*}%
Applying (\ref{fBd}) and (\ref{deltaBd}) with the fact that $E[\mathbf{1}%
_{\{x<F\leq y\}}]=$ $\int_{x}^{y}f_{F}\left( z\right) dz$ one gets (\ref%
{fHld}).
\end{proof}

With the exact proof of \cite[Propositions 2.1.1]{Nu06}, one can prove the
following slightly more general result.

\begin{proposition}
Let $F\in \mathbb{D}^{1,p}$ and $h:\Omega \rightarrow \mathfrak{H}$, and
suppose that $\left\langle DF,h\right\rangle _{\mathfrak{H}}\neq 0$ a.s. and
$\frac{h}{\left\langle DF,h\right\rangle _{\mathfrak{H}}}\in \mathbb{D}%
^{1,q}(\mathfrak{H})$ for some $p,q>1$. Then the law of $F$ has a density
given by%
\begin{equation}
f_{F}\left( x\right) =E\left[ \mathbf{1}_{\left\{ F>x\right\} }\delta \left(
\frac{h}{\left\langle DF,h\right\rangle _{\mathfrak{H}}}\right) \right] .
\label{Fmla2}
\end{equation}
\end{proposition}

Our next goal is to take $h$ to be $-DL^{-1}F$ in formula (\ref{Fmla2}) and
get a result similar to Theorem \ref{density}. First, to get a sufficient
condition for $\frac{-DL^{-1}F}{\left\langle DF,-DL^{-1}F\right\rangle _{%
\mathfrak{H}}}\in \mathbb{D}^{1,p^{\prime }}$ for some $p^{\prime }>1$, we
need some technical estimates on $DL^{-1}F$ and $D^{2}L^{-1}F$. Estimates of
this type have been obtained by Nourdin, Peccati and Reinert \cite{NPR09}
(see also Nourdin and Peccati's book \cite[Lemma 5.3.8]{NP12}), when proving
an infinite-dimensional Poincar\'{e} inequality. More precisely, by using
Mehler's formula, they proved that for any $p\geq 1$, if $F\in \mathbb{D}%
^{2,p}$, then
\begin{equation}
E[\left\Vert DL^{-1}F\right\Vert _{\mathfrak{H}}^{p}]\leq E[\left\Vert
DF\right\Vert _{\mathfrak{H}}^{p}].  \label{NP1}
\end{equation}%
\begin{equation}
E[\left\Vert D^{2}L^{-1}F\right\Vert _{op}^{p}]\leq 2^{-p}E[\left\Vert
D^{2}F\right\Vert _{op}^{p}],  \label{NP2}
\end{equation}%
where $\left\Vert D^{2}F\right\Vert _{op}$ denotes the operator norm of the
Hilbert-Schmidt operator from $\mathfrak{H}$ to $\mathfrak{H}:f\mapsto
f\otimes _{1}D^{2}F$. Furthermore, the operator norm $\left\Vert
D^{2}F\right\Vert _{op}$ satisfies the following \textquotedblleft \emph{%
random contraction inequality}\textquotedblright\
\begin{equation}
\left\Vert D^{2}F\right\Vert _{op}^{4}\leq \left\Vert D^{2}F\otimes
_{1}D^{2}F\right\Vert _{\mathfrak{H}^{\otimes 2}}^{2}\leq \left\Vert
D^{2}F\right\Vert _{\mathfrak{H}^{\otimes 2}}^{4}.  \label{rci}
\end{equation}

Sometimes in application, the use of $L^{-1}$ in the integration by parts
formula is useful. The next proposition gives a density formula with
estimates similar to Theorem \ref{density} with the use of $L^{-1}$. Let
\begin{equation*}
\bar{w}=\left\langle DF,-DL^{-1}F\right\rangle _{\mathfrak{H}},~\bar{u}=-%
\bar{w}^{-1}DL^{-1}F.
\end{equation*}

\begin{proposition}
\label{density2} Let $F\in \mathbb{D}^{2,s}$, $E[\left\vert F\right\vert
^{2p}]<\infty $ and suppose that $E[\left\vert \bar{w}\right\vert
^{-r}]<\infty $, where $p>1$, $r>2$, $s>3$ satisfy $\frac{1}{p}+\frac{2}{r}+%
\frac{3}{s}=1$. Then $\bar{u}\in \mathbb{D}^{1,p^{\prime }}$ with $p^{\prime
}=\frac{p}{p-1}$ and the law of $F$ has a density given by%
\begin{equation}
f_{F}\left( x\right) =E\left[ \mathbf{1}_{\left\{ F>x\right\} }\delta \left(
\bar{u}\right) \right] .  \label{Fmla3}
\end{equation}%
Furthermore, $f_{F}\left( x\right) $ is bounded and H\"{o}lder continuous of
order $\frac{1}{p}$, that is%
\begin{equation}
f_{F}\left( x\right) \leq K_{0}\left( 1\wedge (\left\vert x\right\vert
^{-2}\left\Vert F\right\Vert _{2p}^{2})\right) ,  \label{fBd2}
\end{equation}%
\begin{equation}
\left\vert f_{F}\left( x\right) -f_{F}\left( y\right) \right\vert \leq
K_{0}^{1+\frac{1}{p}}\left\vert x-y\right\vert ^{\frac{1}{p}}  \label{fHld2}
\end{equation}%
for any $x,y\in \mathbb{R}$, where $K_{0}=C_{p}\left\Vert \bar{w}%
^{-1}\right\Vert _{r}\left\Vert F\right\Vert _{2,s}(\left\Vert \bar{w}%
^{-1}\right\Vert _{r}\left\Vert DF\right\Vert _{s}^{2}+1)$, and $C_{p}$
depends only on $p$.
\end{proposition}

\begin{proof}
Note that $D\bar{w}=-D^{2}F\otimes _{1}DL^{-1}F-DF\otimes _{1}D^{2}L^{-1}F$.
Then, applying (\ref{NP1}) and (\ref{NP2}) we obtain%
\begin{equation}
\left\Vert D\bar{w}\right\Vert _{\frac{s}{2}}\leq (1+2^{-s})\left\Vert
\left\Vert D^{2}F\right\Vert _{op}\right\Vert _{s}\left\Vert DF\right\Vert
_{s}.  \label{Dw-}
\end{equation}%
From $\bar{u}=-\bar{w}^{-1}DL^{-1}F$ we get $D\bar{u}=-\bar{w}%
^{-1}D^{2}L^{-1}F+w^{-2}D\bar{w}\otimes DL^{-1}F$. Then, using (\ref{NP1})--(%
\ref{rci}) we have for $t>0$ satisfying $\frac{1}{p^{\prime }}=\frac{1}{r}+%
\frac{1}{t}$,
\begin{equation*}
\left\Vert \bar{u}\right\Vert _{p^{\prime }}\leq \left\Vert \bar{w}%
^{-1}\left\Vert DL^{-1}F\right\Vert _{\mathfrak{H}}\right\Vert _{p^{\prime
}}\leq \left\Vert \bar{w}^{-1}\right\Vert _{r}\left\Vert DF\right\Vert _{t},
\end{equation*}%
and
\begin{eqnarray*}
\left\Vert D\bar{u}\right\Vert _{p^{\prime }} &\leq &\left\Vert \bar{w}%
^{-1}\left\Vert D^{2}L^{-1}F\right\Vert _{\mathfrak{H}\otimes \mathfrak{H}%
}\right\Vert _{p^{\prime }}+\left\Vert \bar{w}^{-2}\left\Vert D\bar{w}%
\right\Vert _{\mathfrak{H}}\left\Vert DL^{-1}F\right\Vert _{\mathfrak{H}%
}\right\Vert _{p^{\prime }} \\
&\leq &\left\Vert \bar{w}^{-1}\right\Vert _{r}\left\Vert D^{2}F\right\Vert
_{t}+\left\Vert \bar{w}^{-2}\right\Vert _{r}\left\Vert D\bar{w}\right\Vert _{%
\frac{s}{2}}\left\Vert DF\right\Vert _{s},.
\end{eqnarray*}%
Noticing that $\left\Vert D^{2}F\right\Vert _{t}\leq \left\Vert
D^{2}F\right\Vert _{s}$ because $t<s$, and applying Meyer's inequality (\ref%
{Meyer}) with (\ref{Dw-}) and (\ref{rci}) we obtain
\begin{equation}
\left\Vert \delta \left( \bar{u}\right) \right\Vert _{p^{\prime }}\leq
C_{p}\left\Vert \bar{u}\right\Vert _{1,p^{\prime }}\leq K_{0}.
\label{deltaBd2}
\end{equation}%
Then $u\in \mathbb{D}^{1,p^{\prime }}$ and the density formula (\ref{Fmla3})
holds. As in the proof of Theorem \ref{density}, (\ref{fBd2}) and (\ref%
{fHld2}) follow from (\ref{deltaBd2}) and
\begin{equation*}
\left\vert E\left[ \mathbf{1}_{\left\{ F>x\right\} }\delta \left( \bar{u}%
\right) \right] \right\vert \leq P\left( \left\vert F\right\vert >\left\vert
x\right\vert \right) ^{\frac{1}{p}}\left\Vert \delta \left( \bar{u}\right)
\right\Vert _{^{p^{\prime }}}\leq \left( 1\wedge (\left\vert x\right\vert
^{-2}\left\Vert F\right\Vert _{2p}^{2})\right) \left\Vert \delta \left( \bar{%
u}\right) \right\Vert _{^{p^{\prime }}},
\end{equation*}%
\begin{equation*}
\left\vert f_{F}\left( x\right) -f_{F}\left( y\right) \right\vert \leq
\left( E[\mathbf{1}_{\{x<F\leq y\}}]\right) ^{\frac{1}{p}}\left\Vert \delta
\left( u\right) \right\Vert _{p^{\prime }}.
\end{equation*}
\end{proof}

\subsection{Derivatives of the density}

Next we present a formula for the derivatives of the density function, under
additional conditions. A sequence of recursively defined random variables
given by $G_{0}=1$ and $G_{k+1}=\delta (G_{k}u)$ where $u$ is an $\mathfrak{H%
}$-valued process, plays an essential role in the formula. The following
technical lemma gives an explicit formula for the sequence $G_{k}$, relating
it to Hermite polynomials. To simplify the notation, for an $\mathfrak{H}$%
--valued random variable $u$, we denote%
\begin{equation}
\delta _{u}=\delta (u),~D_{u}G=\left\langle DG,u\right\rangle _{\mathfrak{H}%
},~D_{u}^{k}G=\left\langle D\left( D_{u}^{k-1}G\right) ,u\right\rangle _{%
\mathfrak{H}}.  \label{notation}
\end{equation}

Recall $H_{k}(x)$ denotes the $k$th Hermite polynomial. For $\lambda >0$ and
$x\in \mathbb{R}$, we define the generalized $k$th Hermite polynomial as
\begin{equation}
H_{k}\left( \lambda ,x\right) =\lambda ^{\frac{k}{2}}H_{k}(\frac{x}{\sqrt{%
\lambda }}).  \label{H-lamda}
\end{equation}%
From the property $H_{k}^{\prime }(x)=kH_{k-1}(x)$ it follows by induction
that the $k$th Hermite polynomials has the form $H_{k}(x)=\sum_{0\leq i\leq
\lfloor k/2\rfloor }c_{k,i}x^{k-2i}$, where we denote by $\lfloor k/2\rfloor
$ the largest integer less than or equal to $k/2$. Then (\ref{H-lamda})
implies
\begin{equation}
H_{k}(\lambda ,x)=\sum_{0\leq i\leq \lfloor k/2\rfloor
}c_{k,i}x^{k-2i}\lambda ^{i}.  \label{gHermite}
\end{equation}

\begin{lemma}
\label{Gk}Fix an integer $m\geq 1$ and a number $p>m$. Suppose $u\in \mathbb{%
D}^{m,p}(\mathfrak{H})$. We define recursively a sequence $\left\{
G_{k}\right\} _{k=0}^{m}$ by $G_{0}=1$ and $G_{k+1}=\delta (G_{k}u)$. Then,
these variables are well-defined and for $k=1,2,\dots ,m$, $G_{k}\in \mathbb{%
D}^{m-k,\frac{p}{k}}$ and
\begin{equation}
G_{k}=H_{k}(D_{u}\delta _{u},\delta _{u})+T_{k}\text{,}  \label{GkFmla}
\end{equation}%
where we denote by $T_{k}$ the higher order derivative terms which can be
defined recursively as follows: $T_{1}=T_{2}=0$ and for $k\geq 2$,
\begin{equation}
T_{k+1}=\delta _{u}T_{k}-D_{u}T_{k}-\partial _{\lambda }H_{k}(D_{u}\delta
_{u},\delta _{u})D_{u}^{2}\delta _{u}.  \label{tk}
\end{equation}
\end{lemma}

The following lemma is proved in the Appendix.

\begin{lemma}
\label{rem1} From $(\ref{tk})$ we can deduce that for $k\geq 3$
\begin{equation}
T_{k}=\sum_{(i_{0},\dots ,i_{k-1})\in J_{k}}a_{i_{0},i_{1},\dots
,i_{k-1}}\delta _{u}^{i_{0}}\left( D_{u}\delta _{u}\right) ^{i_{1}}\left(
D_{u}^{2}\delta _{u}\right) ^{i_{2}}\cdots \left( D_{u}^{k-1}\delta
_{u}\right) ^{i_{k-1}},  \label{HODT}
\end{equation}%
where the coefficients $a_{i_{0},i_{1},\dots ,i_{k-1}}$ are real numbers and
$J_{k}$ is the set of multi-indices $(i_{0},i_{1},\dots ,i_{k-1})\in \mathbb{%
N}^{k}$ satisfying the following three conditions
\begin{equation*}
(a)~i_{0}+\sum_{j=1}^{k-1}ji_{j}\leq k-1;~(b)~i_{2}+\cdots +i_{k-1}\geq
1;~(c)~\sum_{j=1}^{k-1}i_{j}\leq \lfloor \frac{k-1}{2}\rfloor .
\end{equation*}%
From $(b)$ we see that every term in $T_{k}$ contains at least one factor of
the form $D_{u}^{j}\delta _{u}$ with some $j\geq 2$. We shall show this type
of factors will converge to zero. For this reason we call these terms high
order terms.
\end{lemma}

\begin{proof}[Proof of Lemma \protect\ref{Gk}]
First, we prove by induction on $k$ that the above sequence $G_{k}$ is
well-defined and $G_{k}\in \mathbb{D}^{m-k,\frac{p}{k}}$. Suppose first that
$k=1$. Then, Meyer's inequality implies that $G_{1}=\delta _{u}\in \mathbb{D}%
^{m-1,p}$. Assume now that for $k\leq m-1$, $G_{k}\in \mathbb{D}^{m-k,\frac{p%
}{k}}$. Then it follows from Meyer's and H\"{o}lder's inequalities (see \cite%
[Proposition 1.5.6]{Nu06}) that
\begin{equation*}
\left\Vert G_{k+1}\right\Vert _{m-k-1,\frac{p}{k+1}}\leq C_{m,p}\left\Vert
G_{k}u\right\Vert _{m-k,\frac{p}{k+1}}\leq C_{m,p}^{\prime }\left\Vert
G_{k}\right\Vert _{m-k,\frac{p}{k}}\left\Vert u\right\Vert _{m-k,p}<\infty .
\end{equation*}%
Let us now show, by induction, the decomposition (\ref{GkFmla}). When $k=1$ (%
\ref{GkFmla}) is true because $G_{1}=\delta _{u}$ and $T_{1}=0$. Assume now (%
\ref{GkFmla}) holds for $k\leq m-1$. Noticing that $\partial
_{x}H_{k}(\lambda ,x)=kH_{k-1}(\lambda ,x)$ (since $H_{k}^{\prime
}(x)=kH_{k-1}(x)$), we get%
\begin{equation*}
D_{u}H_{k}(D_{u}\delta _{u},\delta _{u})=kH_{k-1}(D_{u}\delta _{u},\delta
_{u})D_{u}\delta _{u}+\partial _{\lambda }H_{k}(D_{u}\delta _{u},\delta
_{u})D_{u}^{2}\delta _{u}.
\end{equation*}%
Hence, applying the operator $D_{u}$ to both sides of (\ref{GkFmla}),
\begin{equation*}
D_{u}G_{k}=kH_{k-1}(D_{u}\delta _{u},\delta _{u})D_{u}\delta _{u}+\widetilde{%
T}_{k+1},
\end{equation*}%
where
\begin{equation}
\widetilde{T}_{k+1}=D_{u}T_{k}+\partial _{\lambda }H_{k}(D_{u}\delta
_{u},\delta _{u})D_{u}^{2}\delta _{u}.  \label{Tk}
\end{equation}%
From the definition of $G_{k+1}$ and using (\ref{factorOut}) we obtain
\begin{eqnarray*}
G_{k+1} &=&\delta (uG_{k})=G_{k}\delta _{u}-D_{u}G_{k} \\
&=&\delta _{u}H_{k}(D_{u}\delta _{u},\delta _{u})+\delta
_{u}T_{k}-kH_{k-1}(D_{u}\delta _{u},\delta _{u})D_{u}\delta _{u}-\widetilde{T%
}_{k+1}\text{.}
\end{eqnarray*}%
Note that $H_{k+1}(x)=xH_{k}(x)-kH_{k-1}(x)$ implies $xH_{k}(\lambda
,x)-k\lambda H_{k-1}(\lambda ,x)=H_{k+1}(\lambda ,x)$. Hence,
\begin{equation*}
G_{k+1}=H_{k+1}(D_{u}\delta _{u},\delta _{u})+\delta _{u}T_{k}-\widetilde{T}%
_{k+1}\text{.}
\end{equation*}%
The term $T_{k+1}=\delta _{u}T_{k}-\widetilde{T}_{k+1}$ has the form given
in (\ref{tk}). This completes the proof.
\end{proof}

Now we are ready to present some formulae for the derivatives of the density
function under certain sufficient conditions on the random variable $F$. For
a random variable $F$ in $\mathbb{D}^{1,2}$ and for any $\beta \geq 1$ we
are going to use the notation
\begin{equation}
M_{\beta }(F)=\left( E[\left\Vert DF\right\Vert _{\mathfrak{H}}^{-\beta
}]\right) ^{\frac{1}{\beta }}.  \label{Mbeta}
\end{equation}

\begin{proposition}
\label{kthD}Fix an integer $m\geq 1$. Let $F$ be a random variable in $%
\mathbb{D}^{m+2,\infty }$ such that $M_{\beta }(F)<\infty $ for some $\beta
>3m+3(\lfloor \frac{m}{2}\rfloor \vee 1)$. Denote $w=\left\Vert
DF\right\Vert _{\mathfrak{H}}^{2}$ and $u=\frac{DF}{w}$. Then, $u\in \mathbb{%
D}^{m+1,p}(\mathfrak{H})$ for some $p>1$, and the random variables $\left\{
G_{k}\right\} _{k=0}^{m+1}$ introduced in Lemma $\ref{Gk}$ are well-defined.
Under these assumptions, $F$ has a density $f$ of class $C^{m}$ with
derivatives given by
\begin{equation}
f_{F}^{(k)}(x)=(-1)^{k}E[\mathbf{1}_{\{F>x\}}G_{k+1}]  \label{kthFml}
\end{equation}%
for $k=1,\dots ,m$.
\end{proposition}

\begin{proof}
It is enough to show that $\left\{ G_{k}\right\} _{k=0}^{m+1}$ are
well-defined, since it follows from \cite[Exercise 2.1.4]{Nu06} that the $k$%
th derivative of the density of $F$ is given by (\ref{kthFml}). To do this
we will show that $G_{k}$ defined in (\ref{GkFmla}) are in $L^{1}(\Omega )$
for all $k=1,\dots ,m+1$. From (\ref{GkFmla}) we can write
\begin{equation*}
E[\left\vert G_{k}\right\vert ]\leq E[\left\vert H_{k}(D_{u}\delta
_{u},\delta _{u})\right\vert ]+E[\left\vert T_{k}\right\vert ].
\end{equation*}%
Recall the explicit expression of $H_{k}(\lambda ,x)$ in (\ref{gHermite}).
Since $\beta >3(m+1)$, we can choose $r_{0}<\frac{\beta }{3},r_{1}<\frac{%
\beta }{6}$ such that
\begin{equation*}
1\geq \frac{k-2i}{r_{0}}+\frac{i}{r_{1}}>\frac{3\left( k-2i\right) }{\beta }+%
\frac{6i}{\beta }=\frac{3k}{\beta },
\end{equation*}%
for any $0\leq i\leq \lfloor k/2\rfloor $ and $1\le k\leq m+1$. Then,
applying H\"{o}lder's inequality with (\ref{gHermite}), (\ref{Dudelta-u0})
and (\ref{Dudelta-uka}) we have
\begin{equation*}
E[\left\vert H_{k}(D_{u}\delta _{u},\delta _{u})\right\vert ]\leq
C_{k}\sum_{0\leq i\leq \lfloor k/2\rfloor }\left\Vert \delta _{u}\right\Vert
_{r_{0}}^{k-2i}\left\Vert D_{u}\delta _{u}\right\Vert _{r_{1}}^{i}<\infty .
\end{equation*}

To prove that $E\left[ \left\vert T_{k}\right\vert \right] <\infty $,
applying H\"{o}lder's inequality to the expression (\ref{HODT}) and choosing
$r_{j}>0$ for $0\leq j\leq k-1$ such that
\begin{equation*}
1\geq \frac{i_{0}}{r_{0}}+\sum_{j=1}^{k-1}\frac{i_{j}}{r_{j}}>\frac{3i_{0}}{%
\beta }+\sum_{j=1}^{k-1}\frac{(3j+3)i_{j}}{\beta },
\end{equation*}%
we obtain that, (assuming $k\geq 3$, otherwise $T_{k}=0$)
\begin{equation*}
E\left[ \left\vert T_{k}\right\vert \right] \leq C\sum_{(i_{0},\dots
,i_{k})\in J_{k}}\left\Vert \delta _{u}\right\Vert
_{r_{0}}^{i_{0}}\prod_{j=1}^{k-1}\left\Vert D_{u}^{j}\delta _{u}\right\Vert
_{r_{j}}^{i_{j}}.
\end{equation*}%
Due to (\ref{Dudelta-u0}) and (\ref{Dudelta-uka}), this expression is
finite, provided $r_{j}<\frac{\beta }{3j+3}$ for $0\leq j\leq k-1$. We can
choose $(r_{j},0\leq j\leq k-1)$ satisfying the above conditions because $%
\beta >3(k-1)+3\lfloor \frac{k-1}{2}\rfloor $ for all $1\le k \le m+1$, and
from properties (a) and (c) of $J_{k}$ in Lemma \ref{rem1} we have
\begin{equation*}
\frac{3i_{0}}{\beta }+\sum_{j=1}^{k-1}\frac{(3j+3)i_{j}}{\beta }\leq \frac{%
3(k-1)+3\lfloor \frac{k-1}{2}\rfloor }{\beta }.
\end{equation*}%
This completes the proof.
\end{proof}

\begin{example}
Consider a random variable in the first Wiener chaos $N=I_{1}(h)$, where $%
\in \mathfrak{H}$ with $\Vert h\Vert _{\mathfrak{H}}=\sigma $. Then $N$ has
the normal distribution $N\sim N(0,\sigma ^{2})$ with density denoted by $%
\phi (x)$. Clearly $\left\Vert DN\right\Vert _{\mathfrak{H}}=\sigma $, $u=%
\frac{h}{\sigma ^{2}}$, $\delta _{u}=\frac{N}{\sigma ^{2}}$ and $D_{u}\delta
_{u}=\frac{h}{\sigma ^{2}}$. Then $G_{k}=H_{k}(\frac{1}{\sigma ^{2}},\frac{N%
}{\sigma ^{2}})$ and from (\ref{kthFml}) we obtain the formula
\begin{equation}
\phi ^{(k)}(x)=(-1)^{k}E\left[ \mathbf{1}_{\{N>x\}}H_{k+1}\left( \frac{1}{%
\sigma ^{2}},\frac{N}{\sigma ^{2}}\right) \right] ,  \label{kthN}
\end{equation}%
which can also be obtained by analytic arguments.
\end{example}

\begin{remark}
Let $g$ be a function differentiable of order $m$, and denote $%
u_{j}(g,x)=\sup_{|t|\leq x}|g^{(j)}(t)-g^{(j)}(0)|$. Let $F$ be a random
variable with $E(F)=0$, $E(F^{2})=1$ and $E\left\{
|F|^{m+1}u_{m}(g,F)\right\} <\infty $. It is proved in {\rm \cite{barbour}} that
the following expansion holds:
\begin{equation*}
E(Fg(F))-E(g^{\prime }(F))=\sum_{j=2}^{m}\frac{\gamma _{j+1}}{j!}%
Eg^{(j)}(F)+R\,,
\end{equation*}%
where $\gamma _{j}$ is the $j$-th cumulant of $X$ and $|R|\leq CE\left\{
|F|^{m+1}u_{m}(g,F)\right\} $ for some constant $C>0$. For any function $h$,
let $f$ be the solution of the Stein's equation $(\ref{Stein's equ})$ given by $(\ref{solu})$. Then
\begin{eqnarray*}
E[h(F)]-E[h(N)] &=&E[f^{\prime }(F)]-E[Ff(F)] \\
&=&-\sum_{j=2}^{m}\frac{\gamma _{j+1}}{j!}E[f^{(j)}(F)]-R\,.
\end{eqnarray*}%
This is the so-called Edgeworth expansion (see also \cite{rinott} and
references therein). Equation $(\ref{kthFml})$ can also be used to compute $%
E[f^{(j)}(F)]$. We have easily
\begin{equation}
E[h(F)]-E[h(N)]=-\sum_{j=2}^{m}\frac{\gamma _{j+1}}{j!}E\left[ f(F)G_{k}%
\right] -R\,,
\end{equation}%
where $G_{k}$ is given in Lemma $\ref{Gk}$. Thus, it is possible to use
Malliavin calculus to obtain the full Edgeworth expansion without assuming
the differentiability of $f$. However, we shall not pursue this aspect in
the present work.
\end{remark}

\begin{remark}
\label{rem_Privault} The recursive algorithms used in Lemma $\ref{Gk}$ have
some similarities with the recursive formula developed by Privault in {\rm \cite
{Pr13}} to compute $E(F[\delta (u)]^{n})$.
\end{remark}

\section{Random variables in the $q$th Wiener chaos\label{ChaosConv}}

In this section we establish our main results on uniform estimates and
uniform convergence of densities and their derivatives. We shall deal first
with the convergence of densities and later we consider their derivatives.

\subsection{Uniform estimates of densities}

Let $F=I_{q}(f)$ for some $f\in \mathfrak{H}^{\odot q}$ and $q\ge 2$. To
simplify the notation, along this section we denote%
\begin{equation*}
w=\left\Vert DF\right\Vert _{\mathfrak{H}}^{2},~u=w^{-1}DF.
\end{equation*}%
Note that $LF=-qF$ and using (\ref{factorOut}) and (\ref{DeltaDL}) we can
write
\begin{equation}
\delta _{u}=\delta (u)=qFw^{-1}-\left\langle Dw^{-1},DF\right\rangle _{%
\mathfrak{H}}.  \label{delta-u}
\end{equation}

\begin{theorem}
\label{qRateThm}Let $F=I_{q}(f)$, $q\geq 2$, for some $f\in \mathfrak{H}%
^{\odot q}$ be a random variable in the $q$th Wiener chaos with $%
E[F^{2}]=\sigma ^{2}$. Assume that $M_{6}(F)<\infty $, where $M_{6}(F)$ is
defined in $(\ref{Mbeta})$. Let $\phi (x)$ be the density of $N\sim
N(0,\sigma ^{2})$. Then $F$ has a density $f_{F}(x)$ given by $(\ref{Fmla1})$%
. Furthermore,
\begin{equation}
\sup_{x\in \mathbb{R}}\left\vert f_{F}(x)-\phi (x)\right\vert \leq C\sqrt{%
E[F^{4}]-3\sigma ^{4}},  \label{qRate4M}
\end{equation}%
where the constant $C$ has the form $C=C_{q}\left( \sigma
^{-1}M_{6}(F)^{2}+M_{6}(F)^{3}+\sigma ^{-3}\right) $ and $C_{q}$ depends
only on $q$.
\end{theorem}

We begin with a lemma giving an estimate for the contraction $D^{k}F\otimes
_{1}D^{l}F$ with $k+l\geq 3$.

\begin{lemma}
Let $F=I_{q}(f)$ be a random variable in the $q$th Wiener chaos with $%
E[F^{2}]=\sigma ^{2}$. Then for any integers $k\geq l\geq 1$ satisfying $%
k+l\geq 3$, there exists a constant $C_{k,l,q}$ depending only on $k,l,q$
such that
\begin{equation}
\left\Vert D^{k}F\otimes _{1}D^{l}F\right\Vert _{2}\leq C_{k,l,q}\left\Vert
q\sigma ^{2}-\left\Vert DF\right\Vert _{\mathfrak{H}}^{2}\right\Vert _{2}.
\label{contrtnBd}
\end{equation}
\end{lemma}

\begin{proof}
Note that $D^{k}F=q(q-1)\cdots (q-k+1)I_{q-k}(f)$. Applying (\ref{MltFml}),
we get%
\begin{eqnarray*}
D^{k}F\otimes _{1}D^{l}F &=&q^{2}(q-1)^{2}\cdots \cdots
(q-l+1)^{2}(q-l)\cdots (q-k+1) \\
&&\times \sum_{r=0}^{q-k}r!\binom{q-k}{r}\binom{q-l}{r}I_{2q-k-l-2r}(f%
\widetilde{\otimes }_{r+1}f).
\end{eqnarray*}%
Taking into account the orthogonality of multiple integrals of different
orders, we obtain
\begin{eqnarray}
&&E[\left\Vert D^{k}F\otimes _{1}D^{l}F\right\Vert _{\mathfrak{H}^{\otimes
\left( k+l-2\right) }}^{2}]=\frac{\left( q!\right) ^{4}}{\left( q-l\right)
!^{2}\left( q-k\right) !^{2}}  \notag \\
&&~~~~\times \sum_{r=0}^{q-k}r!^{2}\binom{q-k}{r}^{2}\binom{q-l}{r}%
^{2}\left( 2q-k-l-2r\right) !\left\Vert f\widetilde{\otimes }%
_{r+1}f\right\Vert _{\mathfrak{H}^{\otimes 2q-2-2r}}^{2}.
\label{contractionE}
\end{eqnarray}%
Applying (\ref{contractionE}) with $k=l=1$, we obtain
\begin{eqnarray}
&&E[\left\Vert DF\right\Vert _{\mathfrak{H}}^{4}]=E[\left\vert DF\otimes
_{1}DF\right\vert ^{2}]  \label{w1} \\
&=&q^{4}\sum_{r=0}^{q-1}r!^{2}\binom{q-1}{r}^{4}\left( 2q-2-2r\right)
!\left\Vert f\widetilde{\otimes }_{r+1}f\right\Vert _{\mathfrak{H}^{\otimes
2q-2-2r}}^{2}  \notag \\
&=&q^{4}\sum_{r=0}^{q-2}r!^{2}\binom{q-1}{r}^{4}\left( 2q-2-2r\right)
!\left\Vert f\widetilde{\otimes }_{r+1}f\right\Vert _{\mathfrak{H}^{\otimes
2q-2-2r}}^{2}+q^{2}q!^{2}\left\Vert f\right\Vert _{\mathfrak{H}^{\otimes
q}}^{4}.  \notag
\end{eqnarray}%
Taking into account that $\sigma ^{2}=E[F^{2}]=q!\left\Vert f\right\Vert _{%
\mathfrak{H}^{\otimes q}}^{2}$, we obtain that for any $k+l\geq 3$, there
exists a constant $C_{k,l,q}$ such that
\begin{equation*}
E[\left\Vert D^{k}F\otimes _{1}D^{l}F\right\Vert _{\mathfrak{H}^{\otimes
\left( k+l-2\right) }}^{2}]\leq C_{k,l,q}^{2}E[\left\Vert DF\right\Vert _{%
\mathfrak{H}}^{4}-q^{2}\sigma ^{4}].
\end{equation*}%
Meanwhile, it follows from $E[\left\Vert DF\right\Vert _{\mathfrak{H}%
}^{2}]=q\left\Vert f\right\Vert _{\mathfrak{H}^{\otimes q}}^{2}=q\sigma ^{2}$
that
\begin{equation}
E[\left\Vert DF\right\Vert _{\mathfrak{H}}^{4}-q^{2}\sigma
^{4}]=E[\left\Vert DF\right\Vert _{\mathfrak{H}}^{4}-2q\sigma ^{2}\left\Vert
DF\right\Vert _{\mathfrak{H}}^{2}+q^{2}\sigma ^{4}]=E[(\left\Vert
DF\right\Vert _{\mathfrak{H}}^{2}-q\sigma ^{2})^{2}].  \label{w2}
\end{equation}%
Combining (\ref{contractionE}), (\ref{w1}) and (\ref{w2}) we have
\begin{equation*}
E[\left\Vert D^{k}F\otimes _{1}D^{l}F\right\Vert _{\mathfrak{H}^{\otimes
\left( k+l-2\right) }}^{2}]\leq C_{k,l,q}^{2}E[(\left\Vert DF\right\Vert _{%
\mathfrak{H}}^{2}-q\sigma ^{2})^{2}],
\end{equation*}%
which completes the proof.
\end{proof}

\begin{proof}[\textbf{Proof of Theorem} $\protect\ref{qRateThm}$]
It follows from Theorem \ref{density} that $F$ admits a density $f_{F}(x)=E%
\left[ \mathbf{1}_{\{F>x\}}\delta \left( u\right) \right] $. By (\ref{kthN})
with $k=1$ we can write $\phi (x)=\frac{1}{\sigma ^{2}}E[\mathbf{1}%
_{\{N>x\}}N]$. Then, using (\ref{delta-u}), for all $x\in \mathbb{R}$ we
obtain
\begin{eqnarray}
&&f_{F}(x)-\phi (x)=E\left[ \mathbf{1}_{\{F>x\}}\delta \left( u\right) %
\right] -\sigma ^{-2}E[\mathbf{1}_{\{N>x\}}N]  \notag \\
&=&E[\mathbf{1}_{\{F>x\}}(F(\frac{q}{w}-\sigma ^{-2})-\left\langle
Dw^{-1},DF\right\rangle _{\mathfrak{H}})]+\sigma ^{-2}E\left[ F\mathbf{1}%
_{\{F>x\}}-N\mathbf{1}_{\{N>x\}}\right]  \notag \\
&=&A_{1}+A_{2}.  \label{qRate0}
\end{eqnarray}%
For the first term $A_{1}$, H\"{o}lder's inequality implies
\begin{eqnarray*}
\left\vert A_{1}\right\vert &=&\left\vert E[\mathbf{1}_{\{F>x\}}(F(\frac{q}{w%
}-\sigma ^{-2})-\left\langle Dw^{-1},DF\right\rangle _{\mathfrak{H}%
})]\right\vert \\
&\leq &\sigma ^{-2}E\left[\left\vert Fw^{-1}(w-q\sigma ^{2})\right\vert %
\right]+2E[w^{-\frac{3}{2}}\left\Vert D^{2}F\otimes _{1}DF\right\Vert _{%
\mathfrak{H}}] \\
&\leq &\sigma ^{-2}\left\Vert w^{-1}\right\Vert _{3}\left\Vert F\right\Vert
_{3}\left\Vert w-q\sigma ^{2}\right\Vert _{3}+2\left\Vert w^{-\frac{3}{2}%
}\right\Vert _{2}\left\Vert \left\Vert D^{2}F\otimes _{1}DF\right\Vert _{%
\mathfrak{H}}\right\Vert _{2}.
\end{eqnarray*}
Note that (\ref{HyperMeyer}) implies
\begin{equation*}
\left\Vert w-q\sigma ^{2}\right\Vert _{3}\leq C\left\Vert w-q\sigma
^{2}\right\Vert _{2}
\end{equation*}%
and $\left\Vert F\right\Vert _{3}\leq C\left\Vert F\right\Vert _{2}=C\sigma $%
. Combining these estimates with (\ref{contrtnBd}) we obtain%
\begin{equation}
\left\vert A_{1}\right\vert \leq C(\sigma ^{-1}\left\Vert w^{-1}\right\Vert
_{3}+\left\Vert w^{-1}\right\Vert _{3}^{\frac{3}{2}})\left\Vert w-q\sigma
^{2}\right\Vert _{2}.  \label{qRate1}
\end{equation}

For the second term $A_{2}$, applying Lemma \ref{MS-ctrl} to the function $%
h(z)=z\mathbf{1}_{\{z>x\}}$, which satisfies $\left\vert h(z)\right\vert
\leq \left\vert z\right\vert $, we have%
\begin{eqnarray}
\left\vert A_{2}\right\vert &=&\sigma ^{-2}\left\vert E\left[ F\mathbf{1}%
_{\{F>x\}}-N\mathbf{1}_{\{N>x\}}\right] \right\vert  \notag \\
&\leq &C\sigma ^{-3}\left\Vert \sigma ^{2}-\left\langle
DF,-DL^{-1}F\right\rangle _{\mathfrak{H}}\right\Vert _{2}\leq C\sigma
^{-3}\left\Vert q\sigma ^{2}-w\right\Vert _{2}.  \label{qRate2}
\end{eqnarray}%
Combining (\ref{qRate0}) with (\ref{qRate1})--(\ref{qRate2}) we obtain
\begin{equation}
\sup_{x\in \mathbb{R}}\left\vert f_{F}(x)-\phi (x)\right\vert \leq C(\sigma
^{-1}\left\Vert w^{-1}\right\Vert _{3}+\left\Vert w^{-1}\right\Vert _{3}^{%
\frac{3}{2}}+\sigma ^{-3})\left\Vert w-q\sigma ^{2}\right\Vert _{2}.  \notag
\end{equation}%
Then (\ref{qRate4M}) follows from (\ref{equi}). This completes the proof.
\end{proof}

Using the estimates shown in Theorem \ref{qRateThm} we can deduce the
following uniform convergence and convergence in $L^p$ of densities for a
sequence of random variables in a fixed $q$th Wiener chaos.

\begin{corollary}
\label{qConv}Let $\left\{ F_{n}\right\} _{n\in \mathbb{N}}$ be a sequence of
random variables in the $q$th Wiener chaos with $q\geq 2$. Set $\sigma
_{n}^{2}=E[F_{n}^{2}]$ and assume that $\lim_{n\rightarrow \infty }\sigma
_{n}^{2}=\sigma ^{2}$, $0<\delta \le \sigma_n^2 \le K$ for all $n$, $%
\lim_{n\rightarrow \infty }E[F_{n}^{4}]=3\sigma ^{4}$ and
\begin{equation}
M:=\sup_{n}\left( \mathbb{E}[\left\Vert DF_{n}\right\Vert _{\mathfrak{H}%
}^{-6}]\right) ^{1/6}<\infty .  \label{qNegM}
\end{equation}%
Let $\phi (x)$ be the density of the law $N(0,\sigma ^{2})$. Then, each $%
F_{n}$ admits a density $f_{F_{n}}\in C(\mathbb{R})$ and there exists a
constant $C $ depending only on $q,\sigma ,\delta $ and $M$ such that
\begin{equation}
\sup_{x\in \mathbb{R}}\left\vert f_{F_{n}}(x)-\phi (x)\right\vert \leq C
\left( \left\vert E[F_{n}^{4}]-3\sigma _{n}^{4}\right\vert ^{\frac{1}{2}%
}+|\sigma _{n}-\sigma |\right) .  \label{qRateF4}
\end{equation}%
Furthermore, for any $p\geq 1$ and $\alpha \in (\frac{1}{2},p)$,
\begin{equation}
\left\Vert f_{F_{n}}-\phi \right\Vert _{L^{p}(\mathbb{R})}\leq C \left(
\left\vert E[F_{n}^{4}]-3\sigma _{n}^{4}\right\vert ^{\frac{1}{2}%
}+\left\vert \sigma _{n}-\sigma \right\vert \right) ^{\frac{p-\alpha }{p}},
\label{f-Lp}
\end{equation}%
where $C $ is a constant depending on $q,\sigma ,M,p ,\alpha $ and $K$.
\end{corollary}

\begin{proof}
Let $\phi _{n}(x)$ be the density of $N(0,\sigma _{n}^{2})$. Then Theorem %
\ref{qRateThm} implies that
\begin{equation*}
\sup_{x\in \mathbb{R}}\left\vert f_{F_{n}}(x)-\phi _{n}(x)\right\vert \leq C
\left\vert E[F_{n}^{4}]-3\sigma _{n}^{4}\right\vert ^{\frac{1}{2}}.
\end{equation*}%
On the other hand, if $N_{n}\sim N(0,\sigma _{n}^{2})$, it is easy to see
that
\begin{equation*}
\sup_{x\in \mathbb{R}}\left\vert \phi _{n}(x)-\phi (x)\right\vert \leq C
\left\vert \sigma _{n}-\sigma \right\vert .
\end{equation*}%
Then (\ref{qRateF4}) follows from triangle inequality. To show (\ref{f-Lp}),
first notice that (\ref{fBd}) implies
\begin{equation*}
f_{F_{n}}(x)\leq C (1\wedge \left\vert x\right\vert ^{-2}).
\end{equation*}%
Therefore, if $\alpha >\frac{1}{2}$ the function $(f_{F_{n}}(x)+\phi
(x))^{\alpha }$ is integrable. Then, (\ref{f-Lp}) follows from (\ref{qRateF4}%
) and the inequality
\begin{equation*}
\left\vert f_{F_{n}}(x)-\phi (x)\right\vert ^{p}\leq \left\vert
f_{F_{n}}(x)-\phi (x)\right\vert ^{p-\alpha }(f_{F_{n}}(x)+\phi (x))^{\alpha
}.
\end{equation*}
\end{proof}

\subsection{Uniform estimation of derivatives of densities}

In this subsection, we establish the uniform convergence for derivatives of
densities of random variables to a normal distribution. We begin with the
following theorem which estimates the uniform distance between the
derivatives of the densities of a random variable $F$ in the $q$th Wiener
chaos and the normal law $N(0,E[F^{2}])$.

\begin{theorem}
\label{qDeriRate} Let $m\geq 1$ be an integer. Let $F$ be a random variable
in the $q$th Wiener chaos, $q\ge 2$, with $E[F^{2}]=\sigma ^{2}$ and $%
M_{\beta }:=M_{\beta }(F)<\infty $ for some $\beta >6m+6(\lfloor \frac{m }{2}%
\rfloor \vee 1)$ (Recall the definition of $M_{\beta }(F)$ in $(\ref{Mbeta})$%
). Let $\phi (x)$ be the density of $N\sim N(0,\sigma ^{2})$. Then $F$ has a
density $f_{F}(x)\in C^{m }(\mathbb{R})$ with derivatives given by $(\ref%
{kthFml})$. Moreover, for any $k=1,\dots ,m $
\begin{equation*}
\sup_{x\in \mathbb{R}}\left\vert f_{F}^{(k)}(x)-\phi ^{(k)}(x)\right\vert
\leq \sigma ^{-k-3}C \sqrt{ E[F^{4}]-3\sigma ^{2}},
\end{equation*}
where the constant $C $ depends on $q$, $\beta $, $m$, $\sigma $ and $%
M_{\beta }$ with polynomial growth in $\sigma $ and $M_{\beta }$.
\end{theorem}

To prove Theorem \ref{qDeriRate}, we need some technical results. Recall the
notation we introduced in (\ref{notation}), where we denote $\delta
_{u}=\delta (u)$, $D_{u}\delta _{u}=\left\langle D\delta _{u},u\right\rangle
_{\mathfrak{H}}$.

\begin{lemma}
\label{Dw} Let $F$ be a random variable in the $q$th Wiener chaos with $%
E[F^{2}]=\sigma ^{2}$. Let $w=\left\Vert DF\right\Vert _{\mathfrak{H}}^{2}$
and$~u=w^{-1}DF$.

\begin{itemize}
\item[(i)] If $M_{\beta }(F)<\infty $ for some $\beta >6$, then for any $%
1<r\leq \frac{2\beta }{\beta +6}$
\begin{equation}
\left\Vert \delta _{u}-\sigma ^{-2}F\right\Vert _{r}\leq C\sigma
^{-1}(M_{\beta }^{3}\vee 1)\left\Vert q\sigma ^{2}-w\right\Vert _{2}.
\label{dl-u1}
\end{equation}

\item[(ii)] If $M_{\beta }(F)<\infty $ for some $\beta >12$, then for any $%
1<r<\frac{2\beta }{\beta +12}$
\begin{equation}
\left\Vert D_{u}\delta _{u}-\sigma ^{-2}\right\Vert _{r}\leq C \sigma
^{-2}(M_{\beta }^{6}\vee 1)\left\Vert q\sigma ^{2}-w\right\Vert _{2},
\label{dudl-u1}
\end{equation}
where the constant $C$ depends on $\sigma$.
\end{itemize}
\end{lemma}

\begin{proof}
Recall that $\delta _{u}=qFw^{-1}-D_{DF}w^{-1}$. Using H\"{o}lder's
inequality and (\ref{DDFw-a}) we can write
\begin{eqnarray*}
\left\Vert \delta _{u}-\sigma ^{-2}F\right\Vert _{r} &\leq &\left\Vert
\sigma ^{-2}Fw^{-1}(q\sigma ^{2}-w)\right\Vert _{r}+\left\Vert
D_{DF}w^{-1}\right\Vert _{r} \\
&\leq &C\left( \sigma ^{-2}\Vert Fw^{-1}\Vert _{s}+(M_{\beta }^{3}\vee
1)\right) \left\Vert q\sigma ^{2}-w\right\Vert _{2},
\end{eqnarray*}%
provided $\frac{1}{r}=\frac{1}{s}+\frac{1}{2}$. By the hypercontractivity
property (\ref{Hyper}) $\Vert F\Vert _{\gamma }\leq C_{q,\gamma }\Vert
F\Vert _{2}$ for any $\gamma \geq 2$. Thus, by H\"{o}lder's inequality, if $%
\frac{1}{s}=\frac{1}{\gamma }+\frac{1}{p}$
\begin{equation*}
\Vert Fw^{-1}\Vert _{s}\leq \Vert F\Vert _{\gamma }\Vert w^{-1}\Vert
_{p}\leq C_{q,\gamma }\sigma M_{2p}^{2}.
\end{equation*}%
Choosing $p$ such that $2p<\beta $ we get (\ref{dl-u1}).

We can compute $D_{u}\delta _{u}$ as%
\begin{equation*}
D_{u}\delta
_{u}=qw^{-1}+qFw^{-1}D_{DF}w^{-1}-w^{-1}D_{DF}^{2}w^{-1}-w^{-1}\left\langle
D^{2}F,DF\otimes Dw^{-1}\right\rangle _{\mathfrak{H}}.
\end{equation*}%
Applying H\"{o}lder's inequality we obtain
\begin{eqnarray*}
&&\left\Vert D_{u}\delta _{u}-\sigma ^{-2}\right\Vert _{r}\leq \left\Vert
w^{-1}\left[ \sigma ^{-2}\left( q\sigma ^{2}-w\right)
+qFD_{DF}w^{-1}-D_{DF}^{2}w^{-1}\right] \right\Vert _{r} \\
&\leq &\sigma ^{-2}\Vert w^{-1}\Vert _{\frac{2r}{2-r}}\left\Vert q\sigma
^{2}-w\right\Vert _{2}+C_{\sigma }\Vert w^{-1}\Vert _{p}\left( \left\Vert
D_{DF}w^{-1}\right\Vert _{s}+\left\Vert D_{DF}^{2}w^{-1}\right\Vert
_{s}\right) ,
\end{eqnarray*}%
if $\frac{1}{r}>\frac{1}{p}+\frac{1}{s}$. Then, using (\ref{DDFw-a}) and (%
\ref{DDFw-b}) with $k=2$ and assuming that $s<\frac{2\beta }{\beta +8}$ and
that $2p<\beta $ we obtain (\ref{dudl-u1}).
\end{proof}

\begin{proof}[\textbf{Proof of Theorem }$\protect\ref{qDeriRate}$]
Proposition \ref{kthD} implies that $f_{F}(x)\in C^{m-1}(\mathbb{R})$ and
for $k=0,1,\dots ,m-1$,%
\begin{equation*}
f_{F}^{(k)}(x)=(-1)^{k}E[\mathbf{1}_{\left\{ F>x\right\} }G_{k+1}],
\end{equation*}%
where $G_{0}=1$ and $G_{k+1}=\delta (G_{k}u)=G_{k}\delta (u)-\left\langle
DG_{k},u\right\rangle _{\mathfrak{H}}$. From (\ref{kthN}),
\begin{equation*}
\phi ^{(k)}(x)=(-1)^{k}E[\mathbf{1}_{\left\{ N>x\right\} }H_{k+1}(\sigma
^{-2},\sigma ^{-2}N)].
\end{equation*}%
Then, the identity $G_{k+1}=H_{k+1}(D_{u}\delta _{u},\delta _{u})+T_{k+1}$
(see formula (\ref{GkFmla})), suggests the following triangle inequality%
\begin{eqnarray*}
\left\vert f_{F}^{(k)}(x)-\phi ^{(k)}(x)\right\vert &=&\left\vert E[\mathbf{1%
}_{\left\{ F>x\right\} }G_{k+1}-\mathbf{1}_{\left\{ N>x\right\}
}H_{k+1}(\sigma ^{-2},\sigma ^{-2}N)]\right\vert \\
&\leq &\left\vert E[\mathbf{1}_{\left\{ F>x\right\} }G_{k+1}-\mathbf{1}%
_{\left\{ F>x\right\} }H_{k+1}(\sigma ^{-2},\sigma ^{-2}F)]\right\vert \\
&&+\left\vert E[\mathbf{1}_{\left\{ F>x\right\} }H_{k+1}(\sigma ^{-2},\sigma
^{-2}F)-\mathbf{1}_{\left\{ N>x\right\} }H_{k+1}(\sigma ^{-2},\sigma
^{-2}N)]\right\vert \\
&=&A_{1}+A_{2}\,.
\end{eqnarray*}

We first estimate the term $A_{2}$. Note that $\left\Vert F\right\Vert
_{2k+2}\leq C_{q,k}\left\Vert F\right\Vert _{2}=C_{q,k}\sigma $ by the
hypercontractivity property (\ref{Hyper}). Applying Lemma \ref{MS-ctrl} with
$h(z)=\mathbf{1}_{\left\{ z>x\right\} }H_{k+1}(\sigma ^{-2},\sigma ^{-2}z)$,
which satisfies $\left\vert h(z)\right\vert \leq C_{k}(\left\vert
z\right\vert ^{k+1}+\sigma ^{-k-1})$, we obtain%
\begin{eqnarray}
A_{2} &=&\left\vert E[h(F)-h(N)]\right\vert  \notag \\
&\leq &C_{q,k}\sigma ^{-2}\left\vert \sigma ^{k}+4\sigma ^{-k-1}\right\vert
\left\Vert \sigma ^{2}-\left\langle DF,-DL^{-1}F\right\rangle _{\mathfrak{H}%
}\right\Vert _{2}  \notag \\
&\leq &C_{q,k,\sigma }\sigma ^{-k-3}\left\Vert q\sigma ^{2}-w\right\Vert
_{2},  \label{deri-A2}
\end{eqnarray}%
where in the second inequality we used the fact that $\left\langle
DF,-DL^{-1}F\right\rangle _{\mathfrak{H}}=\frac{w}{q}$.

For the term $A_{1}$, Lemma \ref{Gk} implies
\begin{equation}
A_{1}\leq E[\left\vert H_{k+1}(D_{u}\delta _{u},\delta _{u})-H_{k+1}(\sigma
^{-2},\sigma ^{-2}F)\right\vert ]+E[\left\vert T_{k+1}\right\vert ].
\label{deri-A1}
\end{equation}%
To proceed with the first term above, applying (\ref{gHermite}) we have
\begin{eqnarray}
&&\left\vert H_{k+1}(D_{u}\delta _{u},\delta _{u})-H_{k+1}(\sigma
^{-2},\sigma ^{-2}F)\right\vert  \label{H-pln} \\
&\leq &\sum_{0\leq i\leq \lfloor (k+1)/2\rfloor }|c_{k,i}|\left\vert \delta
_{u}^{k+1-2i}\left( D_{u}\delta _{u}\right) ^{i}-\left( \sigma ^{-2}F\right)
^{k+1-2i}\sigma ^{-2i}\right\vert  \notag \\
&\leq &\sum_{0\leq i\leq \lfloor (k+1)/2\rfloor }|c_{k,i}|\left[ \left\vert
\delta _{u}^{k+1-2i}-\left( \sigma ^{-2}F\right) ^{k+1-2i}\right\vert
\left\vert D_{u}\delta _{u}\right\vert ^{i}+\left\vert \sigma
^{-2}F\right\vert ^{k+1-2i}\left\vert \left( D_{u}\delta _{u}\right)
^{i}-\sigma ^{-2i}\right\vert \right] .  \notag
\end{eqnarray}%
Using the fact that $\left\vert x^{k}-y^{k}\right\vert \leq C_{k}\left\vert
x-y\right\vert \sum_{0\leq j\leq k-1}\left\vert x\right\vert
^{k-1-j}\left\vert y\right\vert ^{j}$ and applying H\"{o}lder's inequality
and the hypercontractivity property (\ref{Hyper}) we obtain
\begin{eqnarray}
&&E\left[ \left\vert \delta _{u}^{k+1-2i}-\left( \sigma ^{-2}F\right)
^{k+1-2i}\right\vert \left\vert D_{u}\delta _{u}\right\vert ^{i}\right]
\notag \\
&\leq &C_{k}E\left[ \left\vert \delta _{u}-\sigma ^{-2}F\right\vert
\left\vert D_{u}\delta _{u}\right\vert ^{i}\sum_{0\leq j\leq k-2i}\left\vert
\delta _{u}\right\vert ^{k-2i-j}\left\vert \sigma ^{-2}F\right\vert ^{j}%
\right]  \notag \\
&\leq &C_{q,k,\sigma }\left\Vert \delta _{u}-\sigma ^{-2}F\right\Vert
_{r}\left\Vert D_{u}\delta _{u}\right\Vert _{s}^{i}\sum_{0\leq j\leq
k-2i}\left\Vert \delta _{u}\right\Vert _{p}^{k-2i-j}\sigma ^{-j},
\label{deri-A1-1}
\end{eqnarray}%
provided $1\geq \frac{1}{r}+\frac{i}{s}+\frac{k-2i-j}{p}$, which is implied
by $1\geq \frac{1}{r}+\frac{i}{s}+\frac{k-2i}{p}$. In order to apply the
estimates (\ref{dl-u1}), (\ref{Dudelta-uka}) (with $k=1$) and (\ref%
{Dudelta-u0}) we need $\frac{1}{r}>\frac{3}{\beta }+\frac{1}{2}$, $\frac{1}{s%
}>\frac{6}{\beta }$ and $\frac{1}{p}>\frac{3}{\beta }$, respectively. These
are possible because $\beta >6k+6$. Then we obtain an estimate of the form
\begin{equation}
E\left[ \left\vert \delta _{u}^{k+1-2i}-\left( \sigma ^{-2}F\right)
^{k+1-2i}\right\vert \left\vert D_{u}\delta _{u}\right\vert ^{i}\right] \leq
C_{q,k,\sigma }\sigma ^{-k}(M_{\beta }^{3k+3}\vee 1)\Vert q\sigma
^{2}-w\Vert _{2}.  \label{b1}
\end{equation}%
Similarly,
\begin{eqnarray}
&&E\left[ \left\vert \sigma ^{-2}F\right\vert ^{k+1-2i}\left\vert \left(
D_{u}\delta _{u}\right) ^{i}-\sigma ^{-2i}\right\vert \right]  \notag \\
&\leq &C_{q,k,\sigma }E\left[ \left\vert \sigma ^{-2}F\right\vert
^{k+1-2i}\left\vert D_{u}\delta _{u}-\sigma ^{-2}\right\vert \sum_{0\leq
j\leq i-1}\left\vert D_{u}\delta _{u}\right\vert ^{j}\sigma ^{-2(i-1-j)}%
\right]  \notag \\
&\leq &C_{q,k,\sigma }\sigma ^{-(k-1)}\left\Vert D_{u}\delta _{u}-\sigma
^{-2}\right\Vert _{r}\sum_{0\leq j\leq i-1}\left\Vert D_{u}\delta
_{u}\right\Vert _{s}^{j},  \label{deri-A1-2}
\end{eqnarray}%
provided $1>\frac{1}{r}+\frac{j}{s}$. In order to apply the estimates (\ref%
{dudl-u1}) and (\ref{Dudelta-uka}) (with $k=1$) we need $\frac{1}{r}>\frac{6%
}{\beta }+\frac{1}{2}$ and $\frac{1}{s}>\frac{6}{\beta }$, respectively.
This implies
\begin{equation*}
\frac{1}{r}+\frac{j}{s}>\frac{6+6j}{\beta }+\frac{1}{2}.
\end{equation*}%
Notice that $6+6j\leq 6i\leq 3k+3$. So, we need $1>\frac{1}{2}+\frac{3k+3}{%
\beta }$. The above $r,s$ and $p$ exist because $\beta >6k+6$. Thus, we
obtain an estimate of the form
\begin{equation}
E\left[ \left\vert \sigma ^{-2}F\right\vert ^{k+1-2i}\left\vert \left(
D_{u}\delta _{u}\right) ^{i}-\sigma ^{-2i}\right\vert \right] \leq
C_{q,k,\sigma ,\beta }\sigma ^{-(k-1)}(M_{\beta }^{3k+3}\vee 1)\Vert q\sigma
^{2}-w\Vert _{2}.  \label{b2}
\end{equation}%
Combining (\ref{b1}) and (\ref{b2}) we have%
\begin{equation}
E\left[ \left\vert H_{k+1}(D_{u}\delta _{u},\delta _{u})-H_{k+1}(\sigma
^{-2},\sigma ^{-2}F)\right\vert \right] \leq C_{q,k,\sigma ,\beta }\sigma
^{-k}(M_{\beta }^{3k+3}\vee 1)\Vert q\sigma ^{2}-w\Vert _{2}.
\label{deri-A1a}
\end{equation}

Applying H\"{o}lder's inequality to the expression (\ref{HODT}) we obtain,
(assuming $k\geq 2$ , otherwise $T_{k+1}=0$)
\begin{equation*}
E\left[ \left\vert T_{k+1}\right\vert \right] \leq C_{q,k,\sigma ,\beta
}\sum_{(i_{0},\dots ,i_{k})\in J_{k+1}}\left\Vert \delta _{u}\right\Vert
_{r_{0}}^{i_{0}}\prod_{j=1}^{k}\left\Vert D_{u}^{j}\delta _{u}\right\Vert
_{r_{j}}^{i_{j}},
\end{equation*}%
where $1=\frac{i_{0}}{r_{0}}+\sum_{j=1}^{k}\frac{i_{j}}{r_{j}}$. From
property (b) in Lemma \ref{rem1} there is at least one factor of the form $%
\left\Vert D_{u}^{j}\delta _{u}\right\Vert _{s_{j}}$ with $j\geq 2$. We
apply the estimate (\ref{Dudelta-uk}) to one of these factors, and the
estimate (\ref{Dudelta-uka}) to all the remaining factors. We also use the
estimate (\ref{Dudelta-u0}) to control $\left\Vert \delta _{u}\right\Vert
_{r_{0}}$. Notice that
\begin{equation*}
1=\frac{i_{0}}{r_{0}}+\sum_{j=1}^{k}\frac{i_{j}}{r_{j}}>\frac{3i_{0}}{\beta }%
+\sum_{j=1}^{k}\frac{i_{j}(3j+3)}{\beta }+\frac{1}{2},
\end{equation*}%
and, on the other hand, using properties (a) and (c) in Lemma \ref{rem1}
\begin{equation*}
\frac{3i_{0}}{\beta }+\sum_{j=1}^{k}\frac{i_{j}(3j+3)}{\beta }+\frac{1}{2}%
\leq \frac{3k+3\lfloor \frac{k}{2}\rfloor }{\beta }+\frac{1}{2}.
\end{equation*}%
We can choose the $r_{j}$'s satisfying the above properties because $\beta
>6k+6\lfloor \frac{k}{2}\rfloor $, and we obtain
\begin{equation}
E\left\vert T_{k+1}\right\vert \leq C_{q,k,\sigma ,\beta }(M_{\beta
}^{3k+3\lfloor \frac{k}{2}\rfloor }\vee 1)\left\Vert q\sigma
^{2}-w\right\Vert _{2}.  \label{b3}
\end{equation}%
Combining (\ref{deri-A1a}) and (\ref{b3}) we complete the proof.
\end{proof}

\begin{corollary}
\label{qDeriConv} Fix an integer $m\geq 1$. Let $\left\{ F_{n}\right\}
_{n\in \mathbb{N}}$ be a sequence of random variables in the $q$th Wiener
chaos with $q\geq 2$ and $E[F_{n}^{2}]=\sigma _{n}^{2}$. Assume $%
\lim_{n\rightarrow \infty }\sigma _{n}=\sigma $, $0<\delta \leq \sigma
_{n}^{2}\leq K$ for all $n$, $\lim_{n\rightarrow \infty
}E[F_{n}^{4}]=3\sigma ^{4}$ and%
\begin{equation}
M:=\sup_{n}\left( \mathbb{E}[\left\Vert DF_{n}\right\Vert _{\mathfrak{H}%
}^{-\beta }]\right) ^{\frac{1}{\beta }}<\infty  \label{negMmts}
\end{equation}%
for some $\beta >6(m)+6(\lfloor \frac{m}{2}\rfloor \vee 1)$. Let $\phi (x)$
be the density of $N(0,\sigma ^{2})$. Then, each $F_{n}$ admits a
probability density function $f_{F_{n}}\in C^{m}(\mathbb{R})$ with
derivatives given by $\left( \ref{kthFml}\right) $ and for any $k=1,\dots ,m$%
,
\begin{equation*}
\sup_{x\in \mathbb{R}}\left\vert f_{F_{n}}^{(k)}(x)-\phi
^{(k)}(x)\right\vert \leq C\left( \sqrt{E[F_{n}^{4}]-3\sigma _{n}^{4}}%
+\left\vert \sigma _{n}-\sigma \right\vert \right) ,
\end{equation*}%
where the constant $C$ depends only on $q,m,\beta ,M,\sigma ,\delta $ and $K$%
.
\end{corollary}

\begin{proof}
Let $\phi _{n}(x)$ be the density of $N(0,\sigma _{n}^{2})$. Then Theorem %
\ref{qDeriRate} implies that
\begin{equation*}
\sup_{x\in \mathbb{R}}\left\vert f_{F_{n}}^{(k)}(x)-\phi
_{n}^{(k)}(x)\right\vert \leq C_{q,m,\beta ,M,\sigma }\sqrt{
E[F_{n}^{4}]-3\sigma _{n}^{4}}.
\end{equation*}
On the other hand, by the mean value theorem we can write
\begin{equation*}
\left\vert \phi _{n}^{(k)}(x)-\phi ^{(k)}(x)\right\vert \leq |\sigma
_{n}-\sigma |\sup_{\gamma \in \lbrack \frac{\sigma }{2},2\sigma ]}\left\vert
\partial _{\gamma }\phi _{\gamma }^{(k)}(x)\right\vert =\frac{1}{2}|\sigma
_{n}-\sigma |\sup_{\gamma \in \lbrack \frac{\sigma }{2},2\sigma ]}\gamma
\left\vert \phi _{\gamma }^{(k+2)}(x)\right\vert ,
\end{equation*}%
where $\phi _{\gamma }(x)$ is the density of the law $N(0,\gamma ^{2})$.
Then, using the expression
\begin{equation*}
\phi _{\gamma }^{(k+2)}(x)=E[\mathbf{1}_{N>x}H_{k+3}(\gamma ^{-2},\gamma
^{-2}Z)],
\end{equation*}%
where $Z\sim N(0,\gamma ^{2})$ and the explicit form of $H_{k+3}(\lambda ,x)$%
, we obtain
\begin{equation*}
\sup_{\gamma \in \lbrack \frac{\sigma }{2},2\sigma ]}\gamma \left\vert \phi
_{\gamma }^{(k+2)}(x)\right\vert \leq C_{k,\sigma }.
\end{equation*}%
Therefore,
\begin{equation*}
\sup_{x\in \mathbb{R}}\left\vert \phi _{n}^{(k)}(x)-\phi
^{(k)}(x)\right\vert \leq C_{k,\sigma }\left\vert \sigma _{n}-\sigma
\right\vert .
\end{equation*}%
This completes the proof.
\end{proof}

\section{Random vectors in Wiener chaos\label{MultiConv}}

\subsection{Main result}


In this section, we study the multidimensional counterpart of Theorem $\ref%
{qDeriConv}$. We begin with a density formula for a smooth random vector.

A random vector $F=(F_{1},\dots ,F_{d})$ in $\mathbb{D}^{\infty }$ is called
\textit{non-degenerate} if its \textit{Malliavin matrix} $\gamma
_{F}=(\left\langle DF_{i},DF_{j}\right\rangle _{\mathfrak{H}})_{1\leq
i,j\leq d}$ is invertible a.s. and $(\det \gamma _{F})^{-1}\in \mathbb{\cap }%
_{p\geq 1}L^{p}(\Omega )$. For any multi-index
\begin{equation*}
\beta =\left( \beta _{1},\beta _{2},\dots ,\beta _{k}\right) \in \left\{
1,2,\dots ,d\right\} ^{k}
\end{equation*}%
of length $k\geq 1$, the symbol $\partial _{\beta }$ stands for the partial
derivative $\frac{\partial ^{k}}{\partial x_{\beta _{1}}\dots \partial
x_{\beta _{k}}}$. For $\beta $ of length $0$ we make the convention that $%
\partial _{\beta }f=f$. We denote by $\mathcal{S}(\mathbb{R}^{d})$ the
Schwartz space of rapidly decreasing smooth functions, that is, the space of
all infinitely differentiable functions $f:\mathbb{R}^{d}\rightarrow \mathbb{%
R}$ such that $\sup_{x\in \mathbb{R}^{d}}\left\vert x\right\vert
^{m}\left\vert \partial _{\beta }f(x)\right\vert <\infty $ for any
nonnegative integer $m$ and for all multi-index $\beta $. The following
lemma (see Nualart \cite[Proposition 2.1.5]{Nu06}) gives an explicit formula
for the density of $F$.

\begin{lemma}
\label{V-den}Let $F=(F_{1},\dots ,F_{d})$ be a non-degenerate random vector.
Then, $F$ has a density $f_{F}\in \mathcal{S}(\mathbb{R}^{d})$, and $f_{F}$
and its partial derivative $\partial _{\beta }f_{F}$, for any multi-index $%
\beta =\left( \beta _{1},\beta _{2},\dots ,\beta _{k}\right) $ of length $%
k\geq 0$, are given by%
\begin{equation}
f_{F}(x)=E[\mathbf{1}_{\left\{ F>x\right\} }H_{(1,2,\dots ,d)}(F)],
\label{denV1}
\end{equation}%
\begin{equation}
\partial _{\beta }f_{F}(x)=\left( -1\right) ^{k}E[\mathbf{1}_{\left\{
F>x\right\} }H_{\left( 1,2,\dots ,d,\beta _{1},\beta _{2},\dots ,\beta
_{k}\right) }(F)],  \label{denV2}
\end{equation}%
where $\mathbf{1}_{\left\{ F>x\right\} }=\prod_{i=1}^{d}\mathbf{1}_{\left\{
F_{i}>x_{i}\right\} }$ and the elements $H_{\beta }(F)$ are recursively
defined by
\begin{equation}
\left\{
\begin{array}{ll}
H_{\beta }(F)=1, & \text{if }k=0; \\
&  \\
H_{\left( \beta _{1},\beta _{2},\dots ,\beta _{k}\right)
}(F)=\sum_{j=1}^{d}\delta \left( H_{\left( \beta _{1},\beta _{2},\dots
,\beta _{k-1}\right) }(F)\left( \gamma _{F}^{-1}\right) ^{\beta
_{1}j}DF_{j}\right) , & \text{if }k\geq 1.%
\end{array}%
\right.  \label{H1}
\end{equation}
\end{lemma}

Fix $d$ natural numbers $1\leq q_{1}\leq \cdots \leq q_{d}$. We will
consider a random vector of multiple stochastic integrals: $F=(F_{1},\dots
,F_{d})=(I_{q_{1}}(f_{1}),\dots ,I_{q_{d}}(f_{d}))$, where $f_{i}\in
\mathfrak{H}^{\odot q_{i}}$. Denote
\begin{equation}
V=\left( E[F_{i}F_{j}]\right) _{1\leq i,j\leq d}\text{, }\quad Q=\mathrm{diag%
}(q_{1},\dots ,q_{d})\quad (\text{diagonal matrix of elements }q_{1},\dots
,q_{d}).  \label{VQ}
\end{equation}%
Along this section, we denote by $N=(N_{1},\dots ,N_{d})$ a standard normal
vector given by $N_{i}=I_{1}(h_{i})$, where $h_{i}\in \mathfrak{H}$ are
orthonormal. We denote by $I$ the $d$ dimensional identity matrix, and by $%
|\cdot |$ the Hilbert-Schmidt norm of a matrix. The following is the main
theorem of this section.

\begin{theorem}
\label{MultiThm0}Let $F=(F_{1},\dots ,F_{d})=(I_{q_{1}}(f_{1}),\dots
,I_{q_{d}}(f_{d}))$ be non-degenerate and let $\phi $ be the density of $N$.
Then for any multi-index $\beta $ of length $k\geq 0$, the density $f_{F}$
of $F$ satisfies%
\begin{equation}
\sup_{x\in \mathbb{R}^{d}}\left\vert \partial _{\beta }f_{F}(x)-\partial
_{\beta }\phi (x)\right\vert \leq C\left( |V-I|+\sum_{1\leq j\leq d}\sqrt{%
E[F_{j}^{4}]-3(E[F_{j}^{2}])^{2}}\right) \,,  \label{mainresult6}
\end{equation}%
where the constant $C$ depends on\ $d,V,Q,k$ and $\left\Vert \left( \det
\gamma _{F}\right) ^{-1}\right\Vert _{(k+4)2^{k+3}}$.
\end{theorem}

\begin{proof}
Note that $\partial _{\beta }\phi (x)=\left( -1\right) ^{k}E[\mathbf{1}%
_{\left\{ N>x\right\} }H_{\left( 1,2,\dots ,d,\beta _{1},\beta _{2},\dots
,\beta _{k}\right) }(N)]$. Then, in order to estimate the difference between
$\partial _{\beta }f_{F_{n}}$ and $\partial _{\beta }\phi $, it suffices to
estimate
\begin{equation*}
E[\mathbf{1}_{\left\{ F>x\right\} }H_{\beta }(F)]-E[\mathbf{1}_{\left\{
N>x\right\} }H_{\beta }(N)]
\end{equation*}%
for all multi-index $\beta $ of length $k$ for all $k\geq d$.

Fix a multi-index $\beta $ of length $k$ for some $k\geq d$. For the above
standard normal random vector $N$, we have $\gamma _{N}=I$ and $\delta
(DN_{i})=N_{i}$. We can deduce from the expression (\ref{H1}) that $H_{\beta
}(N)=g_{\beta }(N)$, where $g_{\beta }(x)$ is a polynomial on $\mathbb{R}%
^{d} $ (see Remark \ref{g_beta_rmk}). Then,
\begin{eqnarray}
&&\left\vert E[\mathbf{1}_{\left\{ F>x\right\} }H_{\beta }(F)]-E[\mathbf{1}%
_{\left\{ N>x\right\} }H_{\beta }(N)]\right\vert  \notag \\
&\leq &\left\vert E[\mathbf{1}_{\left\{ F>x\right\} }g_{\beta }(F)]-E[%
\mathbf{1}_{\left\{ N>x\right\} }g_{\beta }(N)]\right\vert +E\left[
\left\vert H_{\beta }(F)-g_{\beta }(F)\right\vert \right]  \notag \\
&=&A_{1}+A_{2}.  \label{V-main}
\end{eqnarray}%
The term $A_{1}=\left\vert E[\mathbf{1}_{\left\{ F>x\right\} }g_{\beta }(F)-%
\mathbf{1}_{\left\{ N>x\right\} }g_{\beta }(N)]\right\vert $ will be studied
in Subsection \ref{s.stein} by using the multivariate Stein's method.
Proposition \ref{Stein-prop} will imply that $A_{1}$ is bounded by the
right-hand side of (\ref{mainresult6}).

Consider the term $A_{2}=E\left[ \left\vert H_{\beta }(F)-g_{\beta
}(F)\right\vert \right] $.\ We introduce an auxiliary term $K_{\beta }(F)$,
which is defined similar to $H_{\beta }(F)$ with $\gamma _{F}^{-1}$ replaced
by $\left( VQ\right) ^{-1}$. That is, for any multi-index $\beta =\left(
\beta _{1},\beta _{2},\dots ,\beta _{k}\right) $ of length $k\geq 0$, we
define
\begin{equation}
\left\{
\begin{array}{ll}
K_{\beta }(F)=1 & \qquad \text{if }k=0; \\
&  \\
K_{\beta }(F)=\delta \left( K_{\left( \beta _{1},\beta _{2},\dots ,\beta
_{k-1}\right) }(F)\left( \left( VQ\right) ^{-1}DF\right) _{\beta _{k}}\right)
& \qquad \text{if }k\geq 1.%
\end{array}%
\right.  \label{K}
\end{equation}%
We have
\begin{equation}
A_{2}\leq E\left[ \left\vert H_{\beta }(F)-K_{\beta }(F)\right\vert \right]
+E[\left\vert K_{\beta }(F)-g_{\beta }(F)\right\vert ]=:A_{3}+A_{4}\,.
\label{KH2}
\end{equation}%
Lemma \ref{HKLemma} below shows that the term $A_{3}=E\left[ \left\vert
H_{\beta }(F)-K_{\beta }(F)\right\vert \right] $ is bounded by the
right-hand of (\ref{mainresult6}).

It remains to estimate $A_{4}$. For this we need the following lemma which
provides an explicit expression for the term $K_{\beta }(F)$. Before stating
this lemma we need to introduce some notation. For any 
multi-index $\beta =\left( \beta _{1},\beta _{2},\dots ,\beta _{k}\right) $,
$k\geq 1$, denote by $\widehat{\beta }_{i_{1}\dots i_{m}}$ the multi-index
obtained from $\beta $ after taking away the elements $\beta _{i_{1}},\beta
_{i_{2}},\dots ,\beta _{i_{m}}$. For example, $\widehat{\beta }_{14}=\left(
\beta _{2},\beta _{3}\,,\beta _{5}\,,\dots ,\beta _{k}\right) $. For any $d$
dimensional vector $G$ we denote by $G_{\beta }$ the product $G_{\beta
_{1}}G_{\beta _{2}}\cdots G_{\beta _{k}}$ and set $G_{\beta }=1$ if the
length of $\beta $ is $0$. Denote by $\left( S_{k}^{m};0\leq m\leq \lfloor
\frac{k}{2}\rfloor \right) $ the following sets
\begin{equation}
\left\{
\begin{array}{ll}
S_{k}^{-1}=S_{k}^{0}=\varnothing &  \\
S_{k}^{m}=\left\{
\begin{array}{c}
\{(i_{1},i_{2}),\dots ,(i_{2m-1},i_{2m})\}\in \left\{ 1,2,\dots ,k\right\}
^{2m}: \\
i_{2l-1}<i_{2l}\text{ for }1\leq l\leq m\,\text{and }i_{l}\neq i_{j}\text{
if }l\neq j%
\end{array}%
\right\} &
\end{array}%
\right.  \label{S_k^m}
\end{equation}%
For each element $\{(i_{1},i_{2}),\dots ,(i_{2m-1},i_{2m})\}\in S_{k}^{m}$,
we emphasize that the $m$ pairs of indices are unordered. In other words,
for $m\geq 1$, the set $S_{k}^{m}$ can be viewed as the set of all
partitions of $\{1,2,\dots ,k\}$ into $m$ pairs and $k-2m$ singletons.

Denote $M=V^{-1}\gamma _{F}V^{-1}Q^{-1}$ for $V$ and $Q$ given in $(\ref{VQ}%
) $ and denote $M_{ij}$ the $(i,j)$-th entry of $M$. Denote by $D_{\beta
_{i}}$ the Malliavin derivative in the direction of $\left(
V^{-1}Q^{-1}DF\right) _{\beta _{i}}=V^{-1}Q^{-1}DF_{\beta _{i}}$, that is,
\begin{equation}
D_{\beta _{i}}G=\left\langle DG,\left( V^{-1}Q^{-1}DF\right) _{\beta
_{i}}\right\rangle _{\mathfrak{H}}  \label{D_beta}
\end{equation}%
for any random variable $G\in \mathbb{D}^{1,2}$. %
%

\begin{lemma}
\label{K_beta}Let $F$ be a non-degenerate random vector. For a multi-index $%
\beta =(\beta _{1},\dots ,\beta _{k})$ of length $k\geq 0$, $K_{\beta }(F)$
defined by $(\ref{K})$ can be computed as follows:
\begin{equation}
K_{\beta }(F)=G_{\beta }(F)+T_{\beta }(F),  \label{K-expsn}
\end{equation}%
where 
%
%
\begin{equation}
G_{\beta }(F)=\sum_{m=0}^{\lfloor k/2\rfloor }\left( -1\right)
^{m}\sum_{\{(i_{1},i_{2}),\dots ,(i_{2m-1},i_{2m})\}\in S_{k}^{m}}\left(
V^{-1}F\right) _{\widehat{\beta }_{i_{1}\cdots i_{2m}}}M_{^{_{\beta
_{i_{1}}\beta _{i_{2}}}}}\cdots M_{^{\beta _{i_{2m-1}}\beta _{2m}}},
\label{G-beta-k}
\end{equation}%
and $T_{\beta }(F)$ are defined recursively by
\begin{eqnarray}
T_{\beta }(F) &=&\left( V^{-1}F\right) _{\beta _{k}}T_{\widehat{\beta }%
_{k}}(F)-D_{\beta _{k}}T_{\widehat{\beta }_{k}}(F)  \label{KT} \\
&&-\sum_{\{(i_{1},i_{2}),\dots ,(i_{2m-1},i_{2m})\}\in S_{k-1}^{m}}\left(
V^{-1}F\right) _{\widehat{\beta }_{ki_{1}\cdots i_{2m}}}D_{\beta _{k}}\left(
M_{^{\beta _{i_{1}}\beta _{i_{2}}}}\cdots M_{^{_{\beta _{i_{2m-1}}\beta
_{2m}}}}\right) ,  \notag
\end{eqnarray}%
for $k\geq 2$ and $T_{1}(F)=T_{2}(F)=0$.
\end{lemma}


\begin{proof}
For simplicity, we write $K_{\beta }$, $G_{\beta }$ and $T_{\beta }$ for $%
K_{\beta }(F)$, $G_{\beta }(F)$ and $T_{\beta }(F)$, respectively. By using
the fact that $\delta \left( \left( \left( VQ\right) ^{-1}DF\right) _{\beta
_{i}}\right) =\left( V^{-1}F\right) _{\beta _{i}}$ we obtain from (\ref{K})
that
\begin{equation}
K_{\beta }=\left( V^{-1}F\right) _{\beta _{k}}K_{\widehat{\beta }%
_{k}}-D_{\beta _{k}}K_{\widehat{\beta }_{k}}.  \label{K-1}
\end{equation}%
If if $k=1$, namely, $\beta =(\beta _{1})$, then
\begin{equation*}
K_{\beta }=\left( V^{-1}F\right) _{\beta _{1}}=G_{\beta }.
\end{equation*}%
If $k=2$, namely, $\beta =(\beta _{1},\beta _{2})$, then%
\begin{equation*}
K_{\beta }=\left( V^{-1}F\right) _{\beta }-M_{^{\beta _{1}\beta
_{2}}}=G_{\beta }\,.
\end{equation*}%
Hence, the identity (\ref{K-expsn}) is true for $k=1,2$. Assume now (\ref%
{K-expsn}) is true for all multi-index of length less than or equal to $k$.
Let $\beta =(\beta _{1},\dots ,\beta _{k+1})$. Then, (\ref{K-1}) implies%
\begin{equation}
K_{\beta }=\left( V^{-1}F\right) _{\beta _{k+1}}\left( G_{\widehat{\beta }%
_{k+1}}+T_{\widehat{\beta }_{k+1}}\right) -D_{\beta _{k+1}}\left( G_{%
\widehat{\beta }_{k+1}}+T_{\widehat{\beta }_{k+1}}\right) .  \label{K-2}
\end{equation}%
Noticing that
\begin{equation*}
D_{\beta _{k+1}}\left( V^{-1}F\right) _{\widehat{\beta }_{(k+1)i_{1}\cdots
i_{2m}}}=\sum_{j\in \left\{ 1,\dots ,k\right\} \backslash \left\{
i_{1},\dots ,i_{2m}\right\} }\left( V^{-1}F\right) _{\widehat{\beta }%
_{(k+1)ji_{1}\cdots i_{2m}}}M_{^{\beta _{j}\beta _{k+1}}},
\end{equation*}%
we have
\begin{eqnarray}
&&D_{\beta _{k+1}}G_{\widehat{\beta }_{k+1}}=B_{\beta }  \label{D-beta-k+1}
\\
&&+\sum_{m=0}^{\lfloor k/2\rfloor }\left( -1\right)
^{m}\sum_{\{(i_{1},i_{2}),\dots ,(i_{2m-1},i_{2m})\}\in S_{k}^{m}}\left(
V^{-1}F\right) _{\widehat{\beta }_{(k+1)i_{1}\cdots i_{2m}}}D_{\beta
_{k+1}}\left( M_{^{\beta _{i_{1}}\beta _{i_{2}}}}\cdots M_{^{\beta
_{i_{2m-1}}\beta _{2m}}}\right) ,  \notag
\end{eqnarray}%
where we let
\begin{equation*}
B_{\beta }=\sum_{m=0}^{\lfloor k/2\rfloor }\left( -1\right) ^{m}\sum
_{\substack{ \{(i_{1},i_{2}),\dots ,(i_{2m-1},i_{2m})\}\in S_{k}^{m},  \\ %
j\in \left\{ 1,\dots ,k\right\} \backslash \left\{ i_{1},\dots
,i_{2m}\right\} }}\left( V^{-1}F\right) _{\widehat{\beta }%
_{j(k+1)i_{1}\cdots i_{2m}}}M_{^{\beta _{j}\beta _{k+1}}}M_{^{\beta
_{i_{1}}\beta _{i_{2}}}}\cdots M_{^{\beta _{i_{2m-1}}\beta _{i_{2m}}}}.
\end{equation*}%
Substituting the expression (\ref{D-beta-k+1}) for $D_{\beta _{k+1}}G_{%
\widehat{\beta }_{k+1}}$ into (\ref{K-2}) and using (\ref{KT}) we obtain
\begin{equation*}
K_{\beta }=\left( V^{-1}F\right) _{\beta _{k+1}}G_{\widehat{\beta }%
_{k+1}}-B_{\beta }+T_{\beta }.
\end{equation*}%
To arrive at (\ref{K-expsn}) it remains to verify
\begin{equation}
G_{\beta }=\left( V^{-1}F\right) _{\beta _{k+1}}G_{\widehat{\beta }%
_{k+1}}-B_{\beta }.  \label{G-Beta1}
\end{equation}%
Introduce the following notation
\begin{equation}
C_{k+1}^{m}=\left\{ \left\{ (i_{1},i_{2}),\dots
,(i_{2m-3},i_{2m-2}),(j,k+1)\right\} :\{(i_{1},i_{2}),\dots
,(i_{2m-3},i_{2m-2})\}\in S_{k}^{m-1}\right\}  \label{C_k_m}
\end{equation}%
for $1\leq m\leq \lfloor \frac{k}{2}\rfloor $. Then, $S_{k+1}^{m}$ can be
decomposed as follows%
\begin{equation}
S_{k+1}^{m}=S_{k}^{m}\cup C_{k+1}^{m}.  \label{Set=}
\end{equation}

We consider first the case when $k$ is even. In this case, noticing that for
any element in $\{(i_{1},i_{2}),\dots ,(i_{2m-1},i_{2m})\}\in S_{k}^{\lfloor
\frac{k}{2}\rfloor }$, $\left\{ 1,\dots ,k\right\} \backslash \left\{
i_{1},\dots ,i_{2m}\right\} =\varnothing $. For any collection of indices $%
i_1, \dots, i_{2m} \subset \{1,2, \dots, k\}$, we set
\begin{equation*}
\Phi_{i_1 \dots i_{2m}}=\left( V^{-1}F\right) _{\widehat{\beta }_{
i_{1}\cdots i_{2m}}} M_{^{\beta _{i_{1}}\beta _{i_{2}}}}\cdots M_{^{\beta
_{i_{2m-1}}\beta _{i_{2m}}}}.
\end{equation*}
Then, we have%
\begin{eqnarray}
~ &-B_{\beta }=&\sum_{m=0}^{\lfloor \frac{k}{2}\rfloor -1}\left( -1\right)
^{m+1}\sum_{\substack{ \{(i_{1},i_{2}),\dots ,(i_{2m-1},i_{2m})\}\in
S_{k}^{m},  \\ j\in \left\{ 1,\dots ,k\right\} \backslash \left\{
i_{1},\dots ,i_{2m}\right\} }}\Phi_{j(k+1)i_1 \dots i_{2m}}  \notag \\
&=&\sum_{m=1}^{\lfloor \frac{k}{2}\rfloor }\left( -1\right) ^{n}\sum
_{\substack{ \{(i_{1},i_{2}),\dots ,(i_{2m-3},i_{2m-2})\}\in S_{k}^{m-1},
\\ j\in \left\{ 1,\dots ,k\right\} \backslash \left\{ i_{1},\dots
,i_{2n-2}\right\} }}\Phi_{j(k+1)i_1 \dots i_{2m-2}}  \label{G-Beta2} \\
&=&\sum_{m=1}^{\lfloor \frac{k+1}{2}\rfloor }\left( -1\right) ^{m}\sum
_{\substack{ \left\{ (i_{1},i_{2}),\dots
,(i_{2m-3},i_{2m-2}),(j,k+1)\right\}  \\ \in C_{k+1}^{m}}}\Phi_{j(k+1)i_1
\dots i_{2m-2}}\,,  \notag
\end{eqnarray}%
where in the last equality we used (\ref{C_k_m}) and the fact that $\lfloor
\frac{k}{2}\rfloor =\lfloor \frac{k+1}{2}\rfloor $ since $k$ is even. Taking
into account that $\left( V^{-1}F\right) _{\beta _{k+1}}\left(
V^{-1}F\right) _{\widehat{\beta }_{(k+1)i_{1}\cdots i_{2m}}}=\left(
V^{-1}F\right) _{\widehat{\beta }_{i_{1}\cdots i_{2m}}}$, we obtain from (%
\ref{G-beta-k}) that
\begin{equation}
\left( V^{-1}F\right) _{\beta _{k+1}}G_{\widehat{\beta }_{k+1}}=\sum_{m=0}^{%
\lfloor k/2\rfloor }\left( -1\right) ^{m}\sum_{\{(i_{1},i_{2}),\dots
,(i_{2m-1},i_{2m})\}\in S_{k}^{m}}\left( V^{-1}F\right) _{\widehat{\beta }%
_{i_{1}\cdots i_{2m}}}M_{^{\beta _{i_{1}}\beta _{i_{2}}}}\cdots M_{^{\beta
_{i_{2m-1}}\beta _{i_{2m}}}}.  \label{G-Beta3}
\end{equation}%
Now combining (\ref{G-Beta2}) and (\ref{G-Beta3}) with (\ref{Set=}) and
using again $\lfloor \frac{k}{2}\rfloor =\lfloor \frac{k+1}{2}\rfloor $ we
obtain
\begin{eqnarray*}
\left( V^{-1}F\right) _{\beta _{k+1}}G_{\widehat{\beta }_{k+1}}-B_{\beta }
&=&\sum_{m=0}^{\lfloor \left( k+1\right) /2\rfloor }\left( -1\right)
^{m}\sum_{\{(i_{1},i_{2}),\dots ,(i_{2m-1},i_{2m})\}\in S_{k}^{m}}\Phi_{
i_1\dots i_{2m}} \\
&&+\sum_{m=1}^{\lfloor \left( k+1\right) /2\rfloor }\left( -1\right)
^{m}\sum_{\left\{ (i_{1},i_{2}),\dots ,(i_{2m-3},i_{2m-2}),(j,k+1)\right\}
\in C_{k+1}^{m}}\Phi_{ i_1\dots i_{2m}} \\
&=&\sum_{m=0}^{\lfloor \left( k+1\right) /2\rfloor }\left( -1\right)
^{m}\sum_{\{(i_{1},i_{2}),\dots ,(i_{2m-1},i_{2m})\}\in S_{k+1}^{m}}\Phi_{
i_1\dots i_{2m}} \\
&=&G_{\beta }
\end{eqnarray*}%
as desired. This verifies (\ref{G-Beta1}) for the case $k$ is even.\ The
case when $k$ is odd can be verified similarly. Thus, we have proved (\ref%
{K-expsn}) by induction.
\end{proof}

\begin{remark}
\label{g_beta_rmk}For the random vector $N\sim N(0,I)$, we have $\gamma
_{N}=VQ=I$, so $H_{\beta }(N)=K_{\beta }(N)$. Then, it follows from Lemma $%
\ref{K_beta}$ that $H_{\beta }(N)=K_{\beta }(N)=g_{\beta }(N)$ with the
function $g_{\beta }(x):\mathbb{R}^{d}\rightarrow \mathbb{R}$ given by
\begin{equation}
g_{\beta }(x)=\sum_{m=0}^{\lfloor k/2\rfloor }\left( -1\right)
^{m}\sum_{\{(i_{1},i_{2}),\dots ,(i_{2m-1},i_{2m})\}\in S_{k}^{m}}x_{%
\widehat{\beta }_{i_{1}\cdots i_{2m}}}\delta _{^{\beta _{i_{1}}\beta
_{i_{2}}}}\cdots \delta _{^{_{\beta _{i_{2m-1}}\beta _{2m}}}},
\label{g-beta}
\end{equation}%
where we used $\delta _{ij}$ to denote the Kronecker symbol (without
confusion with the divergence operator). Notice that
\begin{equation*}
g_\beta(x)= \prod_{i=1}^d H_{k_i}(x_i),
\end{equation*}
where $H_{k_i}$ is the $k_i$th Hermite polynomial and for each $i=1,\dots, d$%
, $k_i$ is the number of components of $\beta$ equal to $i$.
\end{remark}

Let us return to the proof of Theorem \ref{MultiThm0} of estimating the term
$A_{4}$. From (\ref{K-expsn}) we can write
\begin{equation}
A_{4}=E\left[ \left\vert K_{\beta }(F)-g_{\beta }(F)\right\vert \right] \leq
E\left[ |G_{\beta }(F)-g_{\beta }(F)|\right] +E\left[ |T_{\beta }(F)|\right]
\,.  \label{A4}
\end{equation}%
Observe from the expression (\ref{KT}) that $T_{\beta }(F)$ is the sum of
terms of the following form
\begin{equation}
\left( V^{-1}F\right) _{\beta _{i_{1}}\beta _{i_{2}}\cdots \beta
_{i_{s}}}D_{\beta _{k_{1}}}D_{\beta _{k_{2}}}\cdots D_{\beta
_{k_{t}}}(\prod_{i}^{r}M_{^{\beta _{j_{i}}\beta _{l_{i}}}})  \label{KT-term}
\end{equation}%
for some $\left\{ i_{1},\dots ,i_{s},k_{1},\dots ,k_{t},j_{1},l_{1},\dots
j_{r},l_{r}\right\} \subset \left\{ 1,2,\dots ,k\right\} $ and $t\geq 1$.
Applying Lemma \ref{Variance} with (\ref{Hyper}) and (\ref{HyperMeyer}) we
obtain
\begin{equation}
E\left[ \left\vert T_{\beta }(F)\right\vert \right] \leq C\sum_{1\leq l\leq
d}\left\Vert \left\Vert DF_{l}\right\Vert _{\mathfrak{H}%
}^{2}-q_{l}E[F_{l}^{2}]\right\Vert _{2}^{\frac{1}{2}}.  \label{T_k(F)}
\end{equation}%
In order to compare $g_{\beta }(F)$ with $G_{\beta }(F)$, from (\ref{g-beta}%
) we can write $g_{\beta }(F)$ as
\begin{equation*}
g_{\beta }(F)=\sum_{m=0}^{\lfloor k/2\rfloor }\left( -1\right)
^{m}\sum_{\{(i_{1},i_{2}),\dots ,(i_{2m-1},i_{2m})\}\in S_{k}^{m}}F_{%
\widehat{\beta }_{i_{1}\cdots i_{2m}}}\delta _{\beta _{i_{1}}\beta
_{i_{2}}}\cdots \delta _{^{\beta _{i_{2m-1}}\beta _{2m}}}\,.
\end{equation*}%
Then, it follows from hypercontractivity property (\ref{Hyper}) that
\begin{equation*}
E\left[ \left\vert G_{\beta }(F)-g_{\beta }(F)\right\vert \right] \leq
C\left( \left\vert V^{-1}-I\right\vert +\left\Vert M-I\right\Vert
_{2}\right) ,
\end{equation*}%
where the constant $C$ depends on $k,V$ and $Q$. From $V^{-1}-I=V^{-1}\left(
I-V\right) $ we have $\left\vert V^{-1}-I\right\vert \leq C\left\vert
V-I\right\vert $, where $C$ depends on $V$. We also have $M-I=V^{-1}\left(
\gamma _{F}-VQ\right) V^{-1}Q^{-1}+V^{-1}-I$. Then, Lemma \ref{Variance}
implies that
\begin{eqnarray*}
\Vert M-I\Vert _{2} &\leq &C\left( \Vert \gamma _{F}-VQ\Vert
_{2}+|V^{-1}-I|\right) \\
&\leq &C\left( \sum_{1\leq l\leq d}\left\Vert \left\Vert DF_{l}\right\Vert _{%
\mathfrak{H}}^{2}-q_{l}EF_{l}^{2}\right\Vert _{2}+|V-I|\right) ,
\end{eqnarray*}%
where the constant $C$ depends on $k,V$ and $Q$. Therefore
\begin{equation}
E\left[ \left\vert G_{\beta }(F)-g_{\beta }(F)\right\vert \right] \leq
C\left( \left\vert V-I\right\vert +\sum_{1\leq l\leq d}\left\Vert \left\Vert
DF_{l}\right\Vert _{\mathfrak{H}}^{2}-q_{l}EF_{l}^{2}\right\Vert _{2}^{\frac{%
1}{2}}\right) \,.  \label{Gkgk}
\end{equation}%
Combining it with (\ref{T_k(F)}) we obtain from (\ref{A4}) that%
\begin{equation*}
A_{4}\leq C\left( \left\vert V-I\right\vert +\sum_{1\leq l\leq d}\left\Vert
\left\Vert DF_{l}\right\Vert _{\mathfrak{H}}^{2}-q_{l}EF_{l}^{2}\right\Vert
_{2}^{\frac{1}{2}}\right) \,,
\end{equation*}%
where the constant $C$ depends on $d,V,Q$. This completes the estimation of
the term $A_{4}$.
\end{proof}

\subsection{Sobolev norms of $\protect\gamma _{F}^{-1}$}

In this subsection we estimate the Sobolev norms of $\gamma _{F}^{-1}$, the
inverse of the Malliavin matrix $\gamma _{F}$ for a random variable $F$ of
multiple stochastic integrals. We begin with the following estimate on the
variance and Sobolev norms of $\left( \gamma _{F}\right) _{ij}=\left\langle
DF_{i},DF_{j}\right\rangle _{\mathfrak{H}}$, $1\leq i,j\leq d$, following
the approach of \cite{NoNo11,NP12,NPRv10}.

\begin{lemma}
\label{Variance}Let $F=I_{p}(f)\,$\ and $G=I_{q}(g)$ with $f\in \mathfrak{H}%
^{\odot p}$ and $g\in \mathfrak{H}^{\odot q}$ for $p,q\geq 1$. Then for all $%
k\geq 0$ there exists a constant $C_{p,q,k}$ such that%
\begin{eqnarray}
&&\left\Vert D^{k}\left( \left\langle DF,DG\right\rangle _{\mathfrak{H}}-%
\sqrt{pq}E[FG]\right) \right\Vert _{2}  \label{D<>} \\
&\leq &C_{p,q,k}(\left\Vert F\right\Vert _{2}^{2}+\left\Vert G\right\Vert
_{2}^{2})(\left\Vert \left\Vert DF\right\Vert _{\mathfrak{H}}^{2}-pE\left[
F^{2}\right] \right\Vert _{2}^{\frac{1}{2}}+\left\Vert \left\Vert
DG\right\Vert _{\mathfrak{H}}^{2}-pE\left[ G^{2}\right] \right\Vert _{2}^{%
\frac{1}{2}}).  \notag
\end{eqnarray}
\end{lemma}

\begin{proof}
Without lost of generality, we assume $p\leq q$. Applying (\ref{MltFml})
with the fact that $DI_{p}(f)=pI_{p-1}(f)$ we have
\begin{eqnarray}
\left\langle DF,DG\right\rangle _{\mathfrak{H}} &=&pq\left\langle
I_{p-1}(f),I_{q-1}(g)\right\rangle _{\mathfrak{H}}  \label{<>} \\
&=&pq\sum_{r=0}^{p-1}r!\binom{p-1}{r}\binom{q-1}{r}I_{p+q-2-2r}\left( f%
\widetilde{\otimes _{r+1}}g\right)  \notag \\
&=&pq\sum_{r=1}^{p}\left( r-1\right) !\binom{p-1}{r-1}\binom{q-1}{r-1}%
I_{p+q-2r}\left( f\widetilde{\otimes _{r}}g\right) .  \notag
\end{eqnarray}%
Note that $E[FG]=0$ if $p<q$ and $E[FG]=\left\langle f,g\right\rangle _{%
\mathfrak{H}^{\otimes p}}=f\widetilde{\otimes _{p}}g$ if $p=q$. Then%
\begin{equation*}
\left\langle DF,DG\right\rangle _{\mathfrak{H}}-\sqrt{pq}E(FG)=pq%
\sum_{r=1}^{p}(1-\delta _{qr})\left( r-1\right) !\binom{p-1}{r-1}\binom{q-1}{%
r-1}I_{p+q-2r}\left( f\widetilde{\otimes _{r}}g\right) ,
\end{equation*}%
where $\delta _{qr}$ is again the Kronecker symbol. It follows that
\begin{eqnarray}
&&E\left[ \left\langle DF,DG\right\rangle _{\mathfrak{H}}-\sqrt{pq}E[FG]%
\right] ^{2}  \label{Var2} \\
&=&p^{2}q^{2}\sum_{r=1}^{p}(1-\delta _{qr})(r-1)!^{2}\binom{p-1}{r-1}^{2}%
\binom{q-1}{r-1}^{2}(p+q-2r)!\left\Vert f\widetilde{\otimes _{r}}%
g\right\Vert _{\mathfrak{H}^{\otimes (p+q-2r)}}^{2}.  \notag
\end{eqnarray}%
Note that if $r<p\leq q$, then (see also \cite[(6.2.7)]{NP12})%
\begin{eqnarray}
\left\Vert f\widetilde{\otimes _{r}}g\right\Vert _{\mathfrak{H}^{\otimes
(p+q-2r)}}^{2} &\leq &\left\Vert f\otimes _{r}g\right\Vert _{\mathfrak{H}%
^{\otimes (p+q-2r)}}^{2}=\left\langle f\otimes _{p-r}f,g\otimes
_{q-r}g\right\rangle _{\mathfrak{H}^{\otimes 2r}}  \notag \\
&\leq &\frac{1}{2}(\left\Vert f\otimes _{p-r}f\right\Vert _{\mathfrak{H}%
^{\otimes 2r}}^{2}+\left\Vert g\otimes _{q-r}g\right\Vert _{\mathfrak{H}%
^{\otimes 2r}}^{2}),  \label{Var3}
\end{eqnarray}%
and if $r=p<q$,
\begin{equation}
\left\Vert f\widetilde{\otimes _{p}}g\right\Vert _{\mathfrak{H}^{\otimes
(q-p)}}^{2}\leq \left\Vert f\otimes _{p}g\right\Vert _{\mathfrak{H}^{\otimes
(q-p)}}^{2}\leq \left\Vert f\right\Vert _{\mathfrak{H}^{\otimes
p}}^{2}\left\Vert g\otimes _{q-p}g\right\Vert _{\mathfrak{H}^{\otimes 2p}}.
\label{Var4}
\end{equation}%
From (\ref{w1}) and (\ref{w2}) it follows that
\begin{equation}
\left\Vert \left\Vert DF\right\Vert _{\mathfrak{H}}^{2}-pE\left[ F^{2}\right]
\right\Vert _{2}^{2}=p^{4}\sum_{r=1}^{p-1}(r-1)!^{2}\binom{p-1}{r-1}%
^{2}(2p-2r)!\left\Vert f\otimes _{r}f\right\Vert _{\mathfrak{H}^{\otimes
(2p-2r)}}^{2}\text{.}  \label{Var5}
\end{equation}%
Combining (\ref{Var2})--(\ref{Var5}) we obtain%
\begin{eqnarray*}
&&E\left[ \left\langle DF,DG\right\rangle _{\mathfrak{H}}-\sqrt{pq}E\left[ FG%
\right] \right] ^{2} \\
&\leq &C_{p,q}(\left\Vert \left\Vert DF\right\Vert _{\mathfrak{H}}^{2}-pE%
\left[ F^{2}\right] \right\Vert _{2}^{2}+\left\Vert F\right\Vert
_{2}^{2}\left\Vert \left\Vert DG\right\Vert _{\mathfrak{H}}^{2}-pE\left[
G^{2}\right] \right\Vert _{2}).
\end{eqnarray*}%
Then (\ref{D<>}) with $k=0$ follows from $\left\Vert \left\Vert
DF\right\Vert _{\mathfrak{H}}^{2}-pE\left[ F^{2}\right] \right\Vert _{2}\leq
C_{p}\left\Vert F\right\Vert _{2}^{2}$, which is implied by (\ref{HyperMeyer}%
). From (\ref{<>}) we deduce
\begin{equation*}
D^{k}\left\langle DF,DG\right\rangle _{\mathfrak{H}}=pq\sum_{r=1}^{p\wedge
\lbrack (p+q-k)/2]}\left( r-1\right) !\binom{p-1}{r-1}\binom{q-1}{r-1}\frac{%
p+q-2r}{p+q-k-2r}I_{p+q-k-2r}\left( f\widetilde{\otimes _{r}}g\right) .
\end{equation*}%
Then it follows from (\ref{Var3})--(\ref{Var5}) that
\begin{eqnarray*}
&&E\left\Vert D^{k}\left\langle DF,DG\right\rangle _{\mathfrak{H}%
}\right\Vert _{\mathfrak{H}^{\otimes k}}^{2} \\
&=&p^{2}q^{2}\sum_{r=1}^{p\wedge \lbrack (p+q-k)/2]}\left( r-1\right) !^{2}%
\binom{p-1}{r-1}^{2}\binom{q-1}{r-1}^{2}\frac{\left( p+q-2r\right) !^{2}}{%
\left( p+q-k-2r\right) !}\left\Vert f\widetilde{\otimes _{r}}g\right\Vert _{%
\mathfrak{H}^{\otimes (p+q-2r)}}^{2} \\
&\leq &C_{p,q}(\left\Vert \left\Vert DF\right\Vert _{\mathfrak{H}}^{2}-pE%
\left[ F^{2}\right] \right\Vert _{2}^{2}+\left\Vert F\right\Vert
_{2}^{2}\left\Vert \left\Vert DG\right\Vert _{\mathfrak{H}}^{2}-pE\left[
G^{2}\right] \right\Vert _{2}).
\end{eqnarray*}%
This completes the proof.
\end{proof}

The following lemma gives estimates on the Sobolev norms of the entries of $%
\gamma _{F}^{-1}$.

\begin{lemma}
\label{Gma-1_norm}Let $F=(F_{1},\dots ,F_{d})=(I_{q_{1}}(f_{1}),\dots
,I_{q_{d}}(f_{d}))$ be non-degenerate and let $\gamma _{F}=\left(
\left\langle DF_{i},DF_{j}\right\rangle _{\mathfrak{H}}\right) _{1\leq
i,j\leq d}$. Set $V=\left( E[F_{i}F_{j}]\right) _{1\leq i,j\leq d}$. Then
for any real number $p>1$,
\begin{equation}
\left\Vert \gamma _{F}^{-1}\right\Vert _{p}\leq C\left\Vert \left( \det
\gamma _{F}\right) ^{-1}\right\Vert _{2p},  \label{Gma-p}
\end{equation}%
where the constant $C$ depends on $q_{1},\dots ,q_{d},d,p$ and $V$.
Moreover, for any integer $k\geq 1$ and any real number $p>1$%
\begin{equation}
\left\Vert \gamma _{F}^{-1}\right\Vert _{k,p}\leq C\left\Vert \left( \det
\gamma _{F}\right) ^{-1}\right\Vert _{(k+2)2p}^{k+1}\sum_{i=1}^{d}\left\Vert
\left\Vert DF_{i}\right\Vert _{\mathfrak{H}}^{2}-q_{i}E\left[ F_{i}^{2}%
\right] \right\Vert _{2},  \label{Gma-kp}
\end{equation}%
where the constant $C$ depends on $q_{1},\dots ,q_{d},d,p,k$ and $V$.
\end{lemma}

\begin{proof}
Let $\gamma _{F}^{\ast }$ be the adjugate matrix of $\gamma _{F}$. Note that
H\"{o}lder inequality and (\ref{HyperMeyer}) imply
\begin{equation*}
\left\Vert \left\langle DF_{i},DF_{j}\right\rangle _{\mathfrak{H}%
}\right\Vert _{p}\leq \left\Vert DF_{i}\right\Vert _{2p}\left\Vert
DF_{j}\right\Vert _{2p}\leq C_{V,p}
\end{equation*}%
for all $1\leq i,j\leq d$, $p\geq 1$. Applying Holder's inequality we obtain
that the $p$ norm of $\gamma _{F}^{\ast }$ is also bounded by a constant. A
further application of Holder's inequality to $\gamma _{F}^{-1} =\left( \det
\gamma _{F}\right) ^{-1} \gamma _{F}^{\ast } $ yields
\begin{equation}
\left\Vert \gamma _{F}^{-1} \right\Vert _{p}\leq \left\Vert \left( \det
\gamma _{F}\right) ^{-1}\right\Vert _{2p}\left\Vert \ \gamma _{F}^{\ast }
\right\Vert _{2p}\leq C_{ V ,p}\left\Vert \left( \det \gamma _{F}\right)
^{-1}\right\Vert _{2p},  \label{Gma-Lp}
\end{equation}%
which implies (\ref{Gma-p}).

Since $F$ is non-degenerate, then (see \cite[Lemma 2.1.6]{Nu06}) $\left(
\gamma _{F}^{-1}\right) _{^{ij}}$ belongs to $\mathbb{D}^{\infty }$ for all $%
i,j$ and
\begin{equation}
D\left( \gamma _{F}^{-1}\right) _{^{ij}}=-\sum_{m,n=1}^{d}\left( \gamma
_{F}^{-1}\right) _{^{im}}\left( \gamma _{F}^{-1}\right) _{^{nj}}D\left(
\gamma _{F}\right) _{mn}.  \label{DkGm-2}
\end{equation}%
Then, applying H\"{o}lder's inequality we obtain
\begin{eqnarray*}
\left\Vert D\left( \gamma _{F}^{-1}\right) \right\Vert _{p} &\leq
&\left\Vert \gamma _{F}^{-1}\right\Vert _{3p}^{2}\left\Vert D\gamma
_{F}\right\Vert _{3p}. \\
&\leq &C_{V,p}\left\Vert \left( \det \gamma _{F}\right) ^{-1}\right\Vert
_{6p}^{2}\sum_{i=1}^{d}\left\Vert \left\Vert DF_{i}\right\Vert _{\mathfrak{H}%
}^{2}-q_{i}E\left[ F_{i}^{2}\right] \right\Vert _{2},
\end{eqnarray*}%
where in the second inequality we used (\ref{Gma-p}) and
\begin{equation*}
\left\Vert D\gamma _{F}\right\Vert _{3p}\leq C_{V,p}\left\Vert D\gamma
_{F}\right\Vert _{2}\leq C_{V,p}\sum_{i=1}^{d}\left\Vert \left\Vert
DF_{i}\right\Vert _{\mathfrak{H}}^{2}-q_{i}E\left[ F_{i}^{2}\right]
\right\Vert _{2}
\end{equation*}%
for all $p\geq 1$, which follows from (\ref{HyperMeyer}) and (\ref{D<>}).
This implies (\ref{Gma-kp}) with $k\,=1$. For higher order derivatives, (\ref%
{Gma-kp}) follows from repeating the use of (\ref{DkGm-2}), (\ref{HyperMeyer}%
) and (\ref{D<>}).
\end{proof}

The following lemma estimates the difference $\gamma _{F}^{-1}-V^{-1}Q^{-1}$.

\begin{lemma}
Let $F=(F_{1},\dots ,F_{d})=(I_{q_{1}}(f_{1}),\dots ,I_{q_{d}}(f_{d}))$ be a
non-degenerate random vector with $1\leq q_{1}\leq \cdots \leq q_{d}$ and $%
f_{i}\in \mathfrak{H}^{\odot q_{i}}$. Let $\gamma _{F}$ be the Malliavin
matrix of $F$. Recall the notation of $V$ and $Q$ in $(\ref{VQ})$. Then, for
every integer $k\geq 1$ and any real number $p>1$ we have
\begin{equation}
\left\Vert \gamma _{F}^{-1}-V^{-1}Q^{-1}\right\Vert _{k,p}\leq C\left\Vert
\left( \det \gamma _{F}\right) ^{-1}\right\Vert _{(k+2)2p}^{k+1}\sum_{1\leq
l\leq d}\left\Vert \left\Vert DF_{l}\right\Vert _{\mathfrak{H}}^{2}-q_{l}E%
\left[ F_{l}^{2}\right] \right\Vert _{2}^{\frac{1}{2}},  \label{InvMtrx-}
\end{equation}%
where the constant $C$ depends on $d,V,Q,p$ and $k$.
\end{lemma}

\begin{proof}
In view of Lemma \ref{Gma-1_norm}, we only need to consider the case when $%
k=0$ because $V$ and $Q$ are deterministic matrices. Note that
\begin{equation*}
\gamma _{F}^{-1}-V^{-1}Q^{-1}=\ \gamma _{F}^{-1}\ (VQ-\gamma
_{F})V^{-1}Q^{-1}.
\end{equation*}%
Then, applying Holder's inequality we have
\begin{equation*}
\left\Vert \gamma _{F}^{-1}-V^{-1}Q^{-1}\right\Vert _{p}\leq
C_{V,Q}\left\Vert \gamma _{F}^{-1}\right\Vert _{2p}\left\Vert VQ-\gamma
_{F}\right\Vert _{2p}.
\end{equation*}%
Note that (\ref{HyperMeyer}) and (\ref{D<>}) with $k=0$ imply
\begin{equation*}
\left\Vert VQ-\gamma _{F}\right\Vert _{2p}\leq C_{V,Q,p}\left\Vert VQ-\gamma
_{F}\right\Vert _{2}\leq C_{V,Q,p}\sum_{i=1}^{d}\left\Vert \left\Vert
DF_{i}\right\Vert _{\mathfrak{H}}^{2}-q_{i}E\left[ F_{i}^{2}\right]
\right\Vert _{2}^{\frac{1}{2}}.
\end{equation*}%
Then, applying (\ref{Gma-Lp}) we obtain
\begin{equation}
\left\Vert \gamma _{F}^{-1}-V^{-1}Q^{-1}\right\Vert _{p}\leq
C_{d,V,Q,p}\sum_{i=1}^{d}\left\Vert \left\Vert DF_{i}\right\Vert _{\mathfrak{%
H}}^{2}-q_{i}E\left[ F_{i}^{2}\right] \right\Vert _{2}^{\frac{1}{2}}
\label{InvMtrx-1}
\end{equation}%
as desired.
\end{proof}


\subsection{Technical estimates\label{s.stein}}

In this subsection, we study the terms $A_{1}=\left\vert
E[h(F)]-E[h(N)]\right\vert $ in Equation (\ref{V-main}) and $A_{3}=E\left[
\left\vert H_{\beta }(F)-K_{\beta }(F)\right\vert \right] $ in (\ref{KH2}).
For $A_{1}$, we shall use the multivariate Stein's method to give an
estimate for a large class of non-smooth test functions $h$.

\begin{lemma}
\label{SM-prop}Let $h:\mathbb{R}^{d}\rightarrow \mathbb{R}$ be an almost
everywhere continuous function such that $\left\vert h(x)\right\vert \leq
c\left( \left\vert x\right\vert ^{m}+1\right) $ for some $m,c>0$. Let $%
F=(F_{1},\dots ,F_{d})$ be non-degenerate with $E[F_{i}]=0,1\leq i\leq d$
and denote $N\sim N(0,I)$. Then there exists a constant $C_{m,c}$ depending
on $m$ and $c$ such that \
\begin{equation}
\left\vert E[h(F)]-E[h(N)]\right\vert \leq C_{m,c}\left( \left\Vert
F\right\Vert _{2m}^{m}+1\right) \sum_{i,j,k=1}^{d}\left\Vert \delta \left(
A_{ij}\left( \gamma _{F}^{-1}\right) _{jk}DF_{k}\right) \right\Vert _{2},
\label{SM-multi}
\end{equation}%
where $\gamma _{F}^{-1}$ is the inverse of the Malliavin matrix of $F$ and
\begin{equation}
A_{ij}=\delta _{ij}-\left\langle DF_{j},-DL^{-1}F_{i}\right\rangle _{%
\mathfrak{H}}.  \label{A_ij}
\end{equation}
\end{lemma}

\begin{proof}
For $\varepsilon >0$, let
\begin{equation*}
h_{\varepsilon }(x)=(\mathbf{1}_{\left\{ \left\vert \cdot \right\vert <\frac{%
1}{\varepsilon }\right\} }h)\ast \rho _{\varepsilon }(x)=\int_{\mathbb{R}%
^{d}}\mathbf{1}_{\left\vert y\right\vert <\frac{1}{\varepsilon }}h(y)\rho
_{\varepsilon }(x-y)dy.
\end{equation*}%
where $\rho _{\varepsilon }$ is the standard mollifier. That is, $\rho
_{\varepsilon }(x)=\frac{1}{\varepsilon ^{d}}\rho (\frac{x}{\varepsilon }),$
where $\rho (x)=C\mathbf{1}_{\left\{ \left\vert x\right\vert <1\right\}
}\exp (\frac{1}{\left\vert x\right\vert ^{2}-1})$ and the constant $C$ is
such that $\int_{\mathbb{R}^d} \rho (x)dx=1$. Then $h_{\varepsilon }$ is
Lipschitz continuous. Hence, the solution $f_{\varepsilon }$ to the
following \emph{Stein's equation}:
\begin{equation}
\Delta f_{\varepsilon }(x)-\left\langle x,\nabla f_{\varepsilon
}(x)\right\rangle _{\mathbb{R}^{d}}=h_{\varepsilon }(x)-E[h_{\varepsilon
}(N)]\text{ }  \label{Stein-equMulti}
\end{equation}%
exists and its derivative has the following expression \cite[Page 82]{NP12}%
\begin{eqnarray}
\partial _{i}f_{\varepsilon }(x) &=&\frac{\partial }{\partial x_{i}}%
\int_{0}^{1}\frac{1}{2t}E[h_{\varepsilon }(\sqrt{t}x+\sqrt{1-t}N)]dt
\label{di-f} \\
&=&\int_{0}^{1}E[h_{\varepsilon }(\sqrt{t}x+\sqrt{1-t}N)N_{i}]\frac{1}{2%
\sqrt{t}\sqrt{1-t}}dt.  \notag
\end{eqnarray}%
It follows directly from the polynomial growth of $h$ that
\begin{equation}
\left\vert h_{\varepsilon }(x)\right\vert \leq C_{1}\left\vert x\right\vert
^{m}+C_{2}  \label{h_epsilon_bd}
\end{equation}%
for all $\varepsilon <1$, where $C_{1},C_{2}>0$ are two constants depending
on $c$ and $m$. Then, from (\ref{Stein-equMulti}) we can write
\begin{equation*}
\left\vert \partial _{i}f_{\varepsilon }(x)\right\vert \leq C_{1}\left\vert
x\right\vert ^{m}+C_{2},
\end{equation*}%
with two possibly different constants $C_1$, and $C_2$. Hence,
\begin{equation}
\left\Vert \partial _{i}f_{\varepsilon }(F)\right\Vert _{2}\leq
C_{1}\left\Vert F\right\Vert _{2m}^{m}+C_{2}.  \label{f_epsilon}
\end{equation}
Meanwhile, note that for $1\leq i\leq d$,
\begin{eqnarray*}
E[F_{i}\partial _{i}f_{\varepsilon }(F)] &=&E[LL^{-1}F_{i}\partial
_{i}f_{\varepsilon }(F)]=E[\left\langle -DL^{-1}F_{i},D\partial
_{i}f_{\varepsilon }(F)\right\rangle ] \\
&=&\sum_{j=1}^{d}E[\left\langle -DL^{-1}F_{i},\partial _{ij}f_{\varepsilon
}(F)DF_{j}\right\rangle ].
\end{eqnarray*}%
Then, replacing $x$ by $F$ and taking expectation in Equation $(\ref%
{Stein-equMulti})$ yields
\begin{equation}
\left\vert E[h_{\varepsilon }(F)]-E[h_{\varepsilon }(N)]\right\vert
=\left\vert \sum_{i,j=1}^{d}E\left[ \partial _{ij}^{2}f_{\varepsilon
}(F)A_{ij}\right] \right\vert .  \label{h_epsilon}
\end{equation}%
Notice that%
\begin{equation*}
\left\langle DF_{i},D\partial _{i}f_{\varepsilon }(F)\right\rangle _{%
\mathfrak{H}}=\left\langle DF_{i},\sum_{j=1}^{d}\partial
_{ij}^{2}f_{\varepsilon }(F)DF_{j}\right\rangle _{\mathfrak{H}%
}=\sum_{j=1}^{d}\partial _{ij}^{2}f_{\varepsilon }(F)\left\langle
DF_{i},DF_{j}\right\rangle _{\mathfrak{H}}
\end{equation*}%
for all $1\leq i,k\leq d$, which implies
\begin{equation*}
\partial _{ij}f_{\varepsilon }(F)=\sum_{k=1}^{d}\left( \gamma
_{F}^{-1}\right) _{jk}\left\langle DF_{k},D\partial _{i}f_{\varepsilon
}(F)\right\rangle _{\mathfrak{H}},
\end{equation*}%
and hence%
\begin{eqnarray*}
\sum_{i,j=1}^{d}E\left[ \partial _{ij}^{2}f_{\varepsilon }(F)A_{ij}\right]
&=&\sum_{i,j=1}^{d}E\left[ A_{ij}\left\langle \sum_{k=1}^{d}\left( \gamma
_{F}^{-1}\right) _{jk}DF_{k},D\partial _{i}f_{\varepsilon }(F)\right\rangle
_{\mathfrak{H}}\right] \\
&=&\sum_{i,j=1}^{d}E\left[ \partial _{i}f_{\varepsilon }(F)\delta \left(
A_{ij}\sum_{k=1}^{d}\left( \gamma _{F}^{-1}\right) _{jk}DF_{k}\right) \right]%
.
\end{eqnarray*}%
Substituting this expression in (\ref{h_epsilon}) and using (\ref{f_epsilon}%
) we obtain
\begin{eqnarray*}
\left\vert E[h_{\varepsilon }(F)]-E[h_{\varepsilon }(N)]\right\vert
&=&\sum_{i,j,k=1}^{d}E\left[ \partial _{i}f_{\varepsilon }(F)\delta \left(
A_{ij}\left( \gamma _{F}^{-1}\right) _{jk}DF_{k}\right) \right] \\
&\leq &\sum_{i,j,k=1}^{d}\left\Vert \partial _{i}f_{\varepsilon
}(F)\right\Vert _{2}\left\Vert \delta \left( A_{ij}\left( \gamma
_{F}^{-1}\right) _{jk}DF_{k}\right) \right\Vert _{2}\, \\
&\leq &\left( C_{1}\left\Vert F\right\Vert _{2m}^{m}+C_{2}\right)
\sum_{i,j,k=1}^{d}\left\Vert \delta \left( A_{ij}\left( \gamma
_{F}^{-1}\right) _{jk}DF_{k}\right) \right\Vert _{2}.
\end{eqnarray*}%
Then, we can conclude the proof by observing that
\begin{equation*}
\lim_{\varepsilon \rightarrow 0}\left\vert E[h_{\varepsilon
}(F)]-E[h_{\varepsilon }(N)]\right\vert =\left\vert
E[h(F)]-E[h(N)]\right\vert ,
\end{equation*}%
which follows from (\ref{h_epsilon_bd}) and the fact that $h_{\varepsilon
}\rightarrow h$ almost everywhere.
\end{proof}

The next lemma gives an estimate for $\left\Vert \delta \left( A_{ij}\left(
\gamma _{F}^{-1}\right) _{jk}DF_{k}\right) \right\Vert _{2}$ when $F$ is a
vector of multiple stochastic integrals.

\begin{lemma}
\label{A1-V}Let $F=(F_{1},\dots ,F_{d})=(I_{q_{1}}(f_{1}),\dots
,I_{q_{d}}(f_{d}))$, where $f_{i}\in \mathfrak{H}^{\odot q_{i}}$, be
non-degenerate and denote $N\sim N(0,I)$. Recall the notation of $V$ and $Q$
in $(\ref{VQ})$ and $A_{ij}$ in $(\ref{A_ij})$. Then, for all $1\leq
i,j,k\leq d$ we have%
\begin{eqnarray}
\left\Vert \delta \left( A_{ij}\left( \gamma _{F}^{-1}\right)
_{jk}DF_{k}\right) \right\Vert _{2} &\leq &C\left\Vert \left( \det \gamma
_{F}\right) ^{-1}\right\Vert _{12}^{3}  \label{Stein_delta} \\
&&\times \left( \ \left\vert V-I\right\vert +\sum_{i=1}^{d}\left\Vert
\left\Vert DF_{i}\right\Vert _{\mathfrak{H}}^{2}-q_{i}E\left[ F_{i}^{2}%
\right] \right\Vert _{2}^{\frac{1}{2}}\right) ,\,  \notag
\end{eqnarray}%
where the constant $C$ depends on $d,V,Q$.
\end{lemma}

\begin{proof}
Applying Meyer's inequality (\ref{Meyer}) we have
\begin{equation*}
\left\Vert \delta \left( A_{ij}\left( \gamma _{F}^{-1}\right)
_{jk}DF_{k}\right) \right\Vert _{2}\leq \left\Vert A_{ij}\left( \gamma
_{F}^{-1}\right) _{jk}DF_{k}\right\Vert _{2}+\left\Vert D\left( A_{ij}\left(
\gamma _{F}^{-1}\right) _{jk}DF_{k}\right) \right\Vert _{2}.
\end{equation*}%
Applying Holder's inequality and (\ref{HyperMeyer}) we have%
\begin{equation*}
\left\Vert A_{ij}\left( \gamma _{F}^{-1}\right) _{jk}DF_{k}\right\Vert
_{2}\leq \left\Vert A_{ij}\right\Vert _{2}\left\Vert \left( \gamma
_{F}^{-1}\right) _{jk}\right\Vert _{4}\left\Vert DF_{k}\right\Vert _{4}\leq
C_{d,V,Q}\left\Vert A_{ij}\right\Vert _{2}\left\Vert \left( \gamma
_{F}^{-1}\right) _{jk}\right\Vert _{4}.
\end{equation*}%
Similarly, Holder's inequality and (\ref{HyperMeyer}) imply
\begin{eqnarray*}
&&\left\Vert D\left( A_{ij}\left( \gamma _{F}^{-1}\right) _{jk}DF_{k}\right)
\right\Vert _{2}\leq C_{d,V,Q}{\LARGE [}\left\Vert DA_{ij}\right\Vert
_{2}\left\Vert \left( \gamma _{F}^{-1}\right) _{jk}\right\Vert _{4} \\
&&~~~~~+\left\Vert A_{ij}\right\Vert _{2}\left\Vert D\left( \gamma
_{F}^{-1}\right) _{jk}\right\Vert _{4}+\left\Vert A_{ij}\right\Vert
_{2}\left\Vert \left( \gamma _{F}^{-1}\right) _{jk}\right\Vert _{4}{\LARGE ]}%
.
\end{eqnarray*}%
Combining the above inequalities we obtain
\begin{equation*}
\left\Vert \delta \left( A_{ij}\left( \gamma _{F}^{-1}\right)
_{jk}DF_{k}\right) \right\Vert _{2}\leq C_{d,V,Q}\left\Vert
A_{ij}\right\Vert _{1,2}\left\Vert \left( \gamma _{F}^{-1}\right)
_{jk}\right\Vert _{1,4}.
\end{equation*}%
Note that
\begin{equation*}
A_{ij}=\delta _{ij}-\left\langle DF_{j},-DL^{-1}F_{i}\right\rangle _{%
\mathfrak{H}}=\delta _{ij}-V_{ij}+V_{ij}-\frac{1}{q_{i}}\left\langle
DF_{j},-DF_{i}\right\rangle _{\mathfrak{H}}\text{. }
\end{equation*}%
Then, it follows from Lemma \ref{Variance} that
\begin{equation*}
\left\Vert A_{ij}\right\Vert _{1,2}\leq C_{d,V,Q}\left( \left\vert
V-I\right\vert +\sum_{i=1}^{d}\left\Vert \left\Vert DF_{i}\right\Vert _{%
\mathfrak{H}}^{2}-q_{i}E\left[ F_{i}^{2}\right] \right\Vert _{2}^{\frac{1}{2}%
}\right) .
\end{equation*}%
Then, the lemma follows by taking into account of (\ref{Gma-kp}) with $k=1$.
\end{proof}

As a consequence of the above two lemmas, we have the following result.

\begin{proposition}
\label{Stein-prop}Let $h:\mathbb{R}^{d}\rightarrow \mathbb{R}$ be an almost
everywhere continuous function such that $\left\vert h(x)\right\vert \leq
c\left( \left\vert x\right\vert ^{m}+1\right) $ for some $m,c>0$. Let $%
F=(F_{1},\dots ,F_{d})=(I_{q_{1}}(f_{1}),\dots ,I_{q_{d}}(f_{d}))$, where $%
f_{i}\in \mathfrak{H}^{\odot q_{i}}$, be non-degenerate and denote $N\sim
N(0,I)$. Recall the notation of $V$ and $Q$ in $(\ref{VQ})$. Then \
\begin{eqnarray}
\left\vert E[h(F)]-E[h(N)]\right\vert &\leq &C\left\Vert \left( \det \gamma
_{F}\right) ^{-1}\right\Vert _{12}^{3}  \label{SM_inq} \\
&&\times \left( \left\vert V-I\right\vert +\sum_{i=1}^{d}\left\Vert
\left\Vert DF_{i}\right\Vert _{\mathfrak{H}}^{2}-q_{i}E\left[ F_{i}^{2}%
\right] \right\Vert _{2}^{\frac{1}{2}}\right) ,  \notag
\end{eqnarray}%
where the constant $C$ depends on $d,V,Q,m,c$.
\end{proposition}

In the following, we estimate the term $A_{3}=E\left[ \left\vert H_{\beta
}(F)-K_{\beta }(F)\right\vert \right] $ in (\ref{KH2}), where $H_{\beta }(F)$
and $K_{\beta }(F)$ are defined in (\ref{H1}) and (\ref{K}), respectively.

\begin{lemma}
\label{HKLemma}Let $F=(F_{1},\dots ,F_{d})=(I_{q_{1}}(f_{1}),\dots
,I_{q_{d}}(f_{d}))$ be non-degenerate. Let $\beta =(\beta _{1},\dots ,\beta
_{k})$ be a multi-index of length $k\geq 1$. Let $H_{\beta }(F)$ and $%
K_{\beta }(F)$ be defined by $(\ref{H1})$ and $(\ref{K})$, respectively.
Then there exists a constant $C$ depending on $d,V,Q,k$ such that
\begin{eqnarray}
E\left[ \left\vert H_{\beta }(F)-K_{\beta }(F)\right\vert \right] &\leq
&C\left\Vert \left( \det \gamma _{F}\right) ^{-1}\right\Vert
_{(k+4)2^{k+3}}^{k\left( k+2\right) }  \label{HK0} \\
&&\times \sum_{i=1}^{d}\left\Vert \left\Vert DF_{i}\right\Vert _{\mathfrak{H}%
}^{2}-q_{i}E\left[ F_{i}^{2}\right] \right\Vert _{2}^{\frac{1}{2}}.  \notag
\end{eqnarray}
\end{lemma}

\begin{proof}
To simplify notation, we write $H_{\beta }$ and $K_{\beta }$ for $H_{\beta
}(F)$ and $K_{\beta }(F)$, respectively. From (\ref{H1}) and (\ref{K}) we
see that
\begin{equation*}
H_{\beta }-K_{\beta }=\delta \left( H_{\widehat{\beta }_{k}}\left( \gamma
_{F}^{-1}DF\right) _{^{\beta _{k}}}-K_{\widehat{\beta }_{k}}\left( \left(
VQ\right) ^{-1}DF\right) _{\beta _{k}}\right) ,
\end{equation*}%
where $\widehat{\beta }_{k}=(\beta _{1},\dots ,\beta _{k-1})$. For any $%
s\geq 0,p>1$, using Meyer's inequality (\ref{Meyer}) we obtain
\begin{eqnarray*}
&&\left\Vert H_{\beta }-K_{\beta }\right\Vert _{s,p} \\
&\leq &C_{s,p}\left\Vert H_{\widehat{\beta }_{k}}\left( \gamma
_{F}^{-1}DF\right) _{\beta _{k}}-K_{\widehat{\beta }_{k}}\left( \left(
VQ\right) ^{-1}DF\right) _{\beta _{k}}\right\Vert _{s+1,p} \\
&\leq &C_{s,p}\left\Vert \left( H_{\widehat{\beta }_{k}}-K_{\widehat{\beta }%
_{k}}\right) \left( \left( VQ\right) ^{-1}DF\right) _{\beta _{k}}\right\Vert
_{s+1,p} \\
&&+C_{s,p}\left\Vert H_{\widehat{\beta }_{k}}\left( \left( \gamma
_{F}^{-1}-\left( VQ\right) ^{-1}\right) DF\right) _{\beta _{k}}\right\Vert
_{s+1,p}.
\end{eqnarray*}%
Then, H\"{o}lder's inequality yields
\begin{eqnarray*}
&&\left\Vert H_{\beta }-K_{\beta }\right\Vert _{s,p} \\
&\leq &\left\Vert H_{\widehat{\beta }_{k}}-K_{\widehat{\beta }%
_{k}}\right\Vert _{s+1,2p}\left\Vert \left( \left( VQ\right) ^{-1}DF\right)
_{\beta _{k}}\right\Vert _{s+1,2p} \\
&&+\left\Vert H_{\widehat{\beta }_{k}}\right\Vert _{s+1,2p}\left\Vert \left(
\left( \gamma _{F}^{-1}-\left( VQ\right) ^{-1}\right) DF\right) _{\beta
_{k}}\right\Vert _{s+1,2p}.
\end{eqnarray*}%
Note that (\ref{HyperMeyer}) implies $\left\Vert \left( \left( VQ\right)
^{-1}DF\right) _{\beta _{k}}\right\Vert _{s+1,2p}\leq C_{d,V,Q,s,p}$. Also
note that (\ref{HyperMeyer}), H\"{o}lder's inequality and (\ref{InvMtrx-})
indicate
\begin{equation*}
\left\Vert \left( \left( \gamma _{F}^{-1}-\left( VQ\right) ^{-1}\right)
DF\right) _{\beta _{k}}\right\Vert _{s+1,2p}\leq C_{d,V,Q,s,p}\Delta
\left\Vert \left( \det \gamma _{F}\right) ^{-1}\right\Vert _{(s+3)8p}^{s+2}.
\end{equation*}%
where we denote
\begin{equation*}
\Delta :=\sum_{1\leq l\leq d}\left\Vert \left\Vert DF_{l}\right\Vert _{%
\mathfrak{H}}^{2}-q_{l}E\left[ F_{l}^{2}\right] \right\Vert _{2}^{\frac{1}{2}%
}
\end{equation*}%
to simplify notation. Thus we obtain
\begin{eqnarray}
\left\Vert H_{\beta }-K_{\beta }\right\Vert _{s,p} &\leq
&C_{d,V,Q,s,p}\left\Vert H_{\widehat{\beta }_{k}}-K_{\widehat{\beta }%
_{k}}\right\Vert _{s+1,2p}  \label{H-K_sp} \\
&&+C_{d,V,Q,s,p}\Delta \left\Vert H_{\widehat{\beta }_{k}}\right\Vert
_{s+1,2p}\left\Vert \left( \det \gamma _{F}\right) ^{-1}\right\Vert
_{(s+3)8p}^{s+2}.  \notag
\end{eqnarray}%
Similarly, from Meyer's inequality (\ref{Meyer}), H\"{o}lder's inequality
and (\ref{HyperMeyer}) we obtain by iteration
\begin{eqnarray}
\left\Vert H_{\beta }\right\Vert _{s,p} &\leq &C_{s,p}\left\Vert H_{\widehat{%
\beta }_{k}}\left( \gamma _{F}^{-1}DF\right) _{\beta _{k}}\right\Vert
_{s+1,p}  \notag \\
&\leq &C_{d,V,Q,s,p}\left\Vert H_{\widehat{\beta }_{k}}\right\Vert
_{s+1,2p}\left\Vert \left( \det \gamma _{F}\right) ^{-1}\right\Vert
_{(s+3)8p}^{s+2}  \notag \\
&&\cdots  \notag \\
&\leq &C_{d,V,Q,s,p,k}\left\Vert \left( \det \gamma _{F}\right)
^{-1}\right\Vert _{(s+k+1)2^{k+2}p}^{k\left( s+k\right) }.  \label{H_beta_sp}
\end{eqnarray}%
Applying (\ref{H_beta_sp}) into (\ref{H-K_sp}) and by iteration we can
obtain
\begin{equation}
\left\Vert H_{\beta }-K_{\beta }\right\Vert _{s,p}\leq
C_{d,V,Q,s,p,k}\left\Vert \left( \det \gamma _{F}\right) ^{-1}\right\Vert
_{(2s+k+4)2^{k+2}p}^{k\left( 2s+k+2\right) }\Delta .  \notag
\end{equation}%
Now (\ref{HK0}) follows by taking $s=0$, $p=2$ in the above inequality.
\end{proof}

\section{Uniform estimates for densities of general random variables\label%
{GenrlConv}}

\label{SectionGen}

In this section, we study the uniform convergence of densities for general
random variables. We first characterize the convergence of densities with
quantitative bounds for a sequence of centered random variables, using the
density formula (\ref{Fmla3}). In the second part of this section, a short
proof of the uniform convergence of densities (without quantitative bounds)
is given, using a compactness argument based on the assumption that the
sequence converges in law.

\subsection{Convergence of densities with quantitative bounds}

In this subsection, we estimate the rate of uniform convergence for
densities of general random variables. The idea is to use the density
formula (\ref{Fmla3}).

We use the following notations throughout this section.
\begin{equation*}
\bar{w}=\left\langle DF,-DL^{-1}F\right\rangle _{\mathfrak{H}},~\bar{u}=-%
\bar{w}^{-1}DL^{-1}F.
\end{equation*}

The following technical lemma is useful.

\begin{lemma}
Let $F\in \mathbb{D}^{2,s}$ with $s\geq 4$ such that $E\left[ F\right] =0$
and $E[F^{2}]=\sigma ^{2}$. Let $m$ be the largest even integer less than or
equal to $\frac{s}{2}$. Then there is a positive constant $C_{m}$ such that
for any $t\leq m$,
\begin{equation}
\left\Vert \bar{w}-\sigma ^{2}\right\Vert _{t}\leq \left\Vert \bar{w}-\sigma
^{2}\right\Vert _{m}\leq C_{m}\left\Vert D\bar{w}\right\Vert _{m}\leq
C_{m}\left\Vert D\bar{w}\right\Vert _{s/2}\,.  \label{GDw}
\end{equation}
\end{lemma}

\begin{proof}
It suffices to show the above second inequality. From the integration by
parts formula in Malliavin calculus it follows
\begin{equation*}
\sigma ^{2}=E[F^{2}]=E\left[ \left\langle DF,-DL^{-1}F\right\rangle _{%
\mathfrak{H}}\right] =E\left[ \bar{w}\right] \,.
\end{equation*}%
Note that from (\ref{rci}) and (\ref{Dw-}) we have $\bar{w}\in \mathbb{D}^{1,%
\frac{s}{2}}$. Then the lemma follows from the following
infinite-dimensional Poincar\'{e} inequality \cite[Lemma 5.3.8]{NP12}:
\begin{equation*}
E[(G-E\left[ G\right] )^{m}]\leq \left( m-1\right) ^{m/2}E\left[ \left\Vert
DG\right\Vert _{\mathfrak{H}}^{m}\right] ,
\end{equation*}%
for any even integer $m$ and $G\in \mathbb{D}^{1,m}$.
\end{proof}

The next theorem gives a bound for the uniform distance between the density
of a random variable $F$ and the normal density.

\begin{theorem}
\label{general-Rate}Let $F\in \mathbb{D}^{2,s}$ with $s\geq 8$ such that $E%
\left[ F\right] =0$, $E[F^{2}]=\sigma ^{2}$. Suppose $\,M^{r}:=E\left[
\left\vert \bar{w}\right\vert ^{-r}\right] <\infty $, where $\bar{w}%
=\left\langle DF,-DL^{-1}F\right\rangle _{\mathfrak{H}}$ and $r>2$. Assume $%
\frac{2}{r}+\frac{4}{s}=1$.
Then $F$ admits a density $f_{F}(x)$ and there is a constant $C_{r,s,\sigma
,M}$ depending on $r,s,\sigma $ and $M$ such that
\begin{equation}
\sup_{x\in \mathbb{R}}\left\vert f_{F}(x)-\phi (x)\right\vert \leq
C_{r,s,\sigma ,M}\left\Vert F\right\Vert _{1,s}^2\left\Vert \left\Vert
D^{2}F\right\Vert _{op}\right\Vert _{0,s},  \label{GRate}
\end{equation}%
where $\phi (x)$ is the density of $N\sim N(0,\sigma ^{2})$ and $\left\Vert
D^{2}F\right\Vert _{op}$ indicates the operator norm of $D^{2}F$ introduced
in $(\ref{rci})$.
\end{theorem}

\begin{proof}
It follows from Proposition \ref{density2} that $F$ admits a density given
by $f_{F}(x)=E\left[ \mathbf{1}_{\left\{ F>x\right\} }\delta \left( \bar{u}%
\right) \right] $. 
Then
\begin{equation}
\sup_{x\in \mathbb{R}}\left\vert f_{F}(x)-\phi (x)\right\vert =\sup_{x\in
\mathbb{R}}\left\vert E\left[ \mathbf{1}_{\left\{ F>x\right\} }\delta \left(
\bar{u}\right) \right] -\sigma ^{-2}E[\mathbf{1}_{\left\{ N>x\right\}
}N]\right\vert .  \label{GRate0}
\end{equation}%
Note that, from (\ref{factorOut})
\begin{equation*}
\delta (\bar{u})=\delta (-DL^{-1}F\bar{w}^{-1})=F\bar{w}^{-1}+\left\langle D%
\bar{w}^{-1},DL^{-1}F\right\rangle _{\mathfrak{H}}.
\end{equation*}%
Then
\begin{eqnarray}
&&\left\vert E\left[ \sigma ^{2}\mathbf{1}_{\left\{ F>x\right\} }\delta
\left( \bar{u}\right) \right] -E[\mathbf{1}_{\left\{ N>x\right\}
}N]\right\vert  \notag \\
&\leq &E\left[ \left\vert F\bar{w}^{-1}(\sigma ^{2}-\bar{w})\right\vert %
\right] +\sigma ^{2}E\left[ \left\vert \left\langle D\bar{w}%
^{-1},DL^{-1}F\right\rangle _{\mathfrak{H}}\right\vert \right]  \notag \\
&&+\left\vert E\left[ F\mathbf{1}_{\left\{ F>x\right\} }-N\mathbf{1}%
_{\left\{ N>x\right\} }\right] \right\vert .  \label{GRate1}
\end{eqnarray}%
Note that for $t=\left( \frac{1}{r}+\frac{3}{s}\right) ^{-1}$, we have $%
\frac{s}{2}-t\geq 2$, so there exists an even integer $m\in \lbrack t,\frac{s%
}{2}]$. Also, we have $\frac{1}{r}+\frac{1}{s}+\frac{1}{t}=1$. Then, we can
apply H\"{o}lder's inequality and (\ref{GDw}) to obtain
\begin{eqnarray}
E\left[ \left\vert F\bar{w}^{-1}(\bar{w}-\sigma ^{2})\right\vert \right]
&\leq &\left\Vert F\right\Vert _{s}\left\Vert \bar{w}^{-1}\right\Vert
_{r}\left\Vert \bar{w}-\sigma ^{2}\right\Vert _{t}  \notag \\
&\leq &C_{r,s}\left\Vert F\right\Vert _{s}\left\Vert \bar{w}^{-1}\right\Vert
_{r}\left\Vert D\bar{w}\right\Vert _{s/2}\,.  \label{three-1}
\end{eqnarray}
Meanwhile, applying H\"{o}lder's inequality and (\ref{NP1}) we have%
\begin{eqnarray}
E\left[ \left\vert \bar{w}^{-2}\left\langle D\bar{w},-DL^{-1}F\right\rangle
_{\mathfrak{H}}\right\vert \right] &\leq &\left\Vert \bar{w}^{-1}\right\Vert
_{r}^{2}\left\Vert D\bar{w}\right\Vert _{\frac{s}{2}}\left\Vert
DL^{-1}F\right\Vert _{\frac{s}{2}}  \notag \\
&\leq &\left\Vert \bar{w}^{-1}\right\Vert _{r}^{2}\left\Vert D\bar{w}%
\right\Vert _{\frac{s}{2}}\left\Vert DF\right\Vert _{s}.  \label{three-2}
\end{eqnarray}%
Also, applying Lemma \ref{MS-ctrl} for $h(y)=$ $y\mathbf{1}_{\{y>x\}}$ and (%
\ref{GDw}) we have%
\begin{equation}
\left\vert E\left[ F\mathbf{1}_{F>x}-N\mathbf{1}_{N>x}\right] \right\vert
\leq C_{\sigma }\left\Vert \sigma ^{2}-\bar{w}\right\Vert _{2}\leq C_{\sigma
}\left\Vert D\bar{w}\right\Vert _{s/2}\,.  \label{three-3}
\end{equation}
Applying the estimates (\ref{three-1})-(\ref{three-3}) to (\ref{GRate1}) we
have
\begin{equation}
\left\vert E\left[ \sigma ^{2}\mathbf{1}_{F>x}\delta \left( \bar{u}\right) %
\right] -E[\mathbf{1}_{N>x}N]\right\vert \leq C_{r,s,\sigma ,M}\left\Vert
F\right\Vert _{1,s}\left\Vert D\bar{w}\right\Vert _{s/2}.  \label{GRateCore}
\end{equation}%
Combining (\ref{GRate0}), (\ref{GRateCore}) and (\ref{Dw-}) one gets%
\begin{equation*}
\sup_{x\in \mathbb{R}}\left\vert f_{F}(x)-\phi (x)\right\vert \leq
C_{r,s,\sigma ,M}\left\Vert F\right\Vert _{1,s}^2\left\Vert \left\Vert
D^{2}F\right\Vert _{op}\right\Vert _{s}.
\end{equation*}%
This completes the proof.
\end{proof}

\begin{corollary}
\label{general} Let $\left\{ F_{n}\right\} _{n\in \mathbb{N}}\subset \mathbb{%
D}^{2,s}$ with $s\geq 8$ such that $E\left[ F_{n}\right] =0$ and $%
\lim_{n\rightarrow \infty }E[F_{n}^{2}]=\sigma ^{2}$. Assume $%
E[F_{n}^{2}]\geq \delta>0$ for all $n$. For $r>2$ such that $\frac{2}{r}+%
\frac{4}{s}=1$, assume

\begin{itemize}
\item[(i)] $M_{1}=\sup_{n}\left\Vert F_{n}\right\Vert _{1,s}<\infty .$

\item[(ii)] $M_{2}=\sup_{n}E\left\vert \left\langle
DF_{n},-DL^{-1}F_{n}\right\rangle _{\mathfrak{H}}\right\vert ^{-r}<\infty .$

\item[(iii)] $E\left\Vert D^{2}F_{n}\right\Vert _{op}^{s}\rightarrow 0$ as $%
n\rightarrow \infty $.
\end{itemize}

\noindent Then each $F_{n}$ admits a density $f_{F_{n}}(x)$ and ,
\begin{equation}
\sup_{x\in \mathbb{R}}\left\vert f_{F_{n}}(x)-\phi (x)\right\vert \leq
C\left( \left\Vert \left\Vert D^{2}F_{n}\right\Vert _{op}\right\Vert
_{s}+\left\vert E[F_{n}^{2}]-\sigma ^{2}\right\vert \right) ,  \label{gen1}
\end{equation}%
where the constant $C$ depends on $\sigma ,M_{1},M_{2}$ and $\delta $.
Moreover, if $\,M_{3}=\sup_{n}\left\Vert F_{n}\right\Vert _{2s}<\infty $,
then for any $k\geq 1$ and $\alpha \in (\frac{1}{2},k)$,%
\begin{equation*}
\left\Vert f_{F_{n}}-\phi \right\Vert _{L^{k}(\mathbb{R})}\leq C\left(
\left\Vert \left\Vert D^{2}F_{n}\right\Vert _{op}\right\Vert _{s}+\left\vert
E[F_{n}^{2}]-\sigma ^{2}\right\vert \right) ^{\frac{k-\alpha }{k}},
\end{equation*}%
where the constant $C$ depends on $\sigma ,M_{1},M_{2},M_{3},\alpha $ and $%
\delta $.
\end{corollary}

\begin{remark}
By the \textquotedblleft random contraction inequality\textquotedblright\ $(%
\ref{rci})$, a sufficient condition for $\mathrm{(iii)}$ is $E\left\Vert
D^{2}F_{n}\otimes _{1}D^{2}F_{n}\right\Vert _{\mathfrak{H}^{\otimes
2}}^{s/2}\rightarrow 0$ or $E\left\Vert D^{2}F_{n}\right\Vert _{\mathfrak{H}%
^{\otimes 2}}^{s}\rightarrow 0$.
\end{remark}

\begin{proof}[Proof of Corollary $\protect\ref{general}$]
It follows from Theorem \ref{general-Rate} and Proposition \ref{density2}
with an argument similar to Corollary \ref{qConv}.
\end{proof}

\subsection{Compactness argument}

In general, convergence in law does not imply convergence of the
corresponding densities even if they exist. The following theorem specifies
some additional conditions which ensure that convergence in law will imply
convergence of densities.

\begin{theorem}
\label{gThm}Let $\left\{ F_{n}\right\} _{n\in \mathbb{N}}$ be a sequence of
random variables in $\mathbb{D}^{2,s}$ satisfying any one of the following
two conditions:
\begin{equation}
\sup_{n}\left\Vert F_{n}\right\Vert _{2,s}+\sup_{n}\left\Vert
F_{n}\right\Vert _{2p}+\sup_{n}\left\Vert \left\Vert DF_{n}\right\Vert _{%
\mathfrak{H}}^{-2}\right\Vert _{r}<\infty  \label{gThm1}
\end{equation}%
for some $p, r,s>1$ satisfying $\frac 1p+ \frac{1}{r}+\frac{1}{s}=1$,{\ or }
\begin{equation}
\sup_{n}\left\Vert F_{n}\right\Vert _{2,s} +\sup_{n}\left\Vert \left\vert
\left\langle DF_{n},-DL^{-1}F_{n}\right\rangle _{\mathfrak{H}}\right\vert
^{-1}\right\Vert _{r}<\infty  \label{GThm2}
\end{equation}%
for some $r,s>1$ satisfying $\frac{2}{r}+\frac{4}{s}=1$.

Suppose in addition that $F_{n}\rightarrow N\sim N(0,\sigma ^{2})$ in law.
Then each $F_{n}$ admits a density $f_{F_{n}}\in C(\mathbb{R})$ given by
either $(\ref{Fmla1})$ or $(\ref{Fmla3})$, and
\begin{equation*}
\sup_{x\in \mathbb{R}}\left\vert f_{F_{n}}(x)-\phi (x)\right\vert
\rightarrow 0
\end{equation*}%
as $n\rightarrow \infty $, where $\phi $ is the density of $N$.
\end{theorem}

\begin{proof}
We assume (\ref{gThm1}). The other condition can be treated identically.
From Theorem \ref{density} it follows that the density formula (\ref{Fmla1})
holds for each $n$ and for all $x,y\in \mathbb{R}$
\begin{equation*}
|f_{F_{n}}(x)|\leq C(1\wedge x^{-2}),
\end{equation*}%
\begin{equation*}
|f_{F_{n}}(x)-f_{F_{n}}(y)|\leq C\left\vert x-y\right\vert ^{\frac{1}{p}}.
\end{equation*}%
Hence the sequence $\left\{ f_{F_{n}}\right\} \subset C(\mathbb{R})$ is
uniformly bounded and equi-continuous. Then applying Azel\`{a}-Ascoli
theorem, we obtain a subsequence $\{f_{F_{n_{k}}}\}$ which converges
uniformly to a continuous function $f$ on $\mathbb{R}$ such that $0\leq
f(x)\leq C(1\wedge x^{-2})$. Then $f_{F_{n_{k}}}\rightarrow f$ in $L^{1}(%
\mathbb{R})$ \thinspace as $k\rightarrow \infty $ with $\left\Vert
f\right\Vert _{L^{1}(\mathbb{R})}=\lim_{k}\left\Vert
f_{F_{n_{k}}}\right\Vert _{L^{1}(\mathbb{R})}=1$. This implies that $f$ is a
density function. Then $f$ must be $\phi $ because $F_{n}$ converges to $N$
in law. Since the limit is unique for any subsequence, we get the uniform
convergence of $f_{F_{n}}$ to $\phi $.
\end{proof}

\begin{corollary}
Let $\left\{ F_{n}\right\} _{n\in \mathbb{N}}$ be a sequence of centered
random variables in $\mathbb{D}^{2,4}$ with the following Wiener chaos
expansions: $F_{n}=\sum_{q=1}^{\infty }J_{q}F_{n}$. Suppose that

\begin{itemize}
\item[(i)] $\lim_{Q\rightarrow \infty }\lim \sup_{n\rightarrow \infty
}\sum_{q=Q+1}^{\infty }E[\left\vert J_{q}F_{n}\right\vert ^{2}]=0.$

\item[(ii)] for every $q\geq 1$, $\lim_{n\rightarrow \infty }E[\left(
J_{q}F_{n}\right) ^{2}]=\sigma _{q}^{2}$.

\item[(iii)] $\sum_{q=1}^{\infty }\sigma _{q}^{2}=\sigma ^{2}$.

\item[(iv)] for all $q\geq 1$, $\left\langle D\left( J_{q}F_{n}\right)
,D(J_{q}F_{n})\right\rangle _{\mathfrak{H}}\longrightarrow q\sigma _{q}^{2}$%
, in $L^{2}(\Omega )$ as $n\rightarrow \infty .$

\item[(v)] $\sup_{n}\left\Vert F_{n}\right\Vert _{2,4}+\sup_{n}E[\left\Vert
DF_{n}\right\Vert _{\mathfrak{H}}^{-8}]<\infty $.
\end{itemize}

Then each $F_{n}$ admits a density $f_{F_{n}}(x)$ and
\begin{equation*}
\sup_{x\in \mathbb{R}}\left\vert f_{F_{n}}(x)-\phi (x)\right\vert
\rightarrow 0
\end{equation*}%
as $n\rightarrow \infty $, where $\phi $ is the density of $N(0,\sigma ^{2})$%
.
\end{corollary}

\begin{proof}
It has been proved by Nualart and Ortiz-Latorre in \cite[Theorem 8]{NuOL08}
that under conditions (i)--(iv), $F_{n}$ converges to $N\sim N(0,\sigma
^{2}) $ in law. The condition (v) implies (\ref{gThm1}) with $s=4,p=2,r=4$.
Then we can conclude from Theorem \ref{gThm}.
\end{proof}

\section{Applications\label{Application}}

The main difficulty in applying Theorem \ref{qRateThm} or Theorem \ref%
{MultiThm0} is the verification of the non-degeneracy condition of the
Malliavin matrix:
$\sup_{n}E[\left\Vert DF_{n}\right\Vert _{\mathfrak{H}}^{-p}]<\infty $ or $%
\sup_{n}E[\left\vert \det \gamma _{F_{n}}\right\vert ^{-p}]<\infty $,
respectively. In this section we consider the particular case of random
variables in the second Wiener chaos and we find sufficient conditions for $%
\sup_{n}E[\left\Vert DF_{n}\right\Vert _{\mathfrak{H}}^{-p}]<\infty $. As an
application we consider the problem of estimating the drift parameter in an
Ornstein-Uhlenbeck process.

A general approach to verify $E[G^{-p}]<\infty $ for some positive random
variable and for some $p\geq 1$ is to obtain a small ball probability
estimate of the form
\begin{equation}
P(G\leq \varepsilon )\leq C\varepsilon ^{\alpha }\quad \text{for some }%
\alpha >p\text{ and for all }\varepsilon \in (0,\varepsilon _{0}),
\label{small-pr}
\end{equation}%
where $\varepsilon _{0}>0$ and $C>0$ is a constant that may depend on $%
\varepsilon _{0}$ and $\alpha $. We refer to the paper by Li and Shao \cite%
{LS01} for a survey on this topic. However, finding upper bounds of this
type is a challenging topic, and the application of small ball probabilities
to Malliavin calculus is still an under-explored domain.

\subsection{Random variables in the second Wiener chaos}

\label{s.7.1}

A random variable $F$ in the second Wiener chaos can always be written as $%
F=I_{2}(f)$ where $f\in \mathfrak{H}^{\odot 2}$. Without loss of generality
we can assume that
\begin{equation}
f=\sum_{i=1}^{\infty }\lambda _{i}e_{i}\otimes e_{i},  \label{f-expansion}
\end{equation}%
where $\left\{ \lambda _{i},i\geq 1\right\} $ verifying $\left\vert \lambda
_{1}\right\vert \geq \left\vert \lambda _{2}\right\vert \geq \dots \geq
\left\vert \lambda _{n}\right\vert \geq \dots $ are the eigenvalues of the
Hilbert-Schmidt operator corresponding to $f$ and $\left\{ e_{i},i\geq
1\right\} $ are the corresponding eigenvectors forming an orthonormal basis
of $\mathfrak{H}$. Then, we have $F=I_{2}(f)=\sum_{i=1}^{\infty }\lambda
_{i}(I_{1}(e_{i})^{2}-1)$,
\begin{equation}
DF=2\sum_{i=1}^{\infty }\lambda _{i}I_{1}(e_{i})e_{i}  \label{DF-expsn}
\end{equation}%
and
\begin{equation}
\left\Vert DF\right\Vert _{\mathcal{\mathfrak{H}}}^{2}=4\sum_{i=1}^{\infty
}\lambda _{i}^{2}I_{1}(e_{i})^{2}.  \label{DF-expsn2}
\end{equation}

For random variables of the form in (\ref{DF-expsn2}), i.e., $G=\left(
\sum_{i=1}^{\infty }\lambda _{i}^{2}X_{i}^{2}\right) ^{\frac{1}{2}}$,
Hoffmann-J\o rgensen, Shepp and Dudley \cite{HSD79} used the volume of the
small ball $B_{n}(0,\varepsilon )$ (the $\mathbb{R}^{n}$ ball centered at $0$
with radius $\varepsilon )$ to control $P(G\leq \varepsilon )$ as
\begin{equation}
P(G\leq \varepsilon )\leq P(\sum_{i=1}^{n}\lambda _{i}^{2}X_{i}^{2}\leq
\varepsilon ^{2})\leq (2\pi )^{-\frac{n}{2}}\varepsilon ^{n}\left\vert
B_{n}(0,1)\right\vert \prod_{i=1}^{n}\lambda _{i}^{-1}.  \label{smb-ctrl}
\end{equation}%
They proved that $P(G\leq \varepsilon )$ converges to zero at the rate $%
O(\varepsilon ^{n})$ for all $n$ as $\varepsilon \rightarrow 0$, under some
implicit conditions on $\left\{ \lambda _{i},i\geq 1\right\} $. This idea
can be used here to prove inequality (\ref{neg-mmt}) in the following lemma.
However, our case is much simpler, and we shall use the Gamma function to give
an alternative proof which leads to a necessary and sufficient condition for
$E[G^{-p}]<\infty $.

\begin{lemma}
\label{neg-moment}Let $G=\left( \sum_{i=1}^{\infty }\lambda
_{i}^{2}X_{i}^{2}\right) ^{\frac{1}{2}}$, where $\left\{ \lambda
_{i}\right\} _{i\geq 1}$ satisfies $\left\vert \lambda _{i}\right\vert \geq
\left\vert \lambda _{i+1}\right\vert $ for all $i\geq 1$ and $\left\{
X_{i}\right\} _{i\geq 1}$ are i.i.d standard normal. Fix an $\alpha >1$.
Then, $E[G^{-2\alpha }]<\infty $ if and only if there exists an integer $%
N>2\alpha $ such that $\left\vert \lambda _{N}\right\vert >0$ and in this
case there exists a constant $C_{\alpha }$ depending only on $\alpha $ such
that
\begin{equation}
E[G^{-2\alpha }]\leq C_{\alpha }N^{-\alpha }|\lambda _{N}|^{-2\alpha }.
\label{neg-mmt}
\end{equation}
\end{lemma}

\begin{proof}
Notice $\lambda ^{-\alpha }=\frac{1}{\Gamma (\alpha )}\int_{0}^{\infty
}e^{-\lambda y}y^{\alpha -1}dy$ and $E[e^{-tX_{i}^{2}}] =\frac{1}{\sqrt{1+2t}%
}$ for all $t>0$. If there exists $N>2\alpha $ such that $\left\vert \lambda
_{N}\right\vert >0$, then
\begin{eqnarray}
E[G^{-2\alpha }] &\leq &E\left[ \left( \sum_{i=1}^{N}\lambda
_{i}^{2}X_{i}^{2}\right) ^{-\alpha }\right] =\frac{1}{\Gamma (\alpha )}E%
\left[ \int_{0}^{\infty }e^{-y\sum_{i=1}^{N}\lambda
_{i}^{2}X_{i}^{2}}y^{\alpha -1}dy\right]  \notag \\
&=&\frac{1}{\Gamma (\alpha )}\int_{0}^{\infty }y^{\alpha
-1}\prod_{i=1}^{N}(1+2\lambda _{i}^{2}y)^{-\frac{1}{2}}dy.
\label{negMmt-ineq}
\end{eqnarray}%
Since $\lambda _{i}^{2}$ is non increasing in $i$ and $N>2\alpha $, using
the change of variables $1+2\lambda _{N}^{2}y=z$ we have
\begin{eqnarray*}
&&\int_{0}^{\infty }y^{\alpha -1}\prod_{i=1}^{N}(1+2\lambda _{i}^{2}y)^{-%
\frac{1}{2}}dy\leq \int_{0}^{\infty }y^{\alpha -1}(1+2\lambda _{N}^{2}y)^{-%
\frac{N}{2}}dy \\
&=&\left( 2\lambda _{N}^{2}\right) ^{-\alpha }\int_{1}^{\infty }\left(
z-1\right) ^{\alpha -1}z^{-\frac{N}{2}}dz=\left( 2\lambda _{N}^{2}\right)
^{-\alpha }\int_{1}^{\infty }\left( \frac{z-1}{z}\right) ^{\alpha
-1}z^{\alpha -1-\frac{N}{2}}dz \\
&= & \left( 2\lambda _{N}^{2}\right) ^{-\alpha } \int_0^1 (1-x)^{\alpha-1}
x^{\frac N2-\alpha-1} dx = \left( 2\lambda _{N}^{2}\right) ^{-\alpha } \frac{
\Gamma(\alpha) \Gamma(\frac N2-\alpha)}{ \Gamma(N/2)},
\end{eqnarray*}%
which implies (\ref{neg-mmt}).

On the other hand, if $\left\vert \lambda _{i}\right\vert =0$ for all $%
i>2\alpha $, let $N\leq 2\alpha $ be the largest nonnegative integer such
that $\left\vert \lambda _{N}\right\vert >0$. Then, the inequality in (\ref%
{negMmt-ineq}) becomes an equality. Using again that $\left\{ \lambda
_{i}^{2}\right\} _{i\geq 1}$ is a decreasing sequence we have
\begin{equation*}
\int_{0}^{\infty }y^{\alpha -1}\prod_{i=1}^{N}(1+2\lambda _{i}^{2}y)^{-\frac{%
1}{2}}dy\geq (1+2\lambda _{1}^{2})^{-\frac{N}{2}}(\int_{0}^{1}y^{\alpha
-1}dy+\int_{1}^{\infty }y^{\alpha -1-\frac{N}{2}}dy)=\infty ,
\end{equation*}%
and we conclude that $E[G^{-2\alpha}]=\infty$. This completes the proof.
\end{proof}

The following theorem describes the distance between the densities of $%
F=I_{2}(f)$ and $N(0,E[F^{2}])$.

\begin{theorem}
Let $F=I_{2}(f)$ with $f\in \mathfrak{H}^{\odot 2}$ given in $(\ref%
{f-expansion})$. Assume that there exists $N> 6m + 6\left( \lfloor \frac{m}{2%
}\rfloor \vee1\right)$, for some integer $m\ge 0$, such that $\lambda
_{N}\neq 0$. Then $F$ admits an $m$ times continuously differentiable
density $f_F$. Furthermore, if $\phi (x)$ denotes the density of $%
N(0,E[F^{2}] )$, then for $k=0,1,\dots ,m$,
\begin{equation*}
\sup_{x\in \mathbb{R}}\left\vert f_{F}^{(k)}(x)-\phi ^{(k)}(x)\right\vert
\leq C \left( \sum_{i=1}^{\infty }\lambda _{i}^{4}\right) ^{\frac{1}{2}}\leq
C \left( E[F^{4}]-3\left( E[F^{2}]\right) ^{2}\right) ^{\frac{1}{2}},
\end{equation*}%
where the constant $C $ depends on $N$ and $\lambda _{N} $.
\end{theorem}

\begin{proof}
Taking into account of (\ref{DF-expsn2}), we have
\begin{equation}
\mathrm{Var}\left( \Vert DF\Vert _{\mathfrak{H}}^{2}\right) =E\left\vert
4\sum_{i=1}^{\infty }\lambda _{i}^{2}\left( I_{1}(e_{i})^{2}-1\right)
\right\vert ^{2}=32\sum_{i=1}^{\infty }\lambda _{i}^{4}.  \label{qSigma-w}
\end{equation}%
From (\ref{DF-expsn2}) and Lemma \ref{neg-moment} it follows that
\begin{equation}
E[\left\Vert DF\right\Vert _{\mathfrak{H}}^{-\beta }]\leq C_{\beta
/2}N^{-\beta /2}|\lambda _{N}|^{-\beta },  \label{DF-negmmt}
\end{equation}%
for all $\beta <N$. Then, the theorem follows from Theorem \ref{qDeriRate},
taking into account (\ref{qSigma-w}).
\end{proof}

Now we are ready to prove convergence of densities of random variables in
the second Wiener chaos. Consider a sequence $F_{n}=I_{2}(f_{n})$ with $%
f_{n}\in \mathfrak{H}^{\odot 2}$, which can be written as
\begin{equation}
f_{n}=\sum_{i=1}^{\infty }\lambda _{n,i}e_{n,i}\otimes e_{n,i},  \label{fT}
\end{equation}%
where $\left\{ \lambda _{n,i},i\geq 1\right\} $ verifies $\left\vert \lambda
_{n,i}\right\vert \geq \left\vert \lambda _{n,i+1}\right\vert $ for all $%
i\geq 1$ and $\left\{ e_{n,i},i\geq 1\right\} $ are the corresponding
eigenvectors.

\begin{theorem}
\label{I2Th}Let $F_{n}=I_{2}(f_{n})$ with $f_{n}\in $ $\mathfrak{H}^{\odot
2} $ given by $(\ref{fT})$. Assume that $\left\{ \lambda _{n,i}\right\}
_{n,i\in \mathbb{N}}$ \ satisfies

\begin{itemize}
\item[(i)] $\sigma ^{2}:=2\lim_{n\rightarrow \infty }\sum_{i=1}^{\infty }
\lambda _{n,i} ^2>0$;

\item[(ii)] $\lim_{n\rightarrow \infty }\sum_{i=1}^{\infty } \lambda _{n,i}
^4=0$;

\item[(iii)] $\inf_n \left(\sup_{i>6m + 6\left( \lfloor \frac{m}{2}\rfloor
\vee1\right)} |\lambda_{n,i}| \sqrt{i} \right)>0$ for some integer $m\ge 0$.
\end{itemize}

Then, each $F_{n}$ admits a density function $f_{F_{n}}\in C^{m }\left(
\mathbb{R}\right) $. Furthermore, for $k=0,1,\dots ,m $ and if $\phi $
denotes the density of the law $N(0,\sigma^2)$, the derivatives of $%
f_{F_{n}}^{(k)}$ converge uniformly to the derivatives of $\phi$ with a
rate given by
\begin{equation*}
\sup_{x\in \mathbb{R}}\left\vert f_{F_{n}}^{(k)}(x)-\phi
^{(k)}(x)\right\vert \leq C \left[ \left (\sum_{i=1}^{\infty } \lambda
_{n,i} ^{4} \right)^{\frac{1}{2}}+\left\vert 2\sum_{i=1}^{\infty } \lambda
_{n,i} ^{2}-\sigma ^{2}\right\vert ^{\frac 12} \right],
\end{equation*}%
where $C $ is a constant depending only on $m $ and the infimum appearing in
condition (iii).
\end{theorem}

\begin{proof}[\textbf{Proof of Theorem }$\protect\ref{I2Th}$]
Note that $E[(I_{1}(e_{n,i})^{2}-1)(I_{1}(e_{n,j})^{2}-1)]=2\delta _{ij}$.
Thus,
\begin{equation*}
\sum_{i=1}^{\infty }\lambda _{n,i}^{2}=\left\Vert f_{n}\right\Vert _{%
\mathfrak{H}^{\odot 2}}^{2}=\frac{1}{2}E[F_{n}^{2}].
\end{equation*}
Then, the result follows from (\ref{qSigma-w}), (\ref{DF-negmmt}) and
Corollary \ref{qDeriConv}.
\end{proof}

\medskip Condition (iii) in Theorem $\ref{I2Th}$ means that there exist a
positive constant $\delta>0$ such that for each $n$ we can find an index $%
i(n)>6m + 6\left( \lfloor \frac{m}{2}\rfloor \vee1\right)$ with $%
|\lambda_{n,i(n)}| \sqrt{i(n)} \ge \delta$.

\begin{remark} \label{I2example}
It is interesting to compare Theorem $\ref{I2Th}$ with the case when%
\begin{equation*}
\displaystyle{\lambda }_{n,i}=%
\begin{cases}
\frac{1}{\sqrt{n}} & \quad \hbox{if}\ 1\leq i\leq n; \\
0 & \quad \hbox{if}\ i\geq n+1,%
\end{cases}%
\end{equation*}
which corresponds to classical case of sum of independent and identically
distributed random variables. In this case all the conditions of Theorem $\ref
{I2Th}$ are satisfied with $\sigma ^{2}=2$.
Moreover we have $\sum_{i=1}^{\infty }\lambda _{n,i}^{4}=\frac{1}{n}$ and $%
\sum_{i=1}^{\infty }\lambda _{n,i}^{2}=1$. Then we obtain a Berry--Essen type
bound for the derivatives of the density. Namely, we have $\sup_{x\in
\mathbb{R}}\left\vert f_{F_{n}}^{(k)}(x)-\phi ^{(k)}(x)\right\vert \leq
\frac{C}{\sqrt{n}}$ for sufficiently large $n$, which provides the right rate of
convergence.
\end{remark}

\subsection{Parameter estimation in Ornstein-Uhlenbeck processes}

Consider the following Ornstein-Uhlenbeck process
\begin{equation*}
X_{t}=-\theta \int_{0}^{t}X_{s}ds+\gamma B_{t},
\end{equation*}%
where $\theta >0$ is an unknown parameter, $\gamma >0$ is known and $%
B=\{B_{t},0\leq t<\infty \}$ is a standard Brownian motion. Assume that the
process $X=\{X_{t},0\leq t\leq T\}$ can be observed continuously in the time
interval $[0,T]$. Then the least squares estimator (or the maximum
likelihood estimator) of $\theta $ is given by $\displaystyle\widehat{\theta
}_{T}=\frac{\int_{0}^{T}X_{t}dX_{t}}{\int_{0}^{T}X_{t}^{2}dt}$. It is known
(see for example, \cite{LipS01}, \cite{Kut04}) that, as $T$ tends to
infinity, $\widehat{\theta }_{T}$ converges to $\theta $ almost surely and
\begin{equation}
\sqrt{T}(\widehat{\theta }_{T}-\theta )=-\frac{TF_{T}}{%
\int_{0}^{T}X_{t}^{2}dt}\overset{\mathcal{L}}{\rightarrow }N(0,2\theta ),
\label{FT_conv}
\end{equation}%
where
\begin{equation}  \label{FT0}
F_{T}=I_{2}(f_{T})=\int_{0}^{T}\int_{0}^{T}f_{T}(t,s)dB_{t}dB_{s},
\end{equation}%
with
\begin{equation}  \label{FT}
f_{T}(t,s)=\frac{\gamma ^{2}}{2\sqrt{T}}e^{-\theta \left\vert t-s\right\vert
}.
\end{equation}
Recently, Hu and Nualart \cite{HN10SPL} extended this result to the case
where $B$ is a fractional Brownian motion with Hurst parameter $H\in \lbrack
\frac{1}{2},\frac{3}{4})$, which includes the standard Brownian motion case.
Since $\frac{1}{T}\int_{0}^{T}X_{t}^{2}dt\rightarrow \frac{1}{2}\gamma
^{2}\theta ^{-1}$ almost surely as $T$ tends to infinity, the main effort in
proving (\ref{FT_conv}) is to show the convergence in law of $F_{T}$ to the
normal law $N(0,\frac{\gamma ^{4}}{2\theta })$. We shall prove that the
density of $F_{T}$ converges as $T$ tends to infinity to the density of the
normal distribution $N(0,\frac{\gamma ^{4}}{2\theta })$.

\begin{theorem}
\label{OUthm}Let $F_{T}$ be given by $\left( \ref{FT}\right) $ and let $\phi
$ be the density of the law $N(0,\sigma ^{2})$, where $\sigma ^{2}=\frac{%
\gamma ^{4}}{2\theta }$. Then for each $T>0$, $F_{T}$ has a smooth
probability density $f_{F_{T}}$ and for any $k\geq 0$,
\begin{equation*}
\sup_{x\in \mathbb{R}}\left\vert f_{F_{T}}^{(k)}(x)-\phi
^{(k)}(x)\right\vert \leq CT^{-\frac{1}{2}},
\end{equation*}%
where the constant $C$ depends on $k$, $\gamma $ and $\theta $.
\end{theorem}

Before proving the theorem, let us first analyze the asymptotic behavior of
the eigenvalues of $f_{T}$. The Hilbert space corresponding to Brownian
motion $B$ is $\mathfrak{H}=L^{2}([0,T])$. Let $Q_{T}:L^{2}([0,T])%
\rightarrow L^{2}([0,T])$ be the Hilbert-Schmidt operator associated to $%
f_{T}$, that is,
\begin{equation}
\left( Q_{T}\varphi \right) (t)=\int_{0}^{T}f_{T}(t,s)\varphi (s)ds
\label{QT}
\end{equation}%
for $\varphi \in L^{2}[0,T]$. The operator $Q_{T}$ has eigenvalues $\lambda
_{T,1}>\lambda _{T,2}>\cdots \geq 0$ and $\sum_{i=1}^{\infty }\lambda
_{T,i}^{2}<\infty $. The following lemma provides upper and lower bounds for
these eigenvalues.

\begin{lemma}
\label{LQ} Fix $T>0$. Let $f_{T}$ be given by $(\ref{FT})$ and $Q_{T}$ be
given by $(\ref{QT})$. The eigenvalues $\lambda _{T,i}$ of $Q_{T}$ (except
maybe one) satisfy the following estimates
\begin{equation}
\frac{\gamma ^{2}\theta }{\sqrt{T}\left( \theta ^{2}+\left( \frac{i\pi +%
\frac{\pi }{2}}{T}\right) ^{2}\right) }<\lambda _{T,i}<\frac{\gamma
^{2}\theta }{\sqrt{T}\left( \theta ^{2}+\left( \frac{i\pi -\frac{\pi }{2}}{T}%
\right) ^{2}\right) }.  \label{Lamda-Tk}
\end{equation}
\end{lemma}

\begin{proof}
Consider the eigenvalue problem $Q_{T}\varphi =\lambda \varphi $, that is,%
\begin{equation}
\int_{0}^{T}f_{T}(t,s)\varphi (s)ds=\frac{\gamma ^{2}}{2\sqrt{T}}\left(
\int_{0}^{t}e^{-\theta (t-s)}\varphi (s)ds+\int_{t}^{T}e^{-\theta
(s-t)}\varphi (s)ds\right) =\lambda \varphi (t).  \label{eigen1}
\end{equation}%
Then, $\phi \,$\ is differentiable and%
\begin{equation}
\frac{\gamma ^{2}\theta }{2\sqrt{T}}\left( -\int_{0}^{t}e^{-\theta
(t-s)}\varphi (s)ds+\int_{t}^{T}e^{-\theta (s-t)}\varphi (s)ds\right)
=\lambda \varphi ^{\prime }(t).  \label{eigen2}
\end{equation}%
Differentiating again we have%
\begin{equation*}
\frac{\gamma ^{2}\theta }{2\sqrt{T}}\left( -2\varphi (t)+\theta
\int_{0}^{t}e^{-\theta (t-s)}\varphi (s)ds+\theta \int_{t}^{T}e^{-\theta
(s-t)}\varphi (s)ds\right) =\lambda \varphi ^{\prime \prime }(t).
\end{equation*}%
Comparing this expression with (\ref{eigen1}), we obtain
\begin{equation}
(\theta ^{2}-\frac{\gamma ^{2}\theta }{\sqrt{T}\lambda })\varphi (t)=\varphi
^{\prime \prime }(t).  \label{eigenEqu}
\end{equation}%
Also, from (\ref{eigen1}) and (\ref{eigen2}) it follows that
\begin{equation}
\varphi (0)=\theta \varphi ^{\prime }(0),~\varphi (T)=-\theta \varphi
^{\prime }(T).  \label{NeumannC}
\end{equation}%
Equations (\ref{eigenEqu}) and (\ref{NeumannC}) form a Sturm-Liouville
system. Its general solution is of the form
\begin{equation*}
\varphi (t)=C_{1}\sin \mu t+C_{2}\cos \mu t,
\end{equation*}%
where $C_{1}$ and $C_{2}$ are constants, and $\mu >0$ is an eigenvalue of
the Sturm-Liouville system. By eliminating the constants $C_{1}$ and $C_{2}$
from (\ref{eigenEqu}) and (\ref{NeumannC}) we obtain
\begin{equation}
-\mu ^{2}=\theta ^{2}-\frac{\gamma ^{2}\theta }{\sqrt{T}\lambda }.
\label{muLamda}
\end{equation}%
Then, the desired estimates on the eigenvalues of $Q_{T}\varphi =\lambda
\varphi $will follow form estimates on $\mu$. Note that the Neumann
condition (\ref{NeumannC}) yields
\begin{equation*}
(\mu ^{2}\theta ^{2}-1)\sin \mu T=2\mu \theta \cos \mu T.
\end{equation*}%
If we write $x=\mu \theta >0$ (since $\mu ,\theta >0$), the above equation
becomes%
\begin{equation*}
(x^{2}-1)\sin \frac{x}{\theta }T=2x\cos \frac{x}{\theta }T.
\end{equation*}
The solution $x=1$ corresponds to the eigenvalue $\mu =\frac 1 \theta$. If $%
x\not =1$, then $\cos \frac{x}{\theta }T\neq 0$ and
\begin{equation}
\tan \frac{x}{\theta }T=\frac{2x}{x^{2}-1}.  \label{SLequ}
\end{equation}%
For any $i\in \mathbb{Z}_{+}$, there is exactly one solution $x_{i}$ to (\ref%
{SLequ}) such that $\frac{x_{i}}{\theta }T\in (i\pi -\frac{\pi }{2},i\pi +%
\frac{\pi }{2})$. Corresponding to each $x_{i}$ is an eigenvalue $\mu _{i}=%
\frac{x_{i}}{\theta }$ of the Sturm-Liouville system, satisfying $%
\displaystyle\frac{i\pi -\frac{\pi }{2}}{T}<\mu _{i}<\frac{i\pi +\frac{\pi }{%
2}}{T}$. The corresponding eigenvalue $\lambda _{i}$ of $Q_{T}$ obtained
from Equation (\ref{muLamda}) satisfies the estimate (\ref{Lamda-Tk}).
\end{proof}

\begin{proof}[Proof of Theorem $\protect\ref{OUthm}$]
For each $T$, let us compute the second moment of $F_{T}$.
\begin{eqnarray*}
E\left[ F_{T}^{2}\right] &=&\left\Vert f_{T}\right\Vert _{\mathfrak{H}%
^{\otimes 2}}^{2}=\int_{0}^{T}\int_{0}^{T}f_{T}(t,s)^{2}dsdt \\
&=&\frac{\gamma ^{4}}{4T}\int_{0}^{T}\int_{0}^{t}e^{-2\theta \left(
t-s\right) }dsdt=\frac{\gamma ^{4}}{2\theta }-\frac{\gamma ^{4}}{8\theta T}%
(1-e^{-2\theta T}).
\end{eqnarray*}%
Also, noticing that $F_{T}=I_{2}(f_{T})=\delta ^{2}(f_{T})$ and
\begin{equation*}
D_{s}D_{t}F_{T}^{3}=3F_{T}^{2}f_{T}(t,s)+6F_{T}I_{1}(f(\cdot ,t))\otimes
I_{1}(f(\cdot ,s)),
\end{equation*}%
and using the duality between $\delta $ and $D$, we can write
\begin{eqnarray*}
E\left[ F_{T}^{4}\right] &=&E\left[ \left\langle
f_{T},D^{2}F_{T}^{3}\right\rangle _{\mathfrak{H}^{\otimes 2}}\right] =3E%
\left[ F_{T}^{2}\left\langle f_{T},f_{T}\right\rangle _{\mathfrak{H}%
^{\otimes 2}}\right] \\
&&+6E\left[ F_{T}\left\langle f_{T}(t,s),I_{1}(f_{T}(\cdot ,t))\otimes
I_{1}(f_{T}(\cdot ,s))\right\rangle _{\mathfrak{H}^{\otimes 2}}\right] \\
&=&3\left( E\left[ F_{T}^{2}\right] \right) ^{2}+6A,
\end{eqnarray*}%
where
\begin{eqnarray*}
A &=&E\left[ F_{T}\left\langle f_{T}(t,s),I_{1}(f_{T}(\cdot ,t))\otimes
I_{1}(f_{T}(\cdot ,s))\right\rangle _{\mathfrak{H}^{\otimes 2}}\right] \\
&=&\left\langle f_{T}(u,v),\left\langle f_{T}(t,s),f_{T}(u,t)\otimes
f_{T}(v,s)\right\rangle _{\mathfrak{H}^{\otimes 2}}\right\rangle _{\mathfrak{%
H}^{\otimes 2}} \\
&=&\frac{\gamma ^{8}}{16T^{2}}\int_{0}^{T}\int_{0}^{T}\int_{0}^{T}%
\int_{0}^{T}e^{-\theta \left( \left\vert u-v\right\vert +\left\vert
t-s\right\vert +\left\vert u-t\right\vert +\left\vert v-s\right\vert \right)
}dudvdtds.
\end{eqnarray*}%
Because the integrand is symmetric, we have
\begin{equation*}
A=\frac{\gamma ^{8}}{16T^{2}}4!\int_{0}^{T}du\int_{0}^{u}dv\int_{0}^{v}ds%
\int_{0}^{s}dt~e^{-2\theta \left( u-t\right) }\leq CT^{-1}.
\end{equation*}%
Then, in order to complete the proof by applying Corollary \ref{qDeriConv},
we only need to verify that condition $\mathrm{(iii)}$ of Theorem \ref{I2Th}
holds for any integer $m\geq 1$, which implies the uniform boundedness of
the negative moments
\begin{equation*}
\sup_{T>0}E\left[ \left\Vert DF_{T}\right\Vert _{\mathcal{\mathfrak{H}}%
}^{-\beta }\right] <\infty
\end{equation*}%
for any $\beta >0$. Fix $\beta >0$, and for each $T$, let $i(T)=\lfloor
\beta +1\rfloor +\lfloor T\rfloor $. Then, the lower bound in (\ref{Lamda-Tk}%
) yields
\begin{eqnarray*}
\sqrt{i(T)}\lambda _{T,i(T)} &\geq &\frac{\sqrt{i(T)}\gamma ^{2}/\theta }{%
\sqrt{T}\left( 1+\left( \frac{\left( i+1/2\right) \pi }{T\theta }\right)
^{2}\right) } \\
&\geq &\frac{\sqrt{i(T)}\gamma ^{2}/\theta }{\sqrt{T}\left( 1+\left( \frac{%
i(T)}{T}\right) ^{2}4\frac{\pi ^{2}}{\theta ^{2}}\right) }\geq \frac{\gamma
^{2}/\theta }{\max_{\left( \beta +2\right) ^{-1}\leq r\leq 1}g(r)}>0,
\end{eqnarray*}%
where in the last inequality we made the substitution $r^{-1}=\frac{i(T)}{T}$
and set $\,$%
\begin{equation*}
g(r):=\sqrt{r}(1+r^{-2}4\frac{\pi ^{2}}{\theta ^{2}}).
\end{equation*}%
This implies condition (iii) and the proof of the theorem is complete.
\end{proof}

\subsection{Multidimensional Case}

Now we give an example for random vectors. Let $X=\{X(h),h\in \mathfrak{H}\}
$ be an isonormal Gaussian process associated with the Hilbert space $%
\mathfrak{H}$. Suppose that $\{ e_{ij}, 1\le i\le d, j\ge 1\}$ is a sequence
of orthonormal elements in $\mathfrak{H}$. Set $e_k=(e_{1k} ,\dots, e_{dk})$
for any $k\ge 1$.  Let $A_n$ be a sequence of $d\times d$ invertible
matrices such that $A_n\rightarrow I$ as $n\rightarrow \infty$. For any $%
k\ge 1$ define Define
\begin{equation*}
\xi_{nk} =\left(%
\begin{matrix}
\xi_{1nk} \\
\vdots \\
\xi_{dnk}%
\end{matrix}%
\right) =A_n \left(%
\begin{matrix}
e_{1k} \\
\vdots \\
e_{dk}%
\end{matrix}%
\right)
\end{equation*}
and, for any $j=1,\dots, d$ set
\begin{equation*}
F_{jn}=\sum_{k=1}^\infty {\lambda}_{jnk} I_2(\xi_{jnk} ^{\otimes 2})=
\sum_{k=1}^\infty {\lambda}_{jnk} \left[ \tilde \xi_{jnk}^2-\|\xi_{jnk}\| ^2%
\right]\,,
\end{equation*}
where $\tilde \xi_{jnk}=I_1(\xi_{jnk})=X(\xi_{jnk})$ and ${\lambda}_{jnk}$
are real numbers which will be specified later. We plan to use Theorem \ref%
{MultiThm0} to study the convergence of the random vectors $F_n=(F_{1n}
\dots F_{dn})$. For this we can follow the approach of Section \ref{s.7.1},
the main extra work being to prove the existence of a uniform bound for the
negative moments of the Malliavin covariance matrices. We have
\begin{equation*}
DF_{jn}=2\sum_{k=1}^\infty {\lambda}_{jnk} \tilde \xi_{jnk} \xi_{jnk}\,.
\end{equation*}
Thus
\begin{equation}
\langle DF_{in}, DF_{jn}\rangle_ {\mathfrak{H}} =4\sum_{k=1}^\infty {\lambda}%
_{ink}{\lambda}_{jnk} \tilde \xi_{ink}\tilde \xi_{jnk} {\alpha}_{ijn}\,,
\label{DF}
\end{equation}
where ${\alpha}_{ijn}$ is the $(i,j)$th entry of the matrix ${\alpha}%
_n=A_nA_n^T$. Consider the matrix $\beta_n:=\left(\beta_{ijn}\right)_{1\le
i,j\le d}$ given by
\begin{equation*}
\beta_{ijn}:=4\sum_{k=1}^\infty {\lambda}_{ink}{\lambda}_{jnk} \tilde
\xi_{ink}\tilde \xi_{jnk}\,.
\end{equation*}
Then from \hbox{(\ref{DF})}, we see that $\langle DF_n, DF_n\rangle
=\left(\langle DF_{in} \,, DF_{jn}\rangle_{\mathfrak{H}} \right)_{1\le
i,j\le d}$ is the Hadamard product of the nonnegative definite matrices ${%
\alpha}_n$ and $\beta_n$. By the Oppenheim's inequality for Hadamard
product, and taking into account that $\det(\alpha_n)$ converges to one,
there exists a constant $c>0$ such that
\begin{equation*}
\det(\langle DF_n, DF_n\rangle)\ge \det({\alpha}_n) \prod_{j=1}^d
\beta_{jjn} \ge c \prod_{j=1}^d \beta_{jjn},
\end{equation*}
for all $n$. Note $\beta_{jjn}= 4\sum_{k=1}^\infty {\lambda}_{jnk}^2\left(
\tilde \xi_{jnk} \right)^2$ and $\xi_{jnk}\rightarrow e_{jk}$. Thus we can
follow Section \ref{s.7.1} to verify the conditions that allow us to apply
Theorem \ref{MultiThm0}. We will write down the theorem and omit the
details. In the following we denote by $\phi_\sigma $ the density of the law
$N(0,\hbox{diag}(\sigma^2_1\,, \dots\,, \sigma_d^2))$ and $\partial _{{\alpha%
}_1} \cdots \partial _{{\alpha}_d} f(x)= \frac{\partial^{|{\alpha}|}} {%
\partial x_1^{{\alpha}_1}\cdots\partial x_d^{{\alpha}_d} } f(x)$ with $|{%
\alpha}|={\alpha}_1+\cdots+{\alpha}_d$.

\begin{theorem}
Let $A_n$ be a sequence of $d\times d$ invertible matrices such that $%
A_n\rightarrow I$ and let $F_n=(F_{1n}, \dots, F_{dn})$ be defined as above.
We assume the ${\lambda}_{jnk}$ satisfy the following conditions for any $%
1\le j\le d$.

\begin{itemize}
\item[(i)] $\sigma_j ^{2}:=\lim_{n\rightarrow \infty }\sum_{k=1}^{\infty }
\lambda _{jnk} ^2>0$;

\item[(ii)] $\lim_{n\rightarrow \infty }\sum_{k=1}^{\infty } \lambda _{jnk}
^4=0$;

\item[(iii)] $\inf_{j,n} \left(\sup_{k>6m + 6\left( \lfloor \frac{m}{2}%
\rfloor \vee1\right)} |\lambda_{jnk}| \sqrt{k} \right)>0$ for some integer $%
m\ge 0$.
\end{itemize}

Then, each $F_{n}$ admits a density function $f_{F_{n}}\in C^{m }\left(
\mathbb{R}^d \right) $. Furthermore, for any ${\alpha}=({\alpha}_1, \dots,{%
\alpha}_d) $, with $|{\alpha}|\le m$, the derivatives of $\partial _{{\alpha}%
_1} \cdots \partial _{{\alpha}_d} f_{F_{n}} $ converge uniformly to the
derivatives of $\partial _{{\alpha}_1} \cdots \partial _{{\alpha}_d}
\phi_\sigma $ with a rate given by
\begin{equation*}
\sup_{x\in \mathbb{R}}\left\vert \partial _{{\alpha}_1} \cdots \partial _{{%
\alpha}_d} f_{F_{n}} (x)-\partial _{{\alpha}_1} \cdots \partial _{{\alpha}%
_d} \phi_\sigma (x)\right\vert \leq C \sum_{j=1}^d \left[ \left
(\sum_{k=1}^{\infty } \lambda _{jnk} ^{4} \right)^{\frac{1}{2}}+\left\vert
\sum_{i=1}^{\infty } \lambda _{jnk} ^{2}-\sigma ^{2}\right\vert ^{\frac 12} %
\right],
\end{equation*}%
where $C $ is a constant depending only on $m $ and the infimum appearing in
condition (iii).
\end{theorem}

\section{Appendix\label{Appdx}}

In this section, we present the omitted proofs and some technical results.

\begin{proof}[Proof of Lemma \protect\ref{solu-ctrl}]
Since $\int_{-\infty }^{\infty }\{h(y)-E[h(N)]\}e^{-y^{2}/(2\sigma
^{2})}dy=0 $, we have
\begin{equation*}
\int_{-\infty }^{x}\{h(y)-E[h(N)]\}e^{-y^{2}/(2\sigma
^{2})}dy=-\int_{x}^{\infty }\{h(y)-E[h(N)]\}e^{-y^{2}/(2\sigma ^{2})}dy.
\end{equation*}%
Hence
\begin{equation*}
\left\vert \int_{-\infty }^{x}\{h(y)-E[h(N)]\}e^{-y^{2}/(2\sigma
^{2})}dy\right\vert \leq \int_{\left\vert x\right\vert }^{\infty
}[ay^{k}+b+E\left\vert h(N)\right\vert ]e^{-y^{2}/(2\sigma ^{2})}dy.
\end{equation*}%
By using the representation (\ref{solu}) of $f_{h}$ and Stein's equation (%
\ref{Stein's equ}) we have
\begin{eqnarray}
\left\vert f_{h}^{\prime }(x)\right\vert &\leq &\left\vert
h(x)-E[h(N)]\right\vert +\frac{\left\vert x\right\vert }{\sigma ^{2}}%
e^{x^{2}/(2\sigma ^{2})}\left\vert \int_{-\infty
}^{x}\{h(y)-E[h(N)]\}e^{-y^{2}/(2\sigma ^{2})}dy\right\vert  \notag \\
&\leq &a\left\vert x\right\vert ^{k}+b+E\left\vert h(N)\right\vert +\frac{1}{%
\sigma ^{2}}e^{x^{2}/(2\sigma ^{2})}\int_{\left\vert x\right\vert }^{\infty
}y[ay^{k}+b+E\left\vert h(N)\right\vert ]e^{-y^{2}/(2\sigma ^{2})}dy  \notag
\\
&=&a\left\vert x\right\vert ^{k}+\left( b+E\left\vert h(N)\right\vert
\right) \left( 1+\frac{1}{\sigma ^{2}}s_{1}(x)\right) +\frac{a}{\sigma ^{2}}%
s_{k+1}(x),  \label{f'-bd}
\end{eqnarray}%
where we let $s_{k}(x)=e^{x^{2}/(2\sigma ^{2})}\int_{\left\vert x\right\vert
}^{\infty }y^{k}e^{-y^{2}/(2\sigma ^{2})}dy$ for any integer $k\geq 0$.

Note that $E\left\vert h(N)\right\vert \leq aE\left\vert N\right\vert
^{k}+b\leq C_{k}a\sigma ^{k}+b$ and
\begin{equation*}
s_{1}(x)=e^{x^{2}/(2\sigma ^{2})}\int_{x}^{\infty }ye^{-\frac{y^{2}}{2\sigma
^{2}}}dy=\sigma ^{2}
\end{equation*}
for all $x\in \mathbb{R}$. Using integration by parts, we see by induction
that for any integer $k\geq 1 $,
\begin{eqnarray*}
&&s_{k+1}(x)=e^{x^{2}/(2\sigma ^{2})}\int_{\left\vert x\right\vert }^{\infty
}y^{k+1}e^{-y^{2}/(2\sigma ^{2})}dy \\
&=&\sigma ^{2}e^{x^{2}/(2\sigma ^{2})}\int_{\left\vert x\right\vert
}^{\infty }y^{k}d(-e^{-y^{2}/(2\sigma ^{2})})=\sigma ^{2}[\left\vert
x\right\vert ^{k}+k~s_{k-1}(x)].
\end{eqnarray*}%
Then if $k\geq 1$ is even, we have
\begin{equation*}
s_{k+1}(x)\leq C_{k}\sigma ^{2}[\left\vert x\right\vert ^{k}+\sigma
^{2}\left\vert x\right\vert ^{k-2}+\cdots +\sigma ^{k-2}s_{1}(x)]\leq
C_{k}\sigma ^{2}\sum_{i=0}^{k}\sigma ^{k-i}\left\vert x\right\vert ^{i}.
\end{equation*}%
If $k\geq 1$ is odd, we have%
\begin{equation*}
s_{k+1}(x)\leq C_{k}\sigma ^{2}[\left\vert x\right\vert ^{k}+\sigma
^{2}\left\vert x\right\vert ^{k-2}+\cdots +\sigma ^{k-1}(\left\vert
x\right\vert +s_{0}(x))]\leq C_{k}\sigma ^{2}\sum_{i=0}^{k}\sigma
^{k-i}\left\vert x\right\vert ^{i},
\end{equation*}%
where we used the fact that $s_{0}(x)\leq s_{0}(0)=\sqrt{\frac{\pi }{2}}%
\sigma $ for all $x\in \mathbb{R}$ (indeed, when $x\geq 0$ we have $%
s_{0}^{\prime }(x)=\frac{x}{\sigma ^{2}}e^{x^{2}/(2\sigma
^{2})}\int_{x}^{\infty }e^{-y^{2}/(2\sigma ^{2})}dy-1\leq e^{x^{2}/(2\sigma
^{2})}\int_{x}^{\infty }\frac{y}{\sigma ^{2}}e^{-\frac{y^{2}}{2\sigma ^{2}}%
}dy-1=0$; similarly when $x<0$, $s_{0}^{\prime }(x)\geq 0$). Putting the
above estimates into (\ref{f'-bd}) we complete the proof.
\end{proof}

\begin{proof}[Proof of Lemma $\protect\ref{rem1}$]
We shall prove these properties by induction. From $T_{1}=T_{2}=0$, (\ref%
{gHermite}) and (\ref{tk}) we know that $T_{3}=D_{u}^{2}\delta _{u}$, with $%
J_{3}=\left\{ (0,0,1)\right\} $; and $T_{4}=\delta _{u}D_{u}^{2}\delta _{u}+$
$D_{u}^{3}\delta _{u}$, with $J_{4}=$ $\left\{ (1,0,1,0),(0,0,0,1)\right\} $%
. Now suppose the statement is true for all $T_{l}$ with $l\leq k-1$ for $%
k\geq 5$. We want to prove the multi-indices of $T_{k}$ satisfy $(a)$--$(c)$%
. This will be done by studying the three operations, $\delta _{u}T_{k-1}$, $%
D_{u}T_{k-1}$ and $\partial _{\lambda }H_{k-1}(D_{u}\delta _{u},\delta
_{u})D_{u}^{2}\delta _{u}$, in expression $(\ref{tk})$.

For the term $\partial _{\lambda }H_{k-1}(D_{u}\delta _{u},\delta
_{u})D_{u}^{2}\delta _{u}$, we observe from (\ref{gHermite}) that
\begin{equation*}
\partial _{\lambda }H_{k-1}(D_{u}\delta _{u},\delta _{u})D_{u}^{2}\delta
_{u}=D_{u}^{2}\delta _{u}\sum_{1\leq i\leq \lfloor (k-1)/2\rfloor
}ic_{k-1,i}\delta _{u}^{k-1-2i}(D_{u}\delta _{u})^{i-1},
\end{equation*}%
whose terms have multi-indices $(k-1-2i,i-1,1,0,\dots ,0)\in \mathbb{N}^{k}$
for $1\leq i\leq \lfloor \frac{k-1}{2}\rfloor $. Then, it is straightforward
to check that these multi-indices satisfy $(a)$, $(b)$ and $(c)$.

The term $\delta _{u}T_{k-1}$ shifts the multi-index $(i_{0},i_{1},\dots
,i_{k-2})\in J_{k-1}$ to $(i_{0}+1,i_{1},\dots ,i_{k-2},0)\in \mathbb{N}^{k}$%
, which obviously satisfies $(a)$, $(b)$ and $(c)$, due to the induction
hypothesis.

The third term $D_{u}T_{k-1}$ shifts the multi-index $(i_{0},i_{1},\dots
,i_{k-2})\in J_{k-1}$ to either $\alpha =(i_{0}-1,i_{1}+1,\dots
,i_{k-2},0)\in \mathbb{N}^{k}$ if $i_{0}\geq 1$, or to
\begin{equation*}
\beta =%
\begin{cases}
(i_{0},i_{1},\dots ,i_{j_{0}}-1,i_{j_{0}+1}+1,\dots ,i_{k-2},0), & \text{for
}1\leq j_{0}\leq k-3; \\
(i_{0},i_{1},\dots ,i_{j_{0}}-1,1), & \text{for }j_{0}=k-2,%
\end{cases}%
\end{equation*}%
if $i_{j_{0}}\geq 1$. It is easy to check that $\beta $ satisfies properties
$(a)$, $(b)$ and $(c)$ and $\alpha$ satisfies properties $(b)$ and $(c)$. We
are left to verify that $\alpha $ satisfies property $(c)$. That is, we want
to show that
\begin{equation}  \label{gg}
1+\sum_{j=1}^{k-2}i_{j}\leq \lfloor \frac{k-1}{2}\rfloor .
\end{equation}%
If $k$ is odd, say $k=2m+1$ for some $m\geq 2$, (\ref{gg}) is true because $%
(i_{0},i_{1},\dots ,i_{k-2})\in J_{k-1}$, which implies by induction
hypothesis that $\sum_{j=1}^{k-2}i_{j}\leq \lfloor \frac{k-2}{2}\rfloor =m-1$%
. If $k$ is even, say $k=2m+2$, (\ref{gg}) is true because the following
claim asserts that if $i_0 \ge 1$, then $\sum_{j=1}^{k-2}i_{j}<\lfloor \frac{%
k-2}{2}\rfloor =m$.

\textbf{Claim:} For $(i_{0},i_{1},\dots ,i_{2m})\in J_{2m+1}$ with $m\geq 1$%
, if $\sum_{j=1}^{2m}i_{j}=m$ then $i_{0}=0$.

Indeed, suppose $(i_{0},i_{1},\dots ,i_{2m})\in J_{2m+1}$, $%
\sum_{j=1}^{2m}i_{j}=m$ and $i_{0}\geq 1$. We are going to show that leads
to a contradiction. First notice that $i_{1}\geq 1$, otherwise $i_{1}=0$ and
$\sum_{j=2}^{2m}i_{j}=m$, which is not possible because
\begin{equation*}
i_{0}+2m\leq i_{0}+\sum_{j=1}^{2m}ji_{j}\leq 2m.
\end{equation*}%
Also, we must have $i_{2m}=0$, because otherwise property $(a)$ implies $%
i_{2m}=1$ and $i_{0}=i_{1}=\dots =i_{2m-1}=0$. Now we trace back to its
parent multi-indices in $J_{2m}$ by reversing the three operations. Of the
three operations, we can exclude $\partial _{\lambda }H_{2m}(D_{u}\delta
_{u},\delta _{u})D_{u}^{2}\delta _{u}$ and $\delta _{u}T_{2m}$, because $%
\partial _{\lambda }H_{2m}(D_{u}\delta _{u},\delta _{u})D_{u}^{2}\delta _{u}$
generates $\left( 2m-2j,j-1,1,0,\dots ,0\right) $ with $1\leq j\leq m$,
where $j$ must be $m$; and $\delta _{u}T_{2m}$ traces it back to $%
(i_{0}-1,i_{1},\dots ,i_{2m-1})\in J_{2m}$, where $i_{1}+\dots +i_{2m-1}=m>$
$\lfloor \frac{2m-1}{2}\rfloor $. Therefore, its parent multi-index in $%
J_{2m}$ must come from the operation $D_{u}T_{2m}$ and hence must be $%
(i_{0}+1,i_{1}-1,\dots ,i_{2m-1})\in J_{2m}$. Note that for this
multi-index,\ $i_{1}-1+\dots +i_{2m-1}=m-1$. Repeating the above process we
will end up at $(i_{0}+i_{1},0,i_{2}\dots ,i_{2m-i_{1}})\in J_{2m+1-i_{1}}$
with $i_{2}+\dots +i_{2m-i_{1}}=m-i_{1}$, which contradicts the property $%
(b) $ of $J_{2m+1-i_{1}}$ because
\begin{equation*}
i_{0}+2m-i_{1}\leq i_{0}+i_{1}+\sum_{j=2}^{2m-i_{1}}ji_{j}\leq 2m-i_{1}.
\end{equation*}
\end{proof}

Recall that we denote $D_{DF}w^{-1}=\left\langle Dw^{-1},DF\right\rangle _{%
\mathfrak{H}}$ and $D_{DF}^{k}w^{-1}=\left\langle
D(D_{DF}^{k-1}w^{-1}),DF\right\rangle _{\mathfrak{H}}$ for any $k\geq 2$.
The following lemma estimates the $L^{p}(\Omega )$ norms of $%
D_{DF}^{k}w^{-1} $.

\begin{lemma}
\label{DDFw} Let $F=I_{q}(f)$ with $q\geq 2$ satisfying $E[F^{2}]=\sigma
^{2} $. For any $\beta\ge 1$ we define and $M_\beta=\left( E\left\Vert
DF\right\Vert _{\mathfrak{H}}^{-\beta}\right) ^{1/\beta} $. Set $%
w=\left\Vert DF\right\Vert _{\mathfrak{H}}^{2}$.

\begin{itemize}
\item[(i)] If $M_\beta<\infty$ for some $\beta \ge 6$, then for any $1\le r
\le \frac {2\beta}{\beta+6}$
\begin{equation}
\left\| D_{DF} w^{-1} \right\| _{r } \le CM^3_\beta \left \| q\sigma^2-w
\right\|_2.  \label{DDFw-a}
\end{equation}

\item[(ii)] If $k\ge 2$ and $M_\beta<\infty$ for some $\beta \ge 2k+4$, then
for any $1< r < \frac {2\beta}{\beta+2k+4}$
\begin{equation}
\left\| D^k_{DF} w^{-1} \right\| _{r } \le C\left( \sigma^{2k-2} \vee 1
\right) \left(M^{k+2}_\beta \vee 1\right) \left \| q\sigma^2-w \right\|_2.
\label{DDFw-b}
\end{equation}

\item[(iii)] If $k\ge 1$and $M_\beta<\infty$ for any $\beta>k+2$, then for
any $1<r<\frac \beta{k+2}$
\begin{equation}
\left\| D^k_{DF} w^{-1} \right\| _{r } \le C\left( \sigma^{2k } \vee 1
\right) \left(M^{k+2}_\beta \vee 1\right).  \label{DDFw-c}
\end{equation}
\end{itemize}
\end{lemma}

\begin{proof}
Note that $D_{DF}w^{-1}=\left\langle Dw^{-1},DF\right\rangle _{\mathfrak{H}%
}=-2w^{-2}\left\langle D^{2}F\otimes _{1}DF,DF\right\rangle $. Then
\begin{equation*}
\left\vert D_{DF}w^{-1}\right\vert \leq 2w^{-\frac{3}{2}}\left\Vert
D^{2}F\otimes _{1}DF\right\Vert _{\mathfrak{H}}.
\end{equation*}%
Applying H\"{o}lder's inequality with $\frac{1}{r}=\frac{1}{p}+\frac{1}{2}$,
yields
\begin{equation*}
\left\Vert D_{DF}w^{-1}\right\Vert _{r}\leq 2\left( E(w^{-\frac{3p}{2}%
})\right) ^{\frac{1}{p}}\left\Vert D^{2}F\otimes _{1}DF\right\Vert _{2},
\end{equation*}%
which implies (\ref{DDFw-a}) by choosing $p\leq \beta /3$ and taking into
account (\ref{contrtnBd}). Notice that we need $1\geq \frac{1}{r}\geq \frac{3%
}{\beta }+\frac{1}{2}=\frac{\beta +6}{2\beta }$.

Consider now the case $k\geq 2$. From the pattern indicated by the first
three terms,
\begin{eqnarray*}
D_{DF}w^{-1} &=&\left\langle Dw^{-1},DF\right\rangle _{\mathfrak{H}}\text{, }
\\
D_{DF}^{2}w^{-1} &=&\left\langle D^{2}w^{-1},(DF)^{\otimes 2}\right\rangle _{%
\mathfrak{H}^{\otimes 2}}+\left\langle Dw^{-1}\otimes DF,D^{2}F\right\rangle
_{\mathfrak{H}^{\otimes 2}}, \\
D_{DF}^{3}w^{-1} &=&\left\langle D^{3}w^{-1},(DF)^{\otimes 3}\right\rangle _{%
\mathfrak{H}^{\otimes 3}}+3\left\langle D^{2}w^{-1}\otimes DF,D^{2}F\otimes
DF\right\rangle _{\mathfrak{H}^{\otimes 3}} \\
&&+\left\langle Dw^{-1}\otimes D^{2}F,D^{2}F\otimes DF\right\rangle _{%
\mathfrak{H}^{\otimes 3}}+\left\langle Dw^{-1}\otimes (DF)^{\otimes
2},D^{3}F\right\rangle _{\mathfrak{H}^{\otimes 3}},
\end{eqnarray*}%
we can prove by induction that
\begin{equation*}
\left\vert D_{DF}^{k}w^{-1}\right\vert \leq C\sum_{i=1}^{k}\left\Vert
D^{i}w^{-1}\right\Vert _{\mathfrak{H}^{\otimes i}}\left\Vert DF\right\Vert _{%
\mathfrak{H}}^{i}\left(
\sum_{\sum_{j=1}^{k}i_{j}=k-i}\prod_{j=1}^{k}\left\Vert D^{j}F\right\Vert _{%
\mathfrak{H}^{\otimes j}}^{i_{j}}\right) .
\end{equation*}%
By (\ref{HyperMeyer}), for any $p>1$, $\left\Vert D^{j}F\right\Vert _{p}\leq
C\left\Vert F\right\Vert _{2}=C\sigma $. Applying\ H\"{o}lder's inequality
and assuming that $s>r$, we have,
\begin{equation}
\left\Vert D_{DF}^{k}w^{-1}\right\Vert _{r}\leq C\sum_{i=1}^{k}\left\Vert
\left\Vert D^{i}w^{-1}\right\Vert _{\mathfrak{H}^{\otimes i}}\left\Vert
DF\right\Vert _{\mathfrak{H}}^{i}\right\Vert _{s}\sigma ^{k-i}.
\label{DDFw1}
\end{equation}%
We are going to see that $\left\Vert DF\right\Vert _{\mathfrak{H}}^{i}$ will
contribute to compensate the singularity of $\left\Vert
D^{i}w^{-1}\right\Vert _{\mathfrak{H}^{\otimes i}}$. First by induction one
can prove that for $1\leq i\leq m$, $D^{i}w^{-1}$ has the following
expression%
\begin{equation}
D^{i}w^{-1}=\sum_{l=1}^{i}(-1)^{l}\sum_{(\alpha ,\beta )\in
I_{i,l}}w^{-(l+1)}\bigotimes_{j=1}^{l}\left( D^{\alpha _{j}}F\otimes
_{1}D^{\beta _{j}}F\right) ,  \label{Dkw}
\end{equation}%
where $I_{i,l}=\{(\alpha ,\beta )\in \mathbb{N}^{2l}:\alpha _{j}+\beta
_{j}\geq 3,\sum_{j=1}^{l}(\alpha _{j}+\beta _{j})=i+2l\}$. In fact, for $i=1$%
,
\begin{equation*}
Dw^{-1}=-2w^{-2}D^{2}F\otimes _{1}DF,
\end{equation*}%
which is of the above form because $I_{1,1}=\{(1,2),(2,1)\}$. Suppose that (%
\ref{Dkw}) holds for some $i\leq m-1$. Then,
\begin{eqnarray*}
D^{i+1}w^{-1} &=&\sum_{l=1}^{i}(-1)^{l+1}2(l+1)\sum_{(\alpha ,\beta )\in
I_{i,l}}w^{-(l+2)}(D^{2}F\otimes _{1}DF)\bigotimes_{j=1}^{l}\left( D^{\alpha
_{j}}F\otimes _{1}D^{\beta _{j}}F\right) \\
&&+\sum_{l=1}^{i}(-1)^{l}\sum_{(\alpha ,\beta )\in
I_{i,l}}w^{-(l+1)}\sum_{h=1}^{l}(D^{\alpha _{j}+1}F\otimes _{1}D^{\beta
_{j}}F+D^{\alpha _{j}}F\otimes _{1}D^{\beta _{j}+1}F) \\
&&\times \bigotimes_{j=1,j\not=h}^{l}\left( D^{\alpha _{j}}F\otimes
_{1}D^{\beta _{j}}F\right) ,
\end{eqnarray*}%
which is equal to
\begin{equation*}
\sum_{l=1}^{i+1}(-1)^{l}\sum_{(\alpha ,\beta )\in
I_{i+1,l}}w^{-(l+1)}\bigotimes_{j=1}^{l}\left( D^{\alpha _{j}}F\otimes
_{1}D^{\beta _{j}}F\right) .
\end{equation*}
From (\ref{Dkw}) for any $i=1,\dots ,k$ we can write
\begin{equation}
\left\Vert D^{i}w^{-1}\right\Vert _{\mathfrak{H}^{\otimes i}}\left\Vert
DF\right\Vert _{\mathfrak{H}}^{i}\leq \sum_{l=1}^{i}w^{-(l+1)+\frac{i}{2}%
}\sum_{(\alpha ,\beta )\in I_{i,l}}\prod_{j=1}^{l}\left\Vert D^{\alpha
_{j}}F\otimes _{1}D^{\beta _{j}}F\right\Vert _{\mathfrak{H}^{\otimes \alpha
_{j}+\beta _{j}-2}},  \label{eq1}
\end{equation}%
where $I_{i,l}=\{(\alpha ,\beta )\in \mathbb{N}^{l}\times \mathbb{N}%
^{l}:\alpha _{j}+\beta _{j}\geq 3,\sum_{j=1}^{l}(\alpha _{j}+\beta
_{j})=i+2l\}$. Note that by (\ref{HyperMeyer}),
\begin{equation*}
\left\Vert D^{\alpha _{j}}F\otimes _{1}D^{\beta _{j}}F\right\Vert _{p}\leq
C\left\Vert F\right\Vert _{2}^{2}=C\sigma ^{2}
\end{equation*}%
for all $p\geq 1$ and all $\alpha _{j},\beta _{j}$. This inequality will be
applied to all but one of the contraction terms in the product $%
\prod_{j=1}^{l}\left\Vert D^{\alpha _{j}}F\otimes _{1}D^{\beta
_{j}}F\right\Vert _{\mathfrak{H}^{\otimes \alpha _{j}+\beta _{j}-2}}$. We
decompose the sum in (\ref{eq1}) into two parts. If the index $l$ satisfies $%
l\leq \frac{i}{2}-1$, then the exponent of $w$ is nonnegative, and the $p$%
-norm of $w$ can be estimated by a constant times $\sigma ^{2}$, while for $%
\frac{i}{2}-1<l$ this exponent is negative. Then, using H\"{o}lder's
inequality and assuming that $\frac{1}{s}=\frac{1}{p}+\frac{1}{2}$, we
obtain
\begin{eqnarray}
&&\left\Vert \left\Vert D^{i}w^{-1}\right\Vert _{\mathfrak{H}^{\otimes
i}}\left\Vert DF\right\Vert _{\mathfrak{H}}^{i}\right\Vert _{s}  \notag \\
&\leq &C\left( \mathbf{1}_{\{i\geq 2\}}\sigma ^{i-2}+\sum_{\frac{i}{2}%
-1<l\leq i}\left\Vert w^{-(l+1)+\frac{i}{2}}\right\Vert _{p}\sigma
^{2(l-1)}\right) \left\Vert D^{\alpha _{1}}F\otimes _{1}D^{\beta
_{1}}F\right\Vert _{2}.  \label{DDFw2}
\end{eqnarray}%
Note that for $l\leq i\leq k$, $l+1-\frac{i}{2}\leq \frac{k}{2}+1$.
Therefore, for $\frac{i}{2}-1<l\leq i$
\begin{equation*}
\left\Vert w^{-(l+1)+\frac{i}{2}}\right\Vert _{p}=M_{2(l+1-\frac{i}{2}%
)p}^{2l+2-i}\leq M_{(k+2)p}^{2l+2-i}\leq M_{(k+2)p}^{k+2}\vee 1.
\end{equation*}%
Therefore, using (\ref{contrtnBd}) we obtain
\begin{equation}
\left\Vert \left\Vert D^{i}w^{-1}\right\Vert _{\mathfrak{H}^{\otimes
i}}\left\Vert DF\right\Vert _{\mathfrak{H}}^{i}\right\Vert _{s}\leq C\left(
(\sigma ^{2i-2}\vee 1)(M_{(k+2)p}^{k+2}\vee 1)\right) \left\Vert q\sigma
^{2}-w\right\Vert _{2}.  \label{eq2}
\end{equation}%
Combining (\ref{eq2}) and (\ref{DDFw1}) and choosing $p$ such that $%
(k+2)p\leq \beta $ we get (\ref{DDFw-b}). Note that we need
\begin{equation*}
1>\frac{1}{r}>\frac{k+2}{\beta }+\frac{1}{2}=\frac{\beta +2k+4}{2\beta },
\end{equation*}%
which holds if $1<r<\frac{2\beta }{\beta +2k+4}$. The proof of part (iii) is
similar and omitted.
\end{proof}

The next lemma gives estimates on $D_{u}^{k}\delta _{u}$ for $k\ge 0$.

\begin{lemma}
\label{Dudelta-u} Let $F=I_{q}(f)$ with $q\geq 2$ satisfying $%
E[F^{2}]=\sigma ^{2}$. For any $\beta \geq 1$ we define $M_{\beta }=\left(
E\left\Vert DF\right\Vert _{\mathfrak{H}}^{-\beta }\right) ^{1/\beta }$ and
denote $w=\left\Vert DF\right\Vert _{\mathfrak{H}}^{2}$.

\begin{itemize}
\item[(i)] If $M_\beta<\infty$ for some $\beta>3$, then for any $1<s<\frac
\beta 3$,
\begin{equation}
\| \delta _{u} \|_{s } \le C( \sigma ^2\vee1) (M^{3} _\beta \vee1).
\label{Dudelta-u0}
\end{equation}

\item[(ii)] If $k\ge 1$ and $M_\beta<\infty$ for some $\beta>3k+3$, then for
any $1<s<\frac \beta {3k+3}$,
\begin{equation}
\|D^k_{u} \delta _{u} \|_{s } \le C_\sigma (M^{3k+3} _\beta \vee1).
\label{Dudelta-uka}
\end{equation}

\item[(iii)] If $k\geq 2$ and $M_{\beta }<\infty $ for some $\beta >6k+6$,
then for any $1<s<\frac{2\beta }{\beta +6k+6}$,
\begin{equation}
\Vert D_{u}^{k}\delta _{u}\Vert _{s}\leq C_{\sigma }(M_{\beta }^{3k+3}\vee
1)\Vert q\sigma ^{2}-w\Vert _{2}.  \label{Dudelta-uk}
\end{equation}
\end{itemize}
\end{lemma}

\begin{proof}
Recall that $\delta _{u}=qFw^{-1}-D_{DF}w^{-1}$. Then for any $r>s$,
\begin{equation*}
\Vert \delta _{u}\Vert _{s}\leq C\left( \sigma \Vert w^{-1}\Vert _{r}+\Vert
D_{DF}w^{-1}\Vert _{s}\right) .
\end{equation*}%
Then, $\Vert w^{-1}\Vert _{r}=M_{2r}^{2}$ and the result follows by applying
Lemma \ref{DDFw} (iii) with $k=1$ and by choosing $r<\frac{\beta }{3}$.

To show (ii) and (iii) we need to find a useful expression for $%
D_{u}^{k}\delta _{u}$. Consider the operator $D_{u}=w^{-1}D_{DF}$. We claim
that for any $k\geq 1$ the iterated operator $D_{u}^{k}$ can be expressed as
\begin{equation}
D_{u}^{k}=\sum_{l=1}^{k}w^{-l}\sum_{\mathbf{i}\in I_{l,k}}b_{\mathbf{i}}%
\left[ \prod_{j=1}^{k-l}D_{DF}^{i_{j}}w^{-1}\right] D_{DF}^{i_{0}},
\label{Du2}
\end{equation}%
where $b_{\mathbf{i}}>0$ are real numbers and
\begin{equation*}
I_{l,k}=\{\mathbf{i}=(i_{0},i_{1},\dots ,i_{l}):i_{0}\geq 1,i_{j}\geq
0\,\,\forall \,\,j=1,\dots ,l,\sum_{j=0}^{k-l}i_{j}=k\}.
\end{equation*}%
In fact, this is clearly true for $k=1$. Assume (\ref{Du2}) holds for a
given $k$. Then
\begin{eqnarray*}
D_{u}^{k+1} &=&w^{-1}D_{DF}D^{k}u \\
&=&\sum_{l=1}^{k}lw^{-l}D_{DF}w^{-1}\sum_{\mathbf{i}\in I_{l,k}}b_{\mathbf{i}%
}\left[ \prod_{j=1}^{k-l}D_{DF}^{i_{j}}w^{-1}\right] D_{DF}^{i_{0}} \\
&&+\sum_{l=1}^{k}w^{-l-1}\sum_{\mathbf{i}\in I_{l,k}}b_{\mathbf{i}}\left[
\sum_{h=1}^{k-l}D_{DF}^{i_{h}+1}w^{-1}\prod_{j=1,j%
\not=h}^{k-l}D_{DF}^{i_{j}}w^{-1}\right] D_{DF}^{i_{0}} \\
&&+\sum_{l=1}^{k}w^{-l-1}\sum_{\mathbf{i}\in I_{l,k}}b_{\mathbf{i}}\left[
\prod_{j=1}^{k-l}D_{DF}^{i_{j}}w^{-1}\right] D_{DF}^{i_{0}+1}.
\end{eqnarray*}%
Shifting the indexes, this can be written as
\begin{eqnarray*}
D_{u}^{k+1} &=&\sum_{l=1}^{k}lw^{-l}D_{DF}w^{-1}\sum_{\mathbf{i}\in
I_{l,k}}b_{\mathbf{i}}\left[ \prod_{j=1}^{k-l}D_{DF}^{i_{j}}w^{-1}\right]
D_{DF}^{i_{0}} \\
&&+\sum_{l=2}^{k+1}w^{-l}\sum_{\mathbf{i}\in I_{l-1,k}}b_{\mathbf{i}}\left[
\sum_{h=1}^{k+1-l}D_{DF}^{i_{h}+1}w^{-1}\prod_{j=1,j%
\not=h}^{k+1-l}D_{DF}^{i_{j}}w^{-1}\right] D_{DF}^{i_{0}} \\
&&+\sum_{l=2}^{k+1}w^{-l}\sum_{\mathbf{i}\in I_{l-1,k}}b_{\mathbf{i}}\left[
\prod_{j=1}^{k+1-l}D_{DF}^{i_{j}}w^{-1}\right] D_{DF}^{i_{0}+1}.
\end{eqnarray*}%
It easy to check that this coincides with
\begin{equation*}
\sum_{l=1}^{k+1}w^{-l}\sum_{\mathbf{i}\in I_{l,k+1}}b_{\mathbf{i}}\left[
\prod_{j=1}^{k+1-l}D_{DF}^{i_{j}}w^{-1}\right] D_{DF}^{i_{0}}.
\end{equation*}%
Also, note that $\delta _{u}=qFw^{-1}+D_{DF}w^{-1}$ and
\begin{equation*}
D_{DF}\delta _{u}=q+qFD_{DF}w^{-1}+D_{DF}^{2}w^{-1}.
\end{equation*}%
By induction we can show that for any $i_{0}\geq 1$
\begin{equation}
D_{DF}^{i_{0}}\delta _{u}=q\delta
_{1i_{0}}+q%
\sum_{j=1}^{i_{0}-1}c_{i,j}D_{DF}^{i_{0}-1-j}wD_{DF}^{j}w^{-1}+qFD_{DF}^{i_{0}}w^{-1}+D_{DF}^{i_{0}+1}w^{-1},
\label{DDFdelta}
\end{equation}%
where $\delta _{1i_{0}}$ is the Kronecker symbol. Combining (\ref{Du2}) and (%
\ref{DDFdelta}) we obtain%
\begin{eqnarray*}
D_{u}^{k}\delta _{u} &=&\sum_{l=1}^{k}w^{-l}\sum_{\mathbf{i}\in I_{l,k}}b_{%
\mathbf{i}}\left[ \prod_{j=1}^{k-l}D_{DF}^{i_{j}}w^{-1}\right] \times \Bigg[%
q\delta _{1i_{0}} \\
&&+q%
\sum_{j=1}^{i_{0}-1}c_{i,_{0}j}D_{DF}^{i_{0}-1-j}wD_{DF}^{j}w^{-1}+qFD_{DF}^{i_{0}}w^{-1}+D_{DF}^{i_{0}+1}w^{-1}%
\Bigg ].
\end{eqnarray*}%
Next we shall apply H\"{o}lder's inequality to estimate $\left\Vert
D_{u}^{k}\delta _{u}\right\Vert _{s}$. Notice that for $l=k$, $i_{0}=k\geq 2$%
. Therefore,
\begin{eqnarray*}
\left\Vert D_{u}^{k}\delta _{u}\right\Vert _{s} &\leq &C_{\sigma
}\sum_{l=1}^{k-1}\sum_{\mathbf{i}\in I_{l,k}}\Vert w^{-l}\Vert
_{p}\prod_{j=1}^{k-l}\Vert D_{DF}^{i_{j}}w^{-1}\Vert _{r_{j}}\left( \delta
_{1i_{0}}+\max_{1\leq h\leq i_{0}+1}\Vert D_{DF}^{h}w^{-1}\Vert
_{r_{0}}\right) \\
&&+C_{\sigma }\Vert w^{-k}\Vert _{p}\max_{1\leq h\leq k+1}\Vert
D_{DF}^{h}w^{-1}\Vert _{\rho _{0}}=B_{1}+B_{2},
\end{eqnarray*}%
assuming that for $l=1,\dots ,k-1$, $\frac{1}{s}>\frac{1}{p}+\sum_{j=0}^{k-l}%
\frac{1}{r_{j}}$ and $\frac{1}{s}>\frac{1}{p}+\frac{1}{\rho _{0}}$, and
where $C_{\sigma }$ denotes a function of $\sigma $ of the form $C(1+\sigma
^{M})$.

Let us consider first the term $B_{1}$. Note that if $i_{0}=1$ there is at
least one factor of the form $\Vert D_{DF}^{r_{j}}w^{-1}\Vert _{r_{j}}$ in
the above product, because $\sum_{j=1}^{k-l}i_{j}=k-1\geq 1$. Then, we will
apply the inequality (\ref{DDFw-b}) to one of these factors and the
inequality (\ref{DDFw-c}) to the remaining ones. The estimate (\ref{DDFw-c})
requires $\frac{1}{r_{j}}>\frac{i_{j}+2}{\beta }$ for $j=1,\dots ,k-l$ and $%
\frac{1}{r_{0}}>\frac{i_{0}+3}{\beta }$. On the other hand, the estimate (%
\ref{DDFw-b}) requires $\frac{1}{r_{j}}>\frac{i_{j}+2}{\beta }+\frac{1}{2}$
for $j=1,\dots ,k-l$ and $\frac{1}{r_{0}}>\frac{i_{0}+3}{\beta }+\frac{1}{2}$%
. Then, choosing $p$ such that $2pl<\beta $, and taking into account that $%
\sum_{j=0}^{k-l}i_{j}=k$ we obtain the inequalities
\begin{equation*}
\frac{1}{s}>\frac{1}{p}+\sum_{j=1}^{k-l}\frac{i_{j}+2}{\beta }+\frac{i_{0}+3%
}{\beta }+\frac{1}{2}>\frac{3k+3}{\beta }+\frac{1}{2}.
\end{equation*}%
Hence, if $s<\frac{2\beta }{\beta +6k+6}$ we can write
\begin{equation*}
B_{1}\leq C_{\sigma }\sum_{l=1}^{k-1}M_{\beta
}^{2l}\prod_{j=1}^{k-l}(M_{\beta }^{i_{j}+2}\vee 1)(M_{\beta }^{i_{0}+3}\vee
1)\Vert q\sigma ^{2}-w^{-1}\Vert _{2}\leq C_{\sigma }(M_{\beta }^{3k+3}\vee
1)\Vert q\sigma ^{2}-w^{-1}\Vert _{2}.
\end{equation*}%
For the term $B_{2}$ we use the estimate (\ref{DDFw-b}) assuming $2pk<\beta $
and
\begin{equation*}
\frac{1}{s}>\frac{1}{p}+\frac{k+3}{\beta }+\frac{1}{2}>\frac{3k+3}{\beta }+%
\frac{1}{2}.
\end{equation*}%
This leads to the same estimate and the proof of (\ref{Dudelta-uk}) is
complete. To show the estimate (\ref{Dudelta-uka}) we proceed as before but
using the inequality (\ref{DDFw-c}) for all the factors. In this case the
summand $\frac{1}{2}$ does not appear and we obtain (\ref{Dudelta-uka}).
\end{proof}

\bigskip
$\mathbf{Acknowledgement}$ We would like to thank the anonymous referee's
constructive comments and bringing the topic about entropic CLT to our
attention. 







\end{document}